\documentclass[12pt]{amsart}
\usepackage {amscd}
\usepackage[german,english]{babel}
\usepackage {amsfonts}
\usepackage[all]{xy}
\usepackage{amssymb,amsmath, amsthm,latexsym,verbatim}
\usepackage{a4wide}
\usepackage{enumitem}
\usepackage{hyperref}
\newcommand{\C}{\mathbb{C}}
\newcommand{\R}{\mathbb{R}}
\newcommand{\Z}{\mathbb{Z}}

\newcommand{\rk}{\operatorname{rk}}

\newcommand{\kL}{\mathfrak{k}}
\newcommand{\mL}{\mathfrak{m}}
\newcommand{\aL}{\mathfrak{a}}
\newcommand{\gL}{\mathfrak{g}}
\newcommand{\pL}{\mathfrak{p}}
\newcommand{\nL}{\mathfrak{n}}
\newcommand{\hL}{\mathfrak{h}}

\newcommand{\Rep}{\operatorname{Rep}}
\newcommand{\supp}{\operatorname{supp}}
\newcommand{\dist}{\operatorname{dist}}

\newcommand{\Gl}{\operatorname{GL}}

\newcommand{\Aut}{\operatorname{Aut}}

\newcommand{\Id}{\operatorname{Id}}

\newcommand{\Spin}{\operatorname{Spin}}
\newcommand{\SO}{\operatorname{SO}}
\newcommand{\reg}{\operatorname{reg}}
\newcommand{\Hom}{\operatorname{Hom}}
\newcommand{\Rel}{\operatorname{rel}}

\newcommand{\dom}{\operatorname{dom}}

\newcommand{\Ad}{\operatorname{Ad}}

\newcommand{\Sl}{\operatorname{SL}}
\newcommand{\Ker}{\operatorname{Ker}}
\newcommand{\ext}{\operatorname{ext}}

\newcommand{\vol}{\operatorname{vol}}
\newcommand{\Real}{\operatorname{Re}}
\newcommand{\Tr}{\operatorname{Tr}}

\newcommand{\End}{\operatorname{End}}

\newtheorem {thrm}{Theorem}[section]
\newtheorem {prop}[thrm] {Proposition}
\newtheorem {lem}[thrm] {Lemma}
\newtheorem {kor}[thrm]{Corollary}
\theoremstyle{definition}
\newtheorem{defn}[thrm] {Definition}
\theoremstyle{remark}
\newtheorem {bmrk}[thrm] {Remark}

\begin{document}
\title[]{A gluing formula for the analytic torsion on hyperbolic manifolds with
cusps}
\author{Jonathan Pfaff}
\address{Universit\"at Bonn\\
Mathematisches Institut\\
Endenicher Alle 60\\
D -- 53115 Bonn, Germany}
\email{pfaff@math.uni-bonn.de}

\begin{abstract}
For an odd-dimensional oriented hyperbolic manifold with cusps and strongly 
acyclic coefficient systems we 
define the Reidemeister torsion 
of the Borel-Serre compactification of the manifold using 
bases of cohomology classes defined via Eisenstein series by the method of
Harder. 
In the main result of this paper we relate this combinatorial 
torsion to the regularized analytic torsion. 
Together with results on the 
asymptotic behaviour of the regularized analytic 
torsion, established previously, this should have applications 
to study the growth of torsion in the cohomology 
of arithmetic groups. 
Our main result is established via a gluing formula and 
here our approach is heavily inspired by a recent paper of Lesch. 
\end{abstract}

\maketitle

\tableofcontents

\section{Introduction} 
The aim of the present paper is to study 
the relation between analytic and combinatorial 
torsion on odd-dimensional hyperbolic manifolds with 
cusps with coefficients in strongly acyclic coefficient systems by establishing 
a gluing formula. Our approach to the gluing formula is heavily inspired by a
recent paper of Lesch
\cite{Le}, whose work was built on earlier work of Vishik \cite{Vi}.
 
Let us describe the situation we consider in more detail. 
Let $X$ be an odd-dimensional hyperbolic manifold  
of the form $X=\Gamma\backslash\mathbb{H}^d$, $\Gamma$ 
a discrete, torsion-free subgroup of $\SO^0(d,1)$ or $\Spin(d,1)$ such that $X$
is of finite volume but not compact. Then $X$ is a manifold with cusps. 
More precisely, there exist finitely many closed manifolds $T_i$
with Riemannian metrics $g_i$ 
such that for $ Y\geq Y_0$, $Y_0$ sufficiently large, $X$ can be
decomposed into a compact smooth manifold with boundary $X(Y)$, whose 
boundary $\partial X(Y)$ is the 
disjoint union of the $T_i$, and 
the disjoint union of finitely many cusp pieces $F_{i,X}(Y):=[Y,\infty)\times
T_i$, glued 
together with $X(Y)$ along the $T_i$. Moreover, the hyperbolic metric $g$ of $X$
restricted 
to each cusp $F_{i,X}(Y)$ equals the warped product metric $y^{-2}(dy^2+g_i)$.
Over $X$ we consider 
a flat vector bundle $E_\rho$ induced by the 
restriction to $\Gamma$ of a finite-dimensional irreducible representation
$\rho$ of 
the group $\SO^0(d,1)$ or $\Spin(d,1)$. We assume that $\rho$ is not invariant
under
the standard Cartan
involution $\theta$ of $G$. Such representations exist in abundance and
following
Bergeron and
Venkatesh \cite{BV}, we call the 
flat bundle $E_\rho$ strongly acyclic.
The bundle $E_\rho$ is unimodular and moreover possesses 
a canonical metric. 

Since $X$ is complete, the underlying flat $E_\rho$-valued 
Hodge-Laplacians on $X$ are essentially selfadjoint. However, these 
operators have a large continuous spectrum and therefore, the corresponding 
heat operators are not trace class. To overcome this problem, we work 
with the so called regularized trace \cite{Park}, \cite{MP1}, \cite{MP2} whose 
definition is inspired by ideas of Melrose \cite{Me}. More precisely,           
the integral of the heat kernels over the sets $X(Y)$ has 
a well defined asymptotic expansion in $Y$ and one defines the regularized trace
as the constant term in this expansion. In this 
way, one can define the regularized analytic torsion $T_{\reg}(X;E_\rho)$. 
For more details we refer to section \ref{secrelreg}.
\newline

For compact smooth manifolds the analytic torsion, introduced by Ray and Singer
\cite{RS}, 
equals the corresponding 
Reidemeister torsion.  For the case of closed manifolds and unitary
representations of the fundamental group, this was
proved independently by Cheeger
\cite{Cheeger 1} and M\"uller \cite{M1}. The extension to unimodular
representations is due to M\"uller \cite{Mutorsunim} and has also been
obtained 
independently and in a more general setting by Bismut and Zhang
\cite{Bismut}. Several different proofs and also several important
generalizations
of this theorem to the case of manifolds with boundary and the equivariant case
had been obtained in 
the sequel. Without aiming at completeness, let us just mention the work of 
Hassell \cite{Has}, L\"uck \cite{Luck}, Vishik \cite{Vi},
Br\"uning and Ma \cite{BM1} \cite{BM2} and Lott and Rothenberg \cite{LR} here.

In the present case, the manifold $X$ does have a natural compactification
$\overline{X}$. More precisely, 
$\overline{X}$  
is by definition the diffeomorphism class of the manifold $X(Y)$ as above.
Thus, 
$\overline{X}$ is a compact smooth manifold with boundary which is homotopy
equivalent to $X$.
However, the manifold $\overline{X}$ does not have any canonical Riemannian
structure. More precisely, the manifolds $X(Y_1)$ and $X(Y_2)$ respectively the
corresponding 
flat bundles are not isometric for
different 
values $Y_1$ and $Y_2$. Consequently, in particular since the bundle $E_\rho$
restricted to these 
manifolds has non-vanishing De Rham cohomology in our situations, the analytic 
torsions $T(X(Y_1);E_\rho)$ and $T(X(Y_2);E_\rho)$, taken 
with respect to absolute boundary conditions, differ. For this reason  we
study the regularized analytic torsion, which is a spectral 
invariant of the complete hyperbolic manifold $X$. Nevertheless, 
the Cheeger-M\"uller Theorem for manifolds with boundary, due to L\"uck
\cite{Luck}, Vishik \cite{Vi} and 
Br\"uning and Ma \cite{BM1}, \cite{BM2}, will play an important role in our
article.

On the combinatorial side, the situation is similar. If one 
wants to define the Reidemeister torsion of $\overline{X}$ as 
a positive real number and not just as a non-zero element in the determinant
line 
of the cohomology $H^*(\overline{X};E_\rho)$, one needs to fix 
a metric on this determinant line. However, if one 
takes the metrics on the cohomology groups  $H^*(X(Y_1);E_\rho)$ and 
$H^*(X(Y_2);E_\rho)$, $Y_2>Y_1$,
defined by the Hodge De Rham isomorphism with respect to 
absolute boundary conditions, then 
the canonical inclusion $X(Y_1)\hookrightarrow X(Y_2)$ induces an isomorphism,
but not 
an isometry on cohomology. Therefore, one does not obtain a well 
defined Reidemeister torsion which is independent of the truncation parameter. 

To overcome this problem, we will work with bases in the cohomology 
of $\overline{X}$ with coefficients in $E_\rho$ using the constructions 
of Harder \cite{Ha}. The condition
$\rho\neq\rho_\theta$ implies
that the $L^2$-cohomology
of $X$ with 
coefficients in $E_\rho$ vanishes. Therefore, in the present situation the
cohomology
$H^*(\overline{X};E_\rho)$ 
is determined  completely by its image in the cohomology $H^*(\partial
\overline{X};E_\rho)$ of the boundary.
The cohomology of the boundary is in turn understood: By a result 
of van Est \cite{van Est}, for each boundary component it equals the
corresponding twisted cohomology of a nilpotent 
Lie algebra $\nL$ associated to that component; the latter cohomology 
has been evaluated by Kostant \cite{Kostant}. The idea of Harder \cite{Ha} now 
was to lift cohomology classes from the boundary to cohomology classes of $X$
using 
the method of Eisenstein series 
and therefore to obtain a certain inverse of the restriction map in cohomology.
His 
approach was developed further by Schwermer \cite{Sc}. In the 
present context, this construction gives bases of all cohomology groups
$H^*(X;E_\rho)\cong H^*(\overline{X};E_\rho)$. 
Moreover, the resulting Reidemeister torsion does not depend on the  particular
choices
made in the construction of such bases.  
Therefore, one obtains a well defined Reidemeister torsion 
$\tau_{Eis}(\overline{X};E_\rho)\in (0,\infty)$. 
We will pursue these constructions in detail in sections \ref{SecDRcusp} to
\ref{secbases}. At this point we remark that in their recent book \cite{CV},
Calegari 
and Venkatesh also used Eisenstein series to give a canonical definition of
Reidemeister torsion 
in a related 3-dimensional context. 
\newline

To state our main theorem, we need to introduce some more notation. 
Firstly,
if ${\partial \overline{X}}_i$ is a boundary component of $\overline{X}$ and if
$A$ is the split
component of
the $\Gamma$-cuspidal parabolic subgroup 
$P$ of $G$, implicit in the choice of $\partial{\overline{X}}_i$, then each
cohomology group 
$H^k(\partial {\overline{X}}_i;E_\rho)$ 
carries a natural action of $A$. In 
our real rank one case, one has $A\cong\R$. By Kostant's theorem, the
corresponding weights $\lambda_{\rho,k}\in \aL^*\cong \R$  
of this action can be computed explicitly in terms of the highest weight of
$\rho$, see \eqref{lambdatau}, \eqref{lambdadecom}. 
Next, in the notation of the beginning of this section, we let 
$F_X$ be the union of the cusps $[1,\infty)\times T_i$ at height one,
equipped 
with the warped product metric $y^{-2}(dy^2+g_i)$.  
We point out that we do not require that $F_X$ is actually a subset 
of $X$, i.e. the cusps may be separated only at a height which is greater than
$1$. Nevertheless, $F_X$ is 
unambiguously defined. We let $T_{\reg}(F_X,\partial F_X;E_\rho)$ be the 
regularized analytic torsion of the cusp with relative boundary conditions
which 
is defined in the same way as the regularized analytic torsion of $X$. 
The last ingredient of our theorem is the Br\"uning-Ma anomaly formula for the
manifold 
$F_X$. Although $F_X$ is not compact, its boundary is compact and since 
in our present unimodular case the anomaly is completely local in a
neighbourhood of the boundary, it can 
be defined by integrating the same secondary characteristic form as in
\cite{BM1},
\cite{BM2} over $\partial F_X$, see section \ref{SecGlII}. 
By the explicit computations of Br\"uning and Ma \cite{BM1}, for 
the present case the resulting integral equals $c(n)\rk(E_\rho)\vol(\partial
F_X)$, 
where $c(n)\neq 0$ is a constant, defined in \eqref{KonstBM}, which depends only
on the
dimension $d=2n+1$
of $X$. 

We further make an assumption \eqref{asGamma} on $\Gamma$ which 
holds for example if $\Gamma$ is ``neat'' in the sense of Borel
\cite{BorelArGr}. This 
will be discsussed further in section \ref{subsmfld}.   
This assumption is not essential for our approach and can certainly 
be dropped. If
it is not satisfied, the cohomology of the boundary is a bit more complicated
and one has to twist with the additional actions of certain finite groups.
However,
since for example in the arithmetic case, intersecting with a fixed principal 
congruence subgroup of sufficiently high level assumption \eqref{asGamma} is
satisfied, i.e. 
\eqref{asGamma} holds by passing to finite coverings, we shall 
assume it for convenience from now on.

The main result of this article is now the following theorem. 

\begin{thrm}\label{Theorem}
For every $\rho\in\Rep(G)$, $\rho\neq\rho_\theta$ one has 
\begin{align*}
\log \tau_{Eis}(\overline{X};E_\rho)=&\log T_{\reg}(X;E_\rho)-\log
T_{\reg}(F_X,\partial F_X;E_\rho)+c(n)\rk
E_\rho\vol(\partial F_X)\\
&-\frac{1}{4}\sum_{k=0}^{d-1}(-1)^k\log (|\lambda_{\rho,k}|)\cdot \dim
H^k(\partial\overline{X};E_\rho).
\end{align*}
\end{thrm}

We remark that, as one might expect from the Cheeger-M\"uller Theorems 
for manifolds with boundary mentioned above and from Lesch's main result, the
Euler characteristic of $\overline{X}$  resp. $\partial \overline{X}$
occurs during our proof several times. However, it is always zero in our case
since $\partial \overline{X}$ is a finite disjoint 
union of tori and therefore it doesn't appear in our main result.  
Next we remark that if $d=2n+1$ and if 
$\kappa(X)$ is the number of cusps of $X$, then one has $\dim
H^k(\partial\overline{X};E_\rho) 
=\kappa(X)\dim \sigma_{\rho,k}$ for $k\neq n$ and 
$\dim H^n(\partial\overline{X};E_\rho) 
=2\kappa(X)\dim \sigma_{\rho,n}$, where 
$\sigma_{\rho,l}$ are representations of $M\cong\SO(d-1)$ or $M\cong
\Spin(d-1)$ 
whose highest weight can be computed explicitly from the 
highest weight of $\rho$ by Kostant's theorem, see \eqref{sigmatau},
\eqref{lambdadecom}. 
Therefore, the second line in our main theorem is explicit by Weyl's dimension
formula. 
\newline

We shall now describe our approach to prove our main result. We remark 
in the beginning that we do not use any computational properties  of the 
involved regularized traces. In particular, we do not use the 
Selberg trace formula. Instead, we study the problem directly 
via the gluing formula. 

We have the following gluing situation: For a fixed $Y_1>Y_0$ we consider the
decomposition 
$ X=X^+\cup X^-$, where $X^+:=X(Y_1)$ and $X^-:=F_X(Y_1)$. We firstly replace
the 
metrics on $X$ and $E_\rho$ by metrics which are of product structure 
in a neighbourhood of $\partial X^+$. Then we will study a variation 
of the $E_\rho$-valued De Rham complex, which is due to Lesch \cite{Le}, based
on earlier work of Vishik \cite{Vi}. This 
variation interpolates smoothly between the $E_\rho$-valued De Rham complex 
of $X$ and the direct sum of the De Rham complex of $X^+$ with absolute
boundary 
conditions and the De Rham complex of $X^-$ with relative boundary conditions. 
The goal is now to study the behaviour of the torsion under this variation.
Here, in 
order to avoid interchanges of limits, instead 
of using the regularized torsion, we will work with the relative torsion that 
will be defined in section \ref{secrelreg}. In the end, one can pass back 
to the regularized torsion since the difference between both torsions is
explicit. 

In our understanding, Lesch's approach \cite{Le} to study the behaviour of the
torsion 
under the variation can be
divided into two
parts. Firstly, one establishes a variation formula for the weighted sum of the
perturbed heat
traces for each fixed time \cite[Theorem 5.3]{Le}. This variation is local.
Therefore, in our case we will establish it via an approximation of our
variation through variations on closed manifolds,  
to which we can apply Lesch's formula. The result will be Proposition
\ref{PropGl1}. Here we recall that in our situation the Laplace operators 
have a large continuous spectrum, which is the main difference 
to Lesch's paper, who assumes that the spectrum is discrete. Therefore, in order
to generalize Lesch's variation formula in
Proposition \ref{PropGl1}, we need to pursue a detailed analysis of the
underyling 
heat kernels, which will be carried out mostly in sections 
\ref{secPropheatk}, \ref{secheatpertL} 
and \ref{secgl1}. 

However, by the presence of
large time values of the heat operators, the analytic torsion is not a local 
invariant. Due to this fact, a global term in the variation formula of Lesch
appears,
namely the torsion of the long exact cohomology sequence \cite[section 5.2,
section 5.3]{Le}. 
We recall that in our situation the Laplace operators 
of the underlying Hilbert complexes have a large continuous spectrum.
Nevertheless, due to
our assumption on the 
bundle $E_\rho$, these complexes are still Fredholm complexes 
in the sense of Br\"uning and Lesch \cite{BL}. Therefore, as proved by these
authors \cite{BL}, they still satisfy the so called strong Hodge decomposition
theorem. This property is sufficient to generalize 
Lesch's argument of \cite[section 5.2, section 5.3]{Le} and to establish the
full gluing formula,
Theorem \ref{KlebefVI}, 
for the analytic torsion with respect to the perturbed metrics. 

In the next step, we have to pass back from the perturbed metrics to the
original
metrics. Here, we will apply the 
anomaly formula of Br\"uning and Ma \cite{BM1}, \cite{BM2}. This is possible
since in our unimodular case the formula of Br\"uning and Ma is local at the
boundary of the underlying manifolds. Then, via a gluing formula 
for the cusps $F_X(Y_1)$, we get rid of the truncation 
parameter by passing from $F_X(Y_1)$ to $F_X$. 
Moreover, by the Cheeger-M\"uller theorem for manifolds with boundary and the
anomaly
formula, we can 
relate the analytic torsion of the compact part $X^+$ to 
the Reidemeister torsion of $X^+$. However, we need then to change 
the bases in the cohomology and to pass to the Eisenstein cohomology classes 
in order to make this torsion on $X^+$ equal to the torsion
$\tau_{Eis}(\overline{X};E_\rho)$. Also, 
we need to compute the corresponding torsion 
of the long exact cohomology sequence explicitly. For this purpose, in section
\ref{secclcusp} and section 
\ref{secContrcoho} we will 
construct explicitly closed harmonic forms in the cusp which 
satisfy relative boundary conditions by adapting Harder's arguments. 
Then we will obtain the last line in Theorem \ref{Theorem} as 
the contribution of the cohomology.
\newline

In our understanding, a very important example of a situation considered by
Lesch would 
be a compact manifold $X$ with a conical singularity such that in the gluing
formula 
$X^+$ is a
smooth 
manifold with boundary and such that $X^-=N$ contains the cone point, see the 
introduction of \cite{Le}. By the 
work of Cheeger \cite{Cheeger 2}, on such manifolds the flat Hodge Laplacians, 
resp.  
their suitably defined selfadjoint extensions, 
have discrete spectrum and moreover 
the analytic torsion can be defined in exactly the same way as on closed
manifolds.  
One now hopes to compare this analytic torsion to the combinatorial torsion of
a suitable 
finite-dimensional complex on $X$. Since the combinatorial torsion 
of a chain complex satisfies a gluing formula which is exactly of the same type
as
the 
gluing formula established by Lesch \cite[Theorem 6.1]{Le}, 
see \cite[Appendix A]{Le}, by Lesch's main result and by
the Cheeger-M\"uller theorems for smooth manifolds with  boundary \cite{Luck},
\cite{Vi}
\cite{BM2}, the conjectured
equality of analytic and combinatorial torsion on the singular manifold $X$
would
be proved if it
was proved on $N$. 

Our gluing situation is however different. The point is that the combinatorial 
torsion we are interested in, in particular due to possible arithmetic
applications sketched below, is the combinatorial torsion
$\tau_{Eis}(\overline{X};E_\rho)$ of the manifold $\overline{X}$. This torsion
can be computed using the 
manifold $X^+$ which appears in our gluing situation. Therefore, 
one does not have to find a combinatorial analog for the regularized
analytic torsion of the cusps $X^-$ with relative boundary conditions. 
Instead, we think that using a separation of variables technique \cite{M2}, one
could obtain an explicit formula for the regularized analytic torsion of the
cusps 
in terms of special values of spectral zeta functions on its 
boundary components, which are just flat tori. A similar result for the case 
of a conical singularity was obtained by Vertman \cite{Ve}. 
What we have to study in our paper, in contrast to Lesch's situation, where 
this term has a combinatorial analog, is
however 
the contribution of the torsion of the long exact cohomology sequence 
to the gluing formula. The result will be the second 
line in our Theorem \ref{Theorem}. 
\newline

We remark that in our earlier paper \cite{Pf2}, a specific comparison 
between certain regularized analytic torsions and certain Reidemeister 
torsions was obtained for the non-compact, finite-volume 3-dimensional
hyperbolic case.
The proof used a completely different method and was based on an explicit 
evaluation of special values 
of Ruelle zeta functions and on a result 
of Menal-Ferrer and Porti \cite{MePo}.
The method 
of \cite{Pf2} does not generalize to higher dimensions and not even to 
general strongly acyclic coefficient systems in the 3-dimensional case. 
Also, the result of \cite{Pf2} only holds for the quotients of 
two torsions associated to different representations of the fundamental group
and 
it does not use Eisenstein cohomology classes as a basis.   
\newline

Next we would like to point out that the assumption $\rho\neq\rho_\theta$ 
is essential for our method. If this assumption is not satisfied, 
then the $L^2$-cohomology with coefficients in $E_\rho$ 
is infinite-dimensional in the middle dimension by a result 
of Borel and Casselmann \cite[4.6 (a)]{BC}. Equivalently \cite{BL}, zero
belongs 
to the essential spectrum of the corresponding flat Hodge Laplace operator. 
Therefore, to treat such a case is far beyond the scope of the present paper. 
We remark that even in the compact locally symmetric case the analytic
torsion becomes 
extremely more complicated if $\rho$ is no longer strongly acyclic, essentially 
due to the same problem of small eigenvalues. In particular, an analog of the
main result
of Bergeron and Venkatesh \cite{BV} on the asymptotic 
behaviour of analytic torsion is, as far as we know, not established for such a
situation even in the compact case.
   
Mainly due to this particular choice of our flat 
bundle we also restrict to the true locally symmetric case and 
do not work with more general manifolds with cusps in the sense of M\"uller
\cite{Mulecn}, 
or, more precisely, with Riemannian manifolds that have the structure of an
odd-dimensional
hyperbolic 
manifold with cusps only outside some compact set. Our 
method would work here if we were given a flat bundle 
whose $L^2$-cohomology would vanish or at least be finite-dimensional; the 
theory of Eisenstein series, which 
we use to define bases in the cohomology, can be extended to such manifolds 
\cite[chapter 8]{Mulecn}. However,
while 
in the locally symmetric case bundles of the required form exist in abundance,
we do not know
whether there is 
a construction of such bundles on more general manifolds with cusps.

We also remark that at least in the compact locally symmetric
case the analytic torsion with coefficients in a flat bundle induced by a
finite-dimensional representation of 
$G$ is always equal to 1 if the dimension of the manifold is even, \cite{MS1},
\cite{BMZ},
\cite{MPcomp}. Moreover, the present odd-dimensional hyperbolic case 
is exactly the case corresponding to an irreducible  symmetric space of real 
rank one which is odd dimesional. A treatment of a suitable analog of the
analytic torsion in the non-compact 
higher rank case is again far beyond the scope of the present article. 
\newline
 
We finally outline how our main result 
could be applied to study the growth of 
torsion in the cohomology of certain arithmetic groups with 
coefficients in certain integral local systems. 
Here, starting with the paper \cite{BV} of Bergeron and Venkatesh, the analytic
torsion
arose in the following context. 
If the locally symmetric space is associated to an 
arithmetic group $\Gamma$ and if the flat bundle 
is associated to a representation of this group on 
a free finite rank $\Z$-module $L$, then, as 
observed by Bergeron and Venkatesh \cite{BV}, see also
\cite{Cheeger 1}, the
Reidemeister torsion encodes information about the sizes of 
the torsion part in the cohomology of $\Gamma$ with coefficients in $L$. 
This relation continues to hold for non-cocompact $\Gamma$ of 
the present form if one works with the Borel-Serre compactification
$\overline{X}$: 
Since $\overline{X}$ is homotopy equivalent to $X$, one can 
still identify $\Gamma$ with the fundamental group of $\overline{X}$ and
moreover the universal
covering of $\overline{X}$ is contractible. We remark 
in bypassing that this is one of the main reasons why we study the Reidemeister
torsion of $\overline{X}$ and 
not for example of a compactification of $X$ obtained by collapsing the boundary
of $\overline{X}$ to a point.
Now, although the (regularized) analytic torsion can in general not be computed
explicitly 
for a fixed locally symmetric space and strongly acyclic coefficient systems,
its
asymptotic 
behaviour can be determined explicitly if either the manifold 
or the local system varies. In the compact 
case, results about the asymptotic behaviour of the analytic torsion were
established by Bergeron and Venkatesh \cite{BV}, by 
Bismut, Ma and Zhang \cite{BMZ} and by M\"uller and the author \cite{MPIMRN},
\cite{MPcomp}.
The 
extension of these results to 
the non-compact, finite volume real rank one case 
is due to Raimbault \cite{Ra} and to
M\"uller and the
author \cite{MP1}, \cite{MP2}. 
The result is always an asymptotic 
equality between analytic and $L^2$-torsion, where the latter is computed
explicitly and 
is in particular not zero. Therefore, in the compact case, using the theorem of
Cheeger \cite{Cheeger 1}, M\"uller \cite{M1}, \cite{Mutorsunim} and Bismut and
Zhang \cite{Bismut}
on the equality of analytic and Reidemeister torsion, one 
can study the size of the torsion part in the cohomology of arithmetic
groups by analytic methods, see \cite{BV},
\cite{Marshall}, 
\cite{MPcoho}. Theorem \ref{Theorem} now implies that one can do this also
in the non compact, finite-volume hyperbolic case. 
\newline

Here we emphasize that in the non-compact case one also has to treat
additionally the regulator, defined as in
\cite{BV}, with 
respect to the Eisenstein cohomology classes. More precisely, the computations 
of section \ref{secbases} show that one has to estimate the denominator 
of the C-matrix in the middle dimension which appears in Proposition
\ref{PropKoho3}. Using 
Poincar\'e-Lefschetz duality and the universal coefficient theorem, the 
problem of studying the image of the integral cohomology of the manifold in 
the integral cohomology of the boundary, which is 
related to denominators of Eisenstein series and
seems to be much more difficult in general, can be avoided at least if one only
wants 
to obtain exponential growth of torsion.
Moreover, the regularized torsion of the cusps and the volume 
factors in Theorem \ref{Theorem} still need to be analyzed asymptotically. 

We would now like to outline how the latter problem could be treated for example
in the 
case where one is given a sequence $\Gamma_i$, $i\in\mathbb{N}$, of arithmetic
groups which 
is exhaustive in a suitable sense; an example would 
be a tower, i.e. $\Gamma_{i+1}\subsetneq
\Gamma_i$ for each $i$ and 
$\cap_i\Gamma_i=\{1\}$. 
We use the setup of the paper \cite{MP2}.
Let $X_i:=\Gamma_i\backslash\mathbb{H}^d$ be the underlying sequence of
hyperbolic manifolds with cusps. Then one has to take into account that the
normalization of the height functions in \cite{MP2} is adapted to the covering,
see
\cite[section 6]{MP2}. We think that exactly this normalization and the 
arguments  of \cite[section 6]{MP2} should allow 
to control the volumes $\vol(\partial F_{X_i})$ in Theorem \ref{Theorem}. For a
treatment of the asymptotic 
behaviour of the regularized analytic torsion of the cusp we remark 
that the condition on cusp uniformity, which is due to Raimbault \cite{Ra} 
and which is also posed in \cite{MP2}, ensures that the geometry of the
boundary $\partial F_{X_i}$ does not degenerate as $i$ tends to infinity.
Therefore, if one can 
express the torsion $\log{T(F_{X_i},\partial F_{X_i};E_\rho)}$ in 
terms of special values of spectral zeta functions on the boundary $\partial
F_{X_i}$ as indicated above, 
then we think that the same considerations as in \cite[section 8]{MP2}, where 
the special values of Epstein zeta functions appearing on the geometric side of
the trace 
formula were treated, guarantee that 
these values remain bounded in the covering. Since the number of 
cusps of $X_i$ grows of lower order than the volume of $X_i$ in the situations 
that we would like to consider \cite{MP2}, this would show 
that in the limit the logarithm of the torsion of the cusps divided by the
volume goes to zero. 
We remark that at least 
in the case of a sequence of congruence subgroups of 
a Bianchi group $\Sl_2(\mathcal{O}_D)$, the contribution
$\log{T(F_{X_i},\partial F_{X_i};E_\rho)}$ can 
even be treated in a simpler way.  
The point is that here, essentially as a consequence of the finiteness of the
class number, there 
exists a finite set of lattices such that each lattice which defines a component
of 
$\partial F_{X_i}$ 
arises by scaling one of these lattices, see \cite[Proposition 4.3]{Ra},
\cite[Proposition 11.1]{MP2}. For
this case, one 
can therefore also apply the gluing formula for the torsion on the cusps, where
we refer to Corollary 
\ref{CorGlCusp}, equation \eqref{BM} and Corollary \ref{Korcoho}. 
\newline

This paper is organized as follows. In section \ref{SecPr} we fix 
some notations concerning the manifolds and representations we work with. In
section \ref{secDR} we introduce 
the bundles we work with. Sections \ref{SecDRcusp}, \ref{liftd} and
\ref{secclcusp} 
are devoted to an 
analysis of the De Rham complex of the flat vector bundles restricted to the
cusp, which 
is needed in order to define the Eisenstein cohomology classes and to compute 
the torsion of the long exact cohomology sequence. In
section \ref{SecER} we recall 
some properties of Eisenstein series used in the present context. 
In section \ref{secbases} we give the explicit description of the cohomology
in terms of Eisenstein series. In 
section \ref{SecReid} we introduce the Reidemeister torsion with respect 
to Eisenstein cohomology classes. 
In section \ref{secrelreg} we give the definition of the relative and
regularized traces 
and establish some of their analytic properties. In section \ref{secPropheatk}
we establish 
results about the heat kernels which are important for our proof of 
the gluing formula. In section \ref{secvar} we introduce the variation of the De
Rham complex. In section \ref{secheatpertL}
we analyze the construction of the heat kernel for 
the perturbed Laplacians. In section \ref{secgl1} we prove the gluing formula
under the additional 
assumption that our metrics are of product structure in a neighbourhood of 
the set where the manifold is divided. In section \ref{SecGlII} we drop this
assumption 
on the metrics and the anomaly term will appear. In 
the final section \ref{secContrcoho} we compute the torsion of the long exact
cohomology sequence that 
appears in the gluing formula explicitly. 
 
\bigskip
{\bf Acknowledgement.}
I would like to thank Werner M\"uller for his constant interest in my work 
and for several useful discussions.

\section{Preliminaries}\label{SecPr}
\addtocontents{toc}{\protect\setcounter{tocdepth}{1}}
\subsection{}
Let $d=2n+1$ be odd and let $G:=\Spin(d,1)$, $K:=\Spin(d)$ or $G:=\SO^0(d,1)$,
$K:=\SO(d)$, 
and $\widetilde{X}:=G/K$.
Let $\gL$ and $\kL$ be the Lie algebras of $G$ and $K$ and let $\theta$ be the
standard Cartan involution of $\gL$ associated to $K$ and let
$\gL=\kL\oplus\pL$ 
be the Cartan decomposition of $\gL$. Let $B$ be the Killing 
form of $\gL$ and  put
\begin{align}\label{metr}
<X,Y>_{\theta}:=-\frac{1}{2(d-1)}B(X,\theta(Y)),\quad X,Y\in
\mathfrak{g}.
\end{align}
Then the restricion of $<\cdot,\cdot>$ to $\pL$ induces a
$G$-invariant 
metric on $\widetilde{X}$. Equipped with this metric, $\widetilde{X}$ is 
isometric to the hyperbolic $d$-space $\mathbb{H}^d$ which carries 
the metric of constant curvature $-1$. 

\subsection{}
Let $G=NAK$ be the Iwasawa decomposition of $G$ as in \cite[section 2]{MP1} and
let $M$ be
the centralizer of $A$ in
$K$. Let $\gL$, $\nL$, $\aL$, $\kL$, $\mL$ denote the Lie algebras
of $G$, $N$, $A$ $K$ and $M$. 
Fix a Cartan subalgebra $\mathfrak{b}$ of $\mathfrak{m}$. 
Then $\mathfrak{h}:=\mathfrak{a}\oplus\mathfrak{b}$
is a Cartan subalgebra of $\mathfrak{g}$. We can identify
$\mathfrak{g}_\C\cong\mathfrak{so}(d+1,\C)$. Let $e_1\in\aL^*$ be the
positive restricted root defining $\mathfrak{n}$.
Then we fix $e_2,\dots,e_{n+1}\in
i\mathfrak{b}^*$ such that 
the positive roots $\Delta^+(\mathfrak{g}_\C,\mathfrak{h}_\C)$ of the 
roots $\Delta(\mathfrak{g}_\C,\mathfrak{h}_\C)$ are chosen as in
\cite[page 684-685]{Knapp2}
for the root system $D_{n+1}$. Let $W_G$ be the 
Weyl group of $\Delta(\gL_\C,\hL_\C)$. We let
$\Delta^+(\mathfrak{g}_\C,\mathfrak{a}_\C)$ be
the set of roots of $\Delta^+(\mathfrak{g}_\C,\mathfrak{h}_\C)$ which do not
vanish on $\aL_\C$. The positive roots
$\Delta^+(\mathfrak{m}_\C,\mathfrak{b}_\C)$
are chosen such that they are restrictions of elements from
$\Delta^+(\mathfrak{g}_\C,\mathfrak{h}_\C)$.

For $j=1,\dots,n+1$ let
\begin{equation}\label{rho}
\rho_{j}:=n+1-j.
\end{equation}
Then the half-sums of positive roots $\rho_G$ and $\rho_M$, respectively, are
given by
\begin{align}\label{Definition von rho(G)}
\rho_{G}:=\frac{1}{2}\sum_{\alpha\in\Delta^{+}(\mathfrak{g}_{\mathbb{C}},
\mathfrak{h}_\mathbb{C})}\alpha=\sum_{j=1}^{n+1}\rho_{j}e_{j};\quad
\rho_{M}:=\frac{1}{2}\sum_{\alpha\in\Delta^{+}(\mathfrak{m}_{\mathbb{C}},
\mathfrak{b}_{\mathbb{C}})}\alpha=\sum_{j=2}^{n+1}\rho_{j}e_{j}.
\end{align}

\subsection{}\label{subsecrep}
Let ${{\mathbb{Z}}\left[\frac{1}{2}\right]}^{j}$ be the set of all 
$(k_{1},\dots,k_{j})\in\mathbb{Q}^{j}$ such that either all $k_{i}$ are 
integers or all $k_{i}$ are half integers. 
Let $\Rep(G)$ denote the set of finite-dimensional irreducible
representations 
$\rho$ of $G$. These  
are parametrized by their highest weights
\begin{equation}\label{Darstellungen von G}
\Lambda(\rho)=k_{1}(\rho)e_{1}+\dots+k_{n+1}(\rho)e_{n+1};\:\:
k_{1}(\rho)\geq 
k_{2}(\rho)\geq\dots\geq k_{n}(\rho)\geq \left|k_{n+1}(\rho)\right|,
\end{equation}
where $(k_{1}(\rho),\dots, k_{n+1}(\rho))$ belongs to
${{\mathbb{Z}}\left[\frac{1}{2}\right]}^{n+1}$ if 
$G=\Spin(d,1)$ and to ${{\mathbb{Z}}}^{n+1}$ if $G=\SO^0(d,1)$. 
Moreover, the finite-dimensional irreducible representations $\nu\in\hat{K}$ of
$K$ are
parametrized by their highest weights
\begin{equation}\label{Darstellungen von K}
\Lambda(\nu)=k_{2}(\nu)e_{2}+\dots+k_{n+1}(\nu)e_{n+1};\:\:
k_{2}(\nu)\geq 
k_{3}(\nu)\geq\dots\geq k_{n}(\nu)\geq k_{n+1}(\nu)\geq 0,
\end{equation} 
where $(k_{2}(\nu),\dots, k_{n+1}(\nu))$ belongs to
${{\mathbb{Z}}\left[\frac{1}{2}\right]}^{n}$ if 
$G=\Spin(d,1)$ and to ${{\mathbb{Z}}}^{n}$ if $G=\SO^0(d,1)$.
Finally, the  finite-dimensional irreducible representations 
$\sigma\in\hat{M}$ of $M$ 
are parametrized by their highest weights
\begin{equation}\label{Darstellungen von M}
\Lambda(\sigma)=k_{2}(\sigma)e_{2}+\dots+k_{n+1}(\sigma)e_{n+1};\:\:
k_{2}(\sigma)\geq 
k_{3}(\sigma)\geq\dots\geq k_{n}(\sigma)\geq \left|k_{n+1}(\sigma)\right|,
\end{equation}
where $(k_{2}(\sigma),\dots, k_{n+1}(\sigma))$ belongs to
${{\mathbb{Z}}\left[\frac{1}{2}\right]}^{n}$, if 
$G=\Spin(d,1)$, and to ${{\mathbb{Z}}}^{n}$, if $G=\SO^0(d,1)$.
For $\nu\in\hat{K}$ and  $\sigma\in \hat{M}$ we denote by $[\nu:\sigma]$ the
multiplicity of $\sigma$ in the restriction of $\nu$ to $M$. 
For $\rho\in\Rep{G}$ let $\rho_{\theta}:=\rho\circ\theta$. Let $\Lambda(\rho)$ 
denote the highest weight of $\rho$ as in \eqref{Darstellungen von G}.
Then the highest weight $\Lambda(\rho_\theta)$ of $\rho_\theta$ is given by
\begin{equation}\label{Tau theta}
\Lambda(\rho_{\theta})=k_{1}(\rho)e_{1}+\dots+k_{n}(\rho)e_{n}-k_{n+1}(\rho)e_{
n+1}.
\end{equation}

\subsection{}
Let $M'$ be the normalizer of $A$ in $K$ and let $W(A)=M'/M$ be the 
restricted Weyl-group. It has order two and it acts on the finite-dimensional 
representations of $M$ as follows. Let $w_{0}\in W(A)$ be the non-trivial 
element and let $m_0\in M^\prime$ be a representative of $w_0$. Given 
$\sigma\in\hat M$, the representation $w_0\sigma\in \hat M$ is defined by
\[
w_0\sigma(m)=\sigma(m_0mm_0^{-1}),\quad m\in M.
\] 
If $d=2n+2$ one has $w_0\sigma\cong\sigma$ for every $\sigma\in\hat{M}$. 
Assume that $d=2n+1$. Let
$\Lambda(\sigma)=k_{2}(\sigma)e_{2}+\dots+k_{n+1}(\sigma)e_{n+1}$ be the 
highest weight
of $\sigma$ as in \eqref{Darstellungen von M}. Then the highest weight 
$\Lambda(w_0\sigma)$ of $w_0\sigma$ is given by
\begin{equation}\label{wsigma}
\Lambda(w_0\sigma)=k_{2}(\sigma)e_{2}+\dots+k_{n}(\sigma)e_{n}
-k_{n+1}(\sigma)e_{n+1}.
\end{equation}

\subsection{}\label{SecP0}
Let $P_0:=NAM$. We equip $\aL$ with the norm induced from the restriction of the
normalized Killing form on $\gL$. Let $H_1\in\aL$ be
the unique vector which is of norm one and such that the positive restricted
root,
implicit in the choice of $N$, is positive on $H_1$. Let
$\exp:\aL\to
A$ be
the exponential map.
Every $a\in A$ can be written 
as $a=\exp{\log{a}}$, where $\log{a}\in\mathfrak{a}$ is unique.
For $t\in\mathbb{R}$, we let 
$a(t):=\exp{(tH_{1})}$. If $g\in G$, we define $n(g)\in N$, $
H(g)\in \R$ and $\kappa(g)\in K$ by 
\begin{align*}
g=n(g)a(H(g))\kappa(g). 
\end{align*}

Let $P$ be any parabolic subgroup. Then 
there exists a $k_{P}\in K$ such that
$P=N_{P}A_{P}M_{P}$ with $N_{P}=k_{P}Nk_{P}^{-1}$,
$A_{P}=k_{P}Ak_{P}^{-1}$, $M_{P}=k_{P}Mk_{P}^{-1}$. We choose a set of
$k_{P}$'s, which will be fixed from now on. Let
$k_{P_0}=1$.
We let $a_{P}(t):=k_{P}a(t)k_{P}^{-1}$. If $g\in G$, we define
$n_{P}(g)\in N_{P}$, $H_{P}(g)\in
\mathbb{R}$ and $\kappa_{P}(g)\in K$ by 
\begin{align}\label{eqg}
g=n_{P}(g)a_{P}(H_{P}(g))\kappa_{P}(g) 
\end{align}
and we define an identification $\iota_{P}$ of
$\left(0,\infty\right)$ with
$A_{P}$ by $\iota_{P}(t):=a_{P}(\log(t))$. 
For $Y>0$, let $A^{0}_{P}\left[Y\right]:=\iota_{P}(Y,\infty)$ and
$A_{P}\left[Y\right]:=\iota_{P}[Y,\infty)$. For $g\in G$ as in \eqref{eqg} we
let $y_{P}(g):=e^{H_{P}(g)}$.

\subsection{}\label{subsps}
We parametrize the principal series as follows. Given $\sigma\in\hat{M}$ with
$(\sigma,V_\sigma) \in \sigma$, let $\mathcal{H}^{\sigma}$ denote the space of
measurable functions $f\colon K\to V_\sigma$ satisfying
\[
f(mk)=\sigma(m)f(k),\quad\forall k\in K,\, \forall m\in M,
\quad\textup{and}\quad
\int_K\parallel f(k)\parallel^2\,dk=\parallel f\parallel^2<\infty.
\]
Then for $\lambda\in\mathbb{C}$ and $f\in H^{\sigma}$ let
\begin{align*}
\pi_{\sigma,\lambda}(g)f(k):=e^{(i\lambda+(d-1)/2)H(kg)}f(\kappa(kg)).
\end{align*}
Then the representations $\pi_{\sigma,\lambda}$ are unitary iff 
$\lambda\in\mathbb{R}$. Moreover, for $\lambda\in\mathbb{R}-\{0\}$ and 
$\sigma\in\hat{M}$ the representations $\pi_{\sigma,\lambda}$ are irreducible 
and $\pi_{\sigma,\lambda}$ and $\pi_{\sigma',\lambda'}$, $\lambda,
\lambda'\in\mathbb{C}$ are 
equivalent iff either $\sigma=\sigma'$, $\lambda=\lambda'$ or
$\sigma'=w_{0}\sigma$, 
$\lambda'=-\lambda$. We use this parametrization of the principal series 
in order to stay consistent with the notation of \cite{MPIMRN}, \cite{MP1},
\cite{MP2}.
For $\sigma\in\hat{M}$ with highest weight $\Lambda(\sigma)$ given by
\eqref{Darstellungen von M}, let
\begin{align}\label{csigma}
c(\sigma):=\sum_{j=2}^{n+1}(k_{j}(\sigma)+\rho_{j})^{2}-\sum_{j=1}^{n+1}
\rho_{j}^{2}.
\end{align}
Then for the Casimir element $\Omega\in Z(\gL_\C)$ one has
\begin{align*}
\pi_{\sigma,\lambda}(\Omega)=-\lambda^{2}+c(\sigma),
\end{align*}
\cite[Corollary 2.4]{MPIMRN}.
If $V$ is a $(\gL,K)$-module, defined as in \cite{BW}[section 0.2], we 
let $H^*(\gL,K:V)$ denote its $(\gL,K)$-cohomology groups. 
If one lets
\begin{align*}
C^q(\gL,K;V):=\Hom_K(\Lambda^q(\gL/\kL),V_\rho),
\end{align*}
then the $C^q(\gL,K;V)$ naturally form a cochain complex  $C^*(\gL,K;V)$
whose cohomology groups can be identified with the groups $H^*(\gL,K;V)$,
\cite{BW}[Chapter I, section 1]. By $\mathcal{H}^{\sigma,\lambda}_K$ 
we shall denote the $(\gL,K)$-module  associated to $\pi_{\sigma,\lambda}$
which 
is formed by the $K$-infinite vectors in
$\mathcal{H}^\sigma$.

\subsection{}\label{subsmfld}
Now let $X$ be an oriented hyperbolic manifold. Then there exists a discrete, 
torsion-free subgroup $\Gamma$ of $G$ such that 
\begin{align*}
X=\Gamma\backslash\widetilde{X}. 
\end{align*}
One can canonically identify $\Gamma$ with the fundamental group of $X$.
We assume that  $X$ is of finite volume but we do not assume that $X$ is
compact.
A parabolic subgroup $P$ of $G$ is called $\Gamma$-cuspidal if $\Gamma\cap
N_{P}$ 
is a lattice in $N_{P}$. Let $\mathfrak{P}_\Gamma$ be a fixed set of
representatives of the 
$\Gamma$-conjugacy classes of 
$\Gamma$-cuspidal 
parabolic subgroups of $G$. Then $\mathfrak{P}_\Gamma$ is finite and
$\kappa(\Gamma):=\#\mathfrak{P}_\Gamma$ 
equals the number of cusps of $X$; the manifold $X$ is non-compact if and only
if $\kappa(\Gamma)>0$. 
For convenience we assume that for 
each $\Gamma$-cuspidal parabolic subgroup $P=N_{P}A_{P}M_{P}$ of $G$ 
one has 
\begin{align}\label{asGamma}
\Gamma\cap P=\Gamma\cap N_{P}.
\end{align}
This condition is satisfied for example if $\Gamma$ is ``neat'', which means 
that the group generated by the eigenvalues of any $\gamma\in\Gamma$ contains 
no roots of unity $\ne1$. It thus holds for many groups $\Gamma$ which are 
of arithmetic significance: For example, if $F$ is an imaginary quadratic number
field 
with ring of integers $\mathcal{O}_F$ and $\aL$ is a non-zero ideal in
$\mathcal{O}_F$, 
then the associated principal congruence subgroup $\Gamma(\aL)$ of
$\Sl_2(\mathcal{O}_F)$ satisfies 
assumption \eqref{asGamma} for all ideals $\aL$ whose 
norm is greater than a given constant that depends only on $F$. 
Similarly, in higher dimensions, if $G:=\SO^0(d,1)$ 
assumption \eqref{asGamma} holds for 
all principal congruence subgroups $\Gamma(q)$ of $G(\Z)$ associated 
to $q\in\mathbb{N}$ with $q\geq 3$, see
\cite[17.4]{BorelArGr}. 

Assumption \eqref{asGamma} was made in the paper \cite{MP1}, where 
the regularized analytic torsion was defined for the present context. 
However, in that paper it was mainly 
put in order to study the asymptotic behaviour of 
the analytic torsion along rays in the weight lattice, which relied on Fourier 
inversion formulas on the geometric 
side of the trace formula. For 
the pure short-time asymptotic expansion of the heat traces, it is certainly 
not necessary.
We think that it can be dropped without an essential change of 
the methods used in the present paper. 

The geometry of the quotient $\Gamma\backslash G$ and of the manifold $X$ can be
described as follows. 
Firstly, there exists a $Y_{0}>0$ and for every $Y\geq Y_{0}$ a compact
connected
subset $C(Y)$ of $G$  such that in the sense of a disjoint union one has
\begin{align}\label{Zerlegung des FB}
G=\Gamma\cdot C(Y)\sqcup\bigsqcup_{P_i\in\mathfrak{P}_\Gamma}\Gamma\cdot
N_{P_i}A^{0}_{P_i}\left[Y\right]K
\end{align}
and such that for each $P_i\in\mathfrak{P}_\Gamma$ one has 
\begin{align}\label{Eigenschaft des FB}
\gamma\cdot N_{P_i}A^0_{P_i}\left[Y\right]K\cap
N_{P_i}A_{P_i}^{0}\left[Y\right]K\neq
\emptyset\Leftrightarrow\gamma\in \Gamma\cap N_{P_i}. 
\end{align}
For $Y\in (0,\infty)$ and $P_i\in \mathfrak{P}_\Gamma$ we let 
\begin{align}\label{Definition der Spitze}
F_{P}(Y_i):=A_{P_i}\left[Y\right]\times\Gamma\cap N_{P_i}\backslash N_{P_i}\cong
[Y,\infty)\times
\Gamma\cap N_{P_i}\backslash N_{P_i}.
\end{align}
It follows from \eqref{Zerlegung des FB} and \eqref{Eigenschaft des FB} that 
for each $Y\geq Y_0$ there exists a compact manifold $X(Y)$ with smooth boundary
such that $X$
has a decomposition as 
\begin{align}\label{Zerlegung X}
X=X(Y)\cup \bigsqcup_{P_i\in\mathfrak{P}_\Gamma}F_{P_i}(Y)
\end{align}
with
$X(Y)\cap F_{P_i,Y}=\partial X(Y)\cap\partial F_{P_i}(Y)=\partial F_{P_i}(Y)$
and $F_{P_i}(Y)\cap
F_{P_j}(Y)=\emptyset$ if $i\neq j$ and with $\partial
X(Y)=\sqcup_{P_i\in\mathfrak{P}_\Gamma}\partial F_{P_i}(Y)$.
Let $g_{N_{P_i}}$ be the pushdown of the invariant metric on $N_{P_i}$ induced
by the metric in \eqref{metr}  to $\Gamma\cap N_{P_i}\backslash N_{P_i}$ . 
Then the metric on $F_{P_i}(Y)$ is given by
\begin{align}\label{metrik}
\frac{1}{y^{2}}dy^{2}+\frac{1}{y^{2}}g_{N_{P_i}}.
\end{align}
Since  $N_{P_i}$ is abelian in the present case, $\Gamma\cap N_{P_i}\backslash
N_{P_i}$ 
with the metric $g_{N_{P_i}}$ is isometric to a flat torus. 
Equations \eqref{Zerlegung X} and \eqref{metrik} describe the geometry 
of $X$ as a manifold with cusps. We define the Borel-Serre compactification 
$\overline{X}$ of $X$ as the diffeomorphism class of the manifolds $X(Y)$,
$Y\geq Y_0$. 
Thus $\overline{X}$ is a compact smooth manifold with boundary which 
is homotopy equivalent to $X$. 
For $Y\in(0,\infty)$ we shall use the notation
\begin{align*}
F_X(Y):=\bigsqcup_{P_i\in\mathfrak{P}_\Gamma}F_{P_i}(Y). 
\end{align*}
We furthermore let $F_X:=F_X(1)$. 
By \eqref{metrik}, under the isomorphism $F_{P_i,Y}\cong
[Y,\infty)\times\Gamma\cap
N_{P_i}\backslash N_{P_i}$, for $Y\geq Y_0$ 
one has 
\begin{align}\label{integral}
\int_{X}f(x)dx=\int_{X(Y)}f(x)dx+\sum_{P_i\in\mathfrak{P}_\Gamma}
\int_Y^\infty\int_{\Gamma\cap N_{P_i}\backslash N_{P_i}}y^{-(2n+1)}f(y,v)
dg_{N_{P_i}}(v)dy.
\end{align}

\section{Homogeneous and flat vector
bundles}\label{secDR}
\setcounter{equation}{0}
Let $X$ be the manifold from the previous section. 
In this section we introduce the class of vector bundles over $X$ and
$\tilde{X}$ we will be concerned with. These bundles 
were firstly introduced by Matsushima and Murakmami in \cite{Mats} and we refer
to 
this paper for further details. 

Let $\nu$ be a finite-dimensional unitary representation of $K$ on 
$(V_{\nu},\left<\cdot,\cdot\right>_{\nu})$. Let
$\tilde{E}_{\nu}:=G\times_{\nu}V_{\nu}$
be the associated homogeneous vector bundle over $\tilde{X}$. Then 
$\left<\cdot,\cdot\right>_{\nu}$ induces a $G$-invariant metric 
$\tilde{h}_{\nu}$ on $\tilde{E}_{\nu}$. 
Let $E_{\nu}:=\Gamma\backslash(G\times_{\nu}V_{\nu})$
be the associated locally homogeneous bundle over $X$. Since 
$\tilde{h}_{\nu}$ is $G$-invariant, it can be
pushed down to a fiber metric $h_{\nu}$ on 
$E_{\nu}$. Let
\begin{align}\label{globsect}
C^{\infty}(G,\nu):=\{f:G\rightarrow V_{\nu}\colon f\in C^\infty,\;
f(gk)=\nu(k^{-1})f(g),\,\,\forall g\in G, \,\forall k\in K\}.
\end{align}
Let
\begin{align}\label{globsect1}
C^{\infty}(\Gamma\backslash G,\nu):=\left\{f\in C^{\infty}(G,\nu)\colon 
f(\gamma g)=f(g),\;\forall g\in G, \;\forall \gamma\in\Gamma\right\}.
\end{align}
Then the smooth sections $C^{\infty}(\tilde{X},E_{\nu})$ of $\tilde{E}_{\nu}$ 
are 
canonically isomorphic to $C^\infty(G,\nu)$.
Similarly,  for the sapce $C^{\infty}(X,E_{\nu})$ of smooth sections of
$E_{\nu}$ 
there is a canonical isomorphism
$C^{\infty}(X,E_{\nu})\cong C^{\infty}(\Gamma\backslash G,\nu)$
(see \cite[p. 4]{Mi1}).
There is also a corresponding isometry for the space $L^{2}(X,E_{\nu})$ of 
$L^{2}$-sections of $E_{\nu}$.

Let $\rho$ be an irreducible finite-dimensional representation of $G$ on
$V_{\rho}$ and let $E_{\rho}:=\widetilde{X}\times_{\rho|_\Gamma}V_\rho$ be the
flat
vector bundle over $X$ associated to the
restriction of $\rho$ to $\Gamma$. Let $\widetilde E_\rho\to \widetilde X$ be 
the homogeneous vector bundle over $\tilde{X}$ associated to $\rho|_K$. Then by
\cite[Proposition 3.1]{Mats}
there is a canonical isomorphism 
\begin{align}\label{IsoMats}
E_\rho\cong\Gamma\backslash\widetilde E_\rho.
\end{align}
By \cite[Lemma 3.1]{Mats}, there exists an inner product
$\left<\cdot,\cdot\right>$ on
$V_{\rho}$ such that
\begin{enumerate}
\item $\left<\rho(Y)u,v\right>=-\left<u,\rho(Y)v\right>$ for all
$Y\in\mathfrak{k}$, $u,v\in V_{\rho}$
\item $\left<\rho(Y)u,v\right>=\left<u,\rho(Y)v\right>$ for all
$Y\in\mathfrak{p}$, $u,v\in V_{\rho}$.
\end{enumerate}
Such an inner product is called admissible. It is unique up to scaling. Fix an
admissible inner product. Since $\rho|_{K}$ is unitary with respect to this
inner product, it induces a fibre metric on $\widetilde E_\rho$, and hence a
fibre metric $h$
on $E_\rho$. This fibre metric will also  be called admissible. Let 
$\Lambda^{p}(X,E_{\rho})$ be the space of $E_\rho$-valued $p$-forms.
This is the space of smooth sections of the vector bundle
$\Lambda^{p}(E_{\rho}):=
\Lambda^pT^*X\otimes E_\rho$. Let
\begin{equation}
d_{p}(\rho)\colon \Lambda^{p}(X,E_{\rho})\to \Lambda^{p+1}(X,E_{\rho})
\end{equation}
be the exterior derivative and let
\begin{equation}\label{laplace}
\Delta_p(\rho)=d_p(\rho)^* d_p(\rho)+d_{p-1}(\rho)d_{p-1}(\rho)^*
\end{equation}
be the Laplace operator on $E_\rho$-valued $p$-forms. 
Let $\nu_{p}(\rho)$ be the
representation of $K$ defined by
\begin{equation}\label{nutau}
\nu_{p}(\rho):=\Lambda^{p}\Ad^{*}\otimes\rho:\:K\rightarrow\Gl(\Lambda^{p}
\mathfrak{p}^{*}\otimes V_{\rho}).
\end{equation}
Then there is a canonical isomorphism
$\Lambda^{p}(E_{\rho})\cong\Gamma\backslash(G\times_{\nu_{p}(\rho)}(\Lambda^{p}
\mathfrak{p}^{*}\otimes V_{\rho}))$,
which induces an isomorphism
\begin{align}\label{sections}
\Lambda^{p}(X,E_{\rho})\cong C^{\infty}(\Gamma\backslash G,\nu_{p}(\rho)).
\end{align}
There is a corresponding isometry of the $L^{2}$-spaces. 
We let 
\begin{align*}
d:C^\infty(\Gamma\backslash G,\nu_p(\rho))\to C^\infty(\Gamma\backslash
G,\nu_{p+1}(\rho))
\end{align*}
be the map which is induced by the map $d_p(\rho)$ under the isomorphism in
\eqref{sections} 
and we denote the corresponding
complex by $\Omega^*(\Gamma\backslash G,\rho)$.
Then there is a canonical isomorphism of complexes 
\begin{align*}
\Omega^*(\Gamma\backslash G,\rho)\cong C^*(\gL,K;C^\infty(\Gamma\backslash
G)\otimes V_\rho),
\end{align*}
where the complex on the right hand side is as in section \ref{subsps}.
Let $\rho(\Omega)$ be the Casimir eigenvalue of $\rho$. Then with respect to the
isomorphism \eqref{sections} one has
\begin{align}\label{kuga}
\Delta_{p}(\rho)=-\Omega+\rho(\Omega)\Id
\end{align}
(see \cite[(6.9)]{Mats}).

\section{The De Rham complex on the cusp}\label{SecDRcusp}
\setcounter{equation}{0}

Let $\Gamma$ be as in section \ref{subsmfld} and let $P$ be a fixed
$\Gamma$-cuspidal parabolic subgroup of $G$. 
In this section, we describe the De Rham complex with coefficients 
the flat vector bundles corresponding to the bundles $\tilde{E}_\rho$ over the
cusp $F_P:=(\Gamma\cap N_P)\backslash G/K$ 
of infinite volume. Here we equip $F_P$ with the metric induced from the metric 
on $G/K$. 

Without loss of generality, we assume that 
$P=P_0$ is the parabolic subgroup of $G$ from section \ref{SecP0} and to ease
notation we shall
write 
$\Gamma_N:=\Gamma\cap N$. 
Analogous to the notation of the previous section, for $\rho\in\Rep(G)$ 
we shall denote by $E_\rho:=\widetilde{X}\times_{\rho|_{\Gamma_N}}V_\rho$
the flat vector bundle over $F_P$ induced by the restriction 
of $\rho$ to $\Gamma_N$.
The restriction of $\rho$ to $K$ 
defines a locally homogeneous vector bundle $W_\rho:=\Gamma_N\backslash
G\times_{\rho|K}V_\rho$ over $F_P$ and analogously to the situation 
of the previous section there is a
canonical 
isomorphism $A:W_\rho\to E_\rho$. More precisely, $A$ is defined by putting 
\begin{align}\label{Iso1}
A[\Gamma_N g,v]:=[ gK,\rho(g)v],
\end{align}
for an equivalence
class $[\Gamma_N g,v]\in W_\rho$, 
$g\in G$, $v\in V_\rho$ . 
The smooth sections of $W_\rho$ are canonically isomorphic to the space
\begin{align*}
C^\infty(\Gamma_N\backslash G,\rho)=\{f:G\to
V_\rho\colon f\in C^\infty\colon f(\gamma g
k)=\rho(k^{-1})f(g),\forall g\in G,\:\gamma\in\Gamma_N,\: k\in K \}.
\end{align*}
The smooth sections of $E_\rho$ are
canonically 
isomorphic to the space 
\begin{align*}
C^\infty(G/K,\rho|_{{\Gamma_N}}):=\{f:G\to V_\rho\colon f\in C^\infty\colon
f(\gamma gk)=\rho(\gamma)f(g), \forall g\in G,\: k\in K,\:
\gamma\in\Gamma_N\}.
\end{align*}
If we denote by $A: C^\infty(\Gamma_N\backslash G,\rho)\to
C^\infty(G/K,\rho|_{{\Gamma_N}})$ 
the isomorphism corresponding to the isomorphism $A$ from \eqref{Iso1}, then 
for $f\in C^\infty(\Gamma_N\backslash G,\rho|_K)$, $h\in
C^\infty(G/K,\rho|_{{\Gamma_N}})$ we
have
\begin{align}\label{IsoII}
Af(gK)=\rho(g)f(g);\quad A^{-1}h(\Gamma_Ng)=\rho(g^{-1})h(g).
\end{align}
Now we let $\Lambda^p(F_P,E_\rho)$ be the bundle of
$E_\rho$-valued 
$p$-forms on $F_P$. We let 
$d:\Lambda^p(F_P,E_\rho)\to\Lambda^{p+1}(F_P,E_\rho)$ be the exterior 
derivative. Then for smooth vector fields $\tilde{Y_1}, \dots,\tilde{Y}_{p+1}$
on $F_P$ and $f\in\Lambda^p(F_P,E_\rho)$ we have
\begin{align}\label{Formeld}
df(\tilde{Y}_1,\dots,\tilde{Y}_{p+1})=&\sum_{i}(-1)^{i+1}\tilde{Y}_i
\left(f(\tilde{Y}_1,\dots,\hat
{\tilde{Y_i}},
\dots,\tilde{Y}_{p+1})\right)\nonumber\\
+&\sum_{i<j}(-1)^{i+j}
f([\tilde{Y_i},\tilde{Y_j}],\tilde{Y}_1,\dots,\hat{\tilde{Y_i}},\dots,\hat{
\tilde{Y_j}},\dots,\tilde{Y}_{p+1}).
\end{align}
One has $\gL/\kL\cong\aL\oplus\nL$ and one denotes by $\Ad$ the 
action of $K$ on the right hand side which is induced by the 
corresponding action on the left hand side. Let
\begin{align*}
\nu_p(\rho):K\to\Aut(\Lambda^p(\aL\oplus\nL)^*\otimes V_\rho ),\quad
\nu_p(\rho)(k):=\Lambda^p\Ad^*(k)\otimes\rho(k).
\end{align*}
Let $W_{\nu_p(\rho)}:=\Gamma_N\backslash
G\times_{\nu_p(\rho)}\Lambda^p(\aL\oplus\nL)^*\otimes V_\rho $ be the locally
homogeneous bundle over $F_P$ induced by
$\nu_p(\rho)$. Then 
$W_{\nu_p(\rho)}$ is canonically isomorphic to  the bunde of $W_\rho$-valued
$p$-forms on $F_P$. 
Moreover, the smooth section of $W_{\nu_p(\rho)}$ are canonically isomorphic to
the space $C^\infty(\Gamma_N\backslash G,\nu_p(\rho))$ which is defined as the 
space $C^\infty(\Gamma_N\backslash G,\rho|_{K})$, replacing $\rho|_K$ by
$\nu_p(\rho)$. 
We let 
\begin{align*}
d:C^\infty(\Gamma_N\backslash G,\nu_p(\rho))\to C^\infty(\Gamma_N\backslash
G,\nu_{p+1}(\rho))
\end{align*}
be the map which is induced by the exterior derivative 
under the canonical identification of these space with $\Lambda^p(F_P,E_\rho)$
resp. $\Lambda^{p+1}(F_P,E_\rho)$ and we denote the corresponding
complex
by $\Omega^*(\Gamma_N\backslash G,\rho)$.
Then it follows easily from  \eqref{IsoII} 
and \eqref{Formeld} that 
for any vectors
$T_1,\dots,T_{p+1}\in(\aL\oplus\nL)$ and all $\phi\in
C^\infty(\Gamma_N\backslash G,\nu_p(\rho))$ we have
\begin{align}\label{eqextd}
d\phi(T_1,\dots,T_{p+1})=&\sum_{i}(-1)^{i+1}(\tilde{T_i}
+\rho(T_i))\left(\phi(T_1,\dots,\hat
{T_i},
\dots,T_{p+1})\right)\nonumber\\ +&\sum_{i<j}(-1)^{i+j}
\phi([T_i,T_j],T_1,\dots,\hat{T_i},\dots,\hat{T_j},\dots,T_{p+1}),
\end{align}
where $\tilde{T}_i$ denotes the left-invariant vector field on $G$ induced by
$T_i$. 

\begin{bmrk}
If $\gL=\kL\oplus\pL$ is the Cartan decompostion of $\gL$ with respect to 
$\theta$ and if one works with vector-fields constructed out of elements of
$\pL$, as it 
is done in \cite[Proposition 4.1]{Mats}, then the last line in \eqref{eqextd} is
not present since one 
always has $[\pL,\pL]\subset\kL$.
\end{bmrk}

Finally, there is a canonical isomorphism of complexes 
\begin{align}\label{IsomKompl}
\Omega^*(\Gamma_N\backslash G,\rho)\cong C^*(\gL,K;C^\infty(\Gamma_N\backslash
G)\otimes V_\rho),
\end{align}
where the complex on the right hand side is as in section \eqref{subsps}. 
Here the action of $G$ on $C^\infty(\Gamma_N\backslash G)\otimes V_\rho$ 
is the tensor product of the right-regular representation of $G$ and 
the representation $\rho$.  
The isomorphism in \eqref{IsomKompl} is defined as follows. To 
$f\in \Omega^*(\Gamma_N\backslash G,\nu_p(\rho))$ one associates
the element in
$\Hom_K(\Lambda^p(\gL/\kL);C^\infty(\Gamma_N\backslash G)\otimes V_\rho)$
which assigns to $X\in\Lambda^p(\gL/\kL)$ the $V_\rho$-valued function 
$f_{X}$ on $G$ given by
$f_{X}(g):=f(g)(X)$. Here one uses the canonical isomorphisms of $K$-modules: 
$\Lambda^p(\aL\oplus\nL)^*\otimes V_\rho\cong
\Hom(\Lambda^p(\aL\oplus\nL),V_\rho)\cong
\Hom(\Lambda^p(\gL/\kL),V_\rho)$.

\section{Lifts of differential forms from the boundary to the cusp}\label{liftd}
\setcounter{equation}{0}
We keep the notations of the previous section. 
We let $E_\rho^0:=N\times_{\rho|_{\Gamma_N}}V_\rho$ be the flat vector 
bundle over $\Gamma_N\backslash N$ induced by $\rho|_{\Gamma_N}$. 
We can identify the smooth sections of 
$E_{\rho}^0$ with the space
\begin{align*}
C^\infty(N,\rho|_{\Gamma_N}):=\{f:N\to V_\rho\colon f\in C^\infty\colon f(\gamma
n)=\rho(\gamma)f(n),\:\forall \gamma\in\Gamma_N, \forall n\in N \}.
\end{align*}
We let $Y_1,\dots,Y_{2n}$ be a basis of $\nL$ and 
we let $\tilde{Y}_1,\dots,\tilde{Y}_{2n}$ denote the corresponding 
invariant vector fields on $\Gamma_N\backslash N$. 
We can canonically inject the space $\Lambda^p\nL^*\otimes V_\rho$ 
of $V_\rho$-valued $p$-forms on $\nL$ into the space of $E_\rho^0$-valued
$p$-forms 
on $\Gamma_N\backslash N$. Namely, to any $\Phi\in\Lambda^p\nL^*\otimes V_\rho$
we associate the unique $E_\rho^0$-valued $p$-form $\Phi$ on $\Gamma_N\backslash
N$ such that one has in  $C^\infty(N,\rho|_{\Gamma_N})$:
\begin{align}\label{mapPhi}
\Phi(\Gamma_N
n)(\tilde{Y}_{i_1},\dots,\tilde{Y}_{i_p}):=\rho(n)\Phi(Y_{i_1},\dots,Y_{i_p}),
\quad \forall n\in N. 
\end{align}
Here $\Gamma_Nn$ denotes the image of $n$ in $\Gamma_N\backslash N$. 
This form exists since in our case $N$ is abelian - otherwise one would 
have to additionally twist with the adjoint action of $N$ on $\nL$. The maps in
\eqref{mapPhi} also induce a chain map between  the 
$V_\rho$-valued Lie-algebra cohomology of $\nL$ and the De-Rham complex of
$E_\rho^0$-valued
differential 
forms on $\Gamma_N\backslash N$. By Van Est's Theorem \cite{van Est}, the
corresponding 
map on the cohomology is an isomorphism.

Next we lift $V_\rho$-valued $p$-foms on $\nL$ to differential forms on $F_P$.
The space $\Lambda^p(\nL)^*$ is not invariant under $\Ad(K)$. However, we 
can regard the space $\Lambda^p(\nL)^*\otimes V_\rho$ as a subspace of
$\Lambda^p(\nL\oplus\aL)^*\otimes V_\rho$ and thus to 
$\Phi\in \Lambda^p(\nL)^*\otimes V_\rho$
we can associate the element 
\begin{align}\label{DefPhilambda}
\Phi_\lambda\in C^\infty(\Gamma_N\backslash G,\nu_p(\rho))
, \quad \Phi_{\lambda}(na(t)k):=e^{(\lambda+(d-1)/2)t}
\nu_p(\rho)(k^{-1})\Phi,
\end{align}
where $n\in N$, $a(t)\in A$ as in section \ref{SecP0}, $k\in K$.
Let $H_1\in\aL$ be as in section \ref{SecP0}.  
Then the differential of $\Phi_\lambda$ is given as follows.
\begin{lem}\label{Lemd}
For $n\in N$, $a(t)\in A$ and $k\in K$ we have
\begin{align*}
d\Phi_\lambda(na(t)k)=e^{\left(\lambda+\frac{d-1}{2}\right)t}\nu_{p+1}(\rho)(k^{
-1}
)\biggl(dH_1\wedge\bigl((\lambda+(d-1)/2)\Phi
+\left(\Lambda^p\Ad^*\otimes\rho\right)(H_1)\Phi\bigr)\biggr). 
\end{align*}
\end{lem}
\begin{proof}

To prove the Lemma, we can assume that $k=1$. By the definition of
$\Phi_\lambda$, it suffices to
evaluate $d\Phi_\lambda$ on 
$p+1$ - tuples $(H_1,Y_1,\dots,Y_p)$, where $Y_1,\dots,Y_p\in\nL$.
Applying \eqref{eqextd} and using that the interior multiplication of
$\Phi_\lambda$ with 
$H_1$ vanishes, we obtain
\begin{align*}
d\Phi_\lambda(na(t))(H_1,Y_1,\dots,Y_p)=&
e^{\left(\lambda+\frac{d-1}{2}\right)t}\bigl((\lambda+(d-1)/2)\Phi(Y_1,\dots,
Y_p)+\rho(H_1)\Phi(Y_1,
\dots,
Y_p)\bigr)
\\ &+\sum_{i=1}^p(-1)^{i}\Phi([H_1,Y_i],\dots,\hat{Y_i},\dots,Y_{p}).
\end{align*}
For the representation 
$\Lambda^p\Ad^*\otimes\rho :\aL\to\End(\Lambda^p(\nL)^*\otimes V_\rho)$
one has
\begin{align*}
\left(\Lambda^p\Ad^*\otimes\rho\right)(H_1)\Phi(Y_1,\dots,Y_p)=
-\sum_i\Phi(Y_1,\dots,[H_1,Y_i],\dots,Y_p)+\rho(H_1)\Phi(Y_1,\dots,Y_p)
\end{align*}
and the Lemma follows.  
\end{proof}

For every $a\in A$ we have an inclusion 
$\iota_a: \Gamma_N\backslash N\hookrightarrow F_P$ given by
$\iota_a(\Gamma_Nn):=\Gamma_NnaK$, where $\Gamma_Nn$ resp. $\Gamma_NnaK$ denote
the 
equivalence classes of $n$ in $\Gamma_N\backslash N$ resp. $na$ in
$\Gamma_N\backslash G/K$.  
Let $\iota_a^*E_\rho$ be the corresponding 
pullback bundle of $E_\rho$. Then $\iota_a^*E_\rho$ 
is isomorphic to $H^0_\rho$. Thus also $\iota_a^*W_\rho$ 
is isomorphic to $H^0_\rho$, where $W_\rho$ is the locally
homogeneous bundle from the previous section. 
Let 
$\iota_a^*:C^\infty(\Gamma_N\backslash G,\rho|_{K}) \to
C^\infty(N;\rho|_{\Gamma_N})$
be the map between the smooth sections induced by $\iota_a$. Then, by
\eqref{IsoII},
for $f\in C^\infty(\Gamma_N\backslash G,\rho|_{K})$ one has
\begin{align}\label{equiota}
\iota_a^*f(n)= \rho(na)f(na). 
\end{align}
The pullback of $\Phi_\lambda$ under $\iota_a$ is computed as follows. 
\begin{lem}\label{Pullback}
For every $a(t)\in A$ we have
\begin{align}
\iota_{a(t)}^*\Phi_\lambda=e^{(\lambda+(d-1)/2)t}
(\Lambda^p\Ad^*\otimes\rho)(a(t))\Phi.
\end{align}
\end{lem}
\begin{proof}
We identify elements of the Lie algebra $\nL$ with the corresponding left 
invariant vector fields on $N$ resp. $G$. Let $\pi_1:N\to\Gamma_N\backslash N$, 
$\pi_2:G\to\Gamma_N\backslash G/K$ be the projections. Then for
$Y\in\nL$ one has $(\iota_a)_*(\pi_{1,*}(Y))=\pi_{2,*}(\Ad(a^{-1})Y)$.
Thus the Lemma follows immediately from equation \eqref{mapPhi} and equation
\eqref{equiota}. 
\end{proof}

\section{Closed forms on the cusp}\label{secclcusp}
\setcounter{equation}{0}
We consider again the representation $\Lambda^p\Ad^*\otimes\rho$ 
of $MA$ on $\Lambda^p\nL^*\otimes V_\rho$.  
For $\sigma\in\hat{M}$, we denote by $(\Lambda^p\nL^*\otimes V_\rho)_{\sigma}$ 
the $\sigma$-isotypical component of the representation
$\Lambda^p\Ad^*\otimes\rho$
restricted 
to $M$; we have thus a decomposition
\begin{align*}
\Lambda^p\nL^*\otimes V_\rho=\bigoplus_{\sigma\in\hat{M}}(\Lambda^p\nL^*\otimes
V_\rho)_{\sigma}
\end{align*}
of $\Lambda^p\nL^*\otimes V_\rho$, where only finitely many spaces
$(\Lambda^p\nL^*\otimes V_\rho)_{\sigma}$ are not zero.  
Using the preceding computations and Kostant's theorem, one can 
now construct closed and harmonic forms on $F_P$ with values in $E_\rho$.
Let $\rho\in\hat{G}$ and let  $\Lambda(\rho)=\rho_{1}e_{1}+\dots
+\rho_{n+1}e_{n+1}$ be its highest weight as in \eqref{Darstellungen von G}. For
$w\in W_G$ let $l(w)$ denote its 
length with respect to 
the simple roots which define the positive roots $\Delta^+(\gL_\C,\hL_\C)$. Let 
\begin{align}\label{DefW1}
W^{1}:=\{w\in W_{G}\colon w^{-1}\alpha>0\:\forall \alpha\in
\Delta^+(\mathfrak{m}_{\C},\mathfrak{b}_{\C})\}.
\end{align}
As above, let $d=2n+1$, where $d=\dim X$. Then for $k=0,\dots,2n$ let 
$H^{k}(\mathfrak{n},V_{\rho})$ be the cohomology of 
$\mathfrak{n}$ with coefficients in $V_{\rho}$. Then 
$H^{k}(\mathfrak{n},V_{\rho})$ is an $MA$ module. In our case, the 
theorem of Kostant \cite{Kostant} states:
\begin{prop}\label{Prop Kostant}
In the sense of $MA$-modules one has
\begin{align*}
H^{k}(\mathfrak{n};V_{\rho})\cong\sum_{\substack{w\in W^{1}\\ 
l(w)=k}}V_{\rho(w)},
\end{align*}
where $V_{\rho(w)}$ is the $MA$ module of highest weight $w(\Lambda(\rho)
+\rho_{G})-\rho_{G}$.
\end{prop}
\begin{proof}
See for example \cite[Theorem III.3]{BW}, \cite[Theorem 2.5.1.3]{Warner2}.
\end{proof}

For $w\in W^{1}$ let $\sigma_{\rho,w}$ be the representation of $M$ with 
highest weight 
\begin{equation}\label{sigmatauw}
\Lambda(\sigma_{\rho,w}):=w(\Lambda(\rho)
+\rho_{G})|_{\mathfrak{b}_{\mathbb{C}}}-\rho_{M}
\end{equation}
 and let  $\lambda_{\rho,w}\in\mathbb{C}$ such that 
\begin{equation}\label{lambdatauw}
w(\Lambda(\rho)+\rho_{G})|_{\mathfrak{a}_{\mathbb{C}}}=\lambda_{\rho,w}e_{1}.
\end{equation}
We have $\rho_G(H_1)=n$. Thus if for $w\in W_1$, $\ell(w)=k$, we regard
$V_{\rho(w)}$ as 
subset of $H^{k}(\mathfrak{n};V_{\rho})$ via the isomorphism in Proposition
\ref{Prop Kostant}, 
then one has 
\begin{align}\label{Deflambatau}
(\Lambda^k\Ad^*\otimes\rho)(a(t))v=e^{(\lambda_{\rho,w}-n)t}v,\quad \forall v\in
V_{\rho(w)}.
\end{align}

For $k=0,\dots n$ let $\sigma_{\rho,k}$ be the representation of $M$ with
highest weight
\begin{align}\label{sigmatau}
\Lambda(\sigma_{\rho,k}):=(\rho_{1}+1)e_{2}+\dots+(\rho_{k}+1)e_{k+1}
+\rho_{k+2}e_{k+2}+\dots+\rho_{n+1}e_{n+1}.
\end{align}
For $k=0,\dots n$ let
\begin{align}\label{lambdatau}
\lambda_{\rho,k}:=\rho_{k+1}+n-k
\end{align}
For $k=n+1,\dots,2n$ let 
$\lambda_{\rho,k}:=-\lambda_{\rho,2n-k}$
and let $\sigma_{\rho,k}:=w_0\sigma_{\rho,2n-k}$. 
Finally, let
\begin{align*}
\begin{cases}
\lambda_{\rho,n}^{+}:=\lambda_{\rho,n}; \quad \lambda_{\rho,n}^{-}:=-\lambda_{
\rho,n}, \quad \sigma_{\rho,n}^{+}:=\sigma_{\rho,n};\quad
\sigma_{\rho,n}^{-}:=w_0\sigma_{
\rho,n}& \text{if} \:\rho_{n+1}\geq 0\\ 
\lambda_{\rho,n}^{+}:=-\lambda_{\rho,n};\quad \lambda_{\rho,n}^{-}:=\lambda_{
\rho,n}, \quad
\sigma_{\rho,n}^{+}:=w_0\sigma_{\rho,n};\quad\sigma_{\rho,n}^{-}:=\sigma_{
\rho,n}
& \text{if} \:\rho_{n+1}<0.
\end{cases}
\end{align*}
Then by the computations in \cite[Chapter VI.3]{BW} one has
\begin{equation}\label{lambdadecom}
\begin{split}
&\{(\lambda_{\rho,w},\sigma_{\rho,w},l(w))\colon w\in W^{1}\}\\
=&\{(\lambda_{\rho,k},\sigma_{\rho,k},k)\colon
k=0,\dots,n-1\}\sqcup\{(\lambda_{\rho,n}^+,\sigma_{\rho,n}^+,n),
(\lambda_{\rho,n}^-,\sigma_{\rho,n}^{-},n)\}\\
&\sqcup\{(\lambda_{\rho,k},\sigma_{\rho,k},k)\colon k=n+1,\dots,2n\},
\end{split}
\end{equation}
see \cite[section 2.8]{MPIMRN}.
We equip $V_\rho$ with the admissible inner product 
from section \ref{secDR} and we let 
$\mathcal{H}^p(\nL;V_\rho)$ be the space of harmonic forms 
in $\Lambda^p\nL^*\otimes V_\rho$. Then by the finite-dimensional Hodge theorem,
each cohomology 
class in 
$H^p(\nL;V_\rho)$ has a unique harmonic representative in
$\mathcal{H}^p(\nL;V_\rho)$. Moreover,  
$\mathcal{H}^p(\nL;V_\rho)$ is an $MA$-invariant subspace which is
$MA$-equivalent 
to $H^p(\nL;V_\rho)$. 
We put 
\begin{align}\label{DefHpm}
\mathcal{H}^n(\nL;V_\rho)_{\pm}:=\mathcal{H}^n(\nL;V_\rho)\cap(\Lambda^{n}
\nL^*\otimes
V_\rho)_{\sigma_{\rho,n}^\pm}.
\end{align}
Then, if we use the notation \eqref{DefPhilambda}, the following corollary
holds. 
\begin{kor}\label{KorPhilambda}
Let $k\neq n$ and let $\Phi\in \mathcal{H}^k(\nL;V_\rho)$. 
Then the form $\Phi_\lambda$ is closed if and only if
$\lambda=-\lambda_{\rho,k}$. Moreover, for every $t\in\R$ one has
$\iota_{a(t)^*}\Phi_{-\lambda_{\rho,k}}=\Phi$. \\
For $\Phi\in\mathcal{H}^n(\nL;V_\rho)_{\pm}$, the form
$\Phi_\lambda$ is 
closed if and only if $\lambda=-\lambda_{\rho,n}^\pm$. 
Moreover, for every $t\in\R$ one has
$\iota_{a(t)^*}\Phi_{-\lambda_{\rho,n}^\pm}=\Phi$.
\end{kor}
\begin{proof}
This follows immediately from Lemma \ref{Lemd},  Lemma \ref{Pullback} and from
equation \eqref{Deflambatau}.
\end{proof}

For $\Phi\in\mathcal{H}^p(\nL;V_\rho)$, we can also construct a form 
in $C^{\infty}(\Gamma_N\backslash G;\nu_{p+1}(\rho))$ 
as follows. Let 
\begin{align*}
*:\Lambda^p((\nL\oplus\aL)^*)\otimes V_\rho \to
\Lambda^{d-p}((\nL\oplus\aL)^*)\otimes V_\rho
\end{align*}
be the Hodge star operator which is defined as 
the Hodge star operator on $\Lambda^p((\nL\oplus\aL)^*)$ tensored with the
identiy on $V_\rho$ . We let 
$\#:V_\rho\cong V_\rho^*$
be the isomorphism induced by the inner product on $V_\rho$, which extends to an
isomorphism 
\begin{align*}
\#:\Lambda^p((\nL\oplus\aL)^*)\otimes V_\rho \to
\Lambda^{p}((\nL\oplus\aL)^*)\otimes V_\rho^*.
\end{align*}
The operators $\#$ and $*$ are canonically defined on
$C^{\infty}(\Gamma_N\backslash G;\nu_{p}(\rho))$ and
the codifferential 
\begin{align*}
\delta: C^\infty(\Gamma_N\backslash G,\nu_p(\rho))\to
C^\infty(\Gamma_N\backslash G,\nu_{p-1}(\rho)),
\end{align*}
is given by
\begin{align*}
\delta:=(-1)^{dp+d+1}*\circ\#^{-1}d\#\circ *.
\end{align*}
This is the formal adjoint of $d$ on 
$L^2(\Gamma_N\backslash G,\nu_p(\rho)$, where the latter 
space is defined in the same way as the space $C^\infty(\Gamma_N\backslash
G,\nu_p(\rho))$ using 
the induced metric on $F_P$ and the admissible inner product on $V_\rho$. 
We let ${*}_1:\Lambda^p(\nL^*)\otimes V_\rho \to
\Lambda^{2n-p}(\nL^*)\otimes V_\rho$ be the restriction of $*$.
Now for $\Phi\in\Lambda^p\nL^*\otimes V_\rho$ we 
define an $E_\rho$-valued $p+1$-form 
$\tilde{\Phi}_\lambda\in C^\infty(\Gamma_N\backslash G,\nu_{p+1}(\rho))$ 
by 
\begin{align}\label{Phitilde}
\tilde{\Phi}_\lambda(na(t)k):=e^{(\lambda+(d-1)/2)
t}\nu_p(\rho)(k^{-1})(dH_1\wedge\Phi),
\end{align}
where $n\in N$, $a(t)\in A$, $k\in K$.

\begin{lem}\label{HarmForm}
For $\Phi\in \mathcal{H}^k(\nL;V_\rho)$, $k\neq n$ and $\lambda\in\C$ the form 
$\tilde{\Phi}_\lambda$
is closed. It is coclosed if and only if $\lambda=\lambda_{\rho,k}$. 
For $\Phi\in \mathcal{H}^n(\nL;V_\rho)_{\pm}$ and $\lambda\in\C$ the form 
$\tilde{\Phi}_\lambda$
is closed. It is coclosed if and only if $\lambda=\pm \lambda_{\rho,n}^{\pm}$. 
Finally, one has $\iota_a^*\tilde{\Phi}_\lambda=0$ for each $a\in A$.
\end{lem}
\begin{proof}
Since $\Phi$ is closed, the form $\tilde{\Phi}_\lambda$ is obviously closed for
each $\lambda$. In the notation \eqref{DefPhilambda} we have 
$\#\circ *\tilde{\Phi}_\lambda=(\#\circ*_1\Phi)_\lambda$
and since $\Phi\in\mathcal{H}^k(\nL;V_\rho)$, one has
$\#*_1\Phi\in\mathcal{H}^{2n-k}(\nL;V_{\rho^*})$.   
Since the inner product on $V_\rho$ is admissible, for $v\in V_\rho$ 
and $a(t)\in A$ with $\rho(a(t))v=e^{\lambda t}v$ one has 
$\rho^*(a(t))(\# v)=e^{-\lambda t}\# v$. 
On the other hand, $\Lambda^p\Ad^*(a(t))$ acts 
as $e^{-pt}$ on $\Lambda^p\nL^*$ for every $t\in\R$. 
Since $A$ is abelian, it follows that 
if for $\Phi\in\Lambda^p\nL^*\otimes V_\rho$ 
there exists $\lambda\in\R$ with 
\begin{align*}
(\Lambda^p\Ad^*\otimes\rho)(a(t))\Phi=e^{\lambda t}\Phi,\quad t\in\R
\end{align*}
then 
\begin{align*}
(\Lambda^{2n-p}\Ad\otimes\rho^*)(a(t))\left(\#\circ *_1\Phi\right) =
e^{(-2n-\lambda)
t}\left(\#\circ *_1\Phi\right),\quad t\in\R.
\end{align*}
Thus the first statements of the Lemma follow from Lemma \ref{Lemd} and equation
\eqref{Deflambatau}. 
Since $\iota_a^*dH_1=0$, the last statement is obvious. 
\end{proof}

For $Y\in(0,\infty)$ we let $F_P(Y)$ be as in \eqref{Definition der Spitze}.
We denote the restriction of the bundle $E_\rho$ from section \ref{secDR} to 
$F_P(Y)$ by $E_\rho$ too. Then the space of 
smooth $E_\rho$-valued $p$-forms on $F_P(Y)$
is canonically isomorphic to the spaces $C^\infty(\Gamma_N\backslash
NA[Y]K,\nu_p(\rho))$ which is defined analogously to the spaces above. 
We let $L^2(\Gamma_N\backslash NA[Y]K,\nu_p(\rho))$ be the space 
of all measurable functions $f:\Gamma_N\backslash NA[Y]K\to
\Lambda^p(\nL\oplus\aL)^*\otimes V_\rho$ 
which  satisfy
\begin{align*}
f(\Gamma_N nak)=\nu_p(\rho)(k)^{-1}f(na) 
\end{align*}
for allmost all $\Gamma_Nnak\in\Gamma_N\backslash NA[Y]K$, $a\in A[Y]$, $n\in
N$, $k\in K$ and whose norm 
\begin{align*}
\left\|f\right\|_{L^2(F_P(Y);\Lambda^p E_\rho)}^2=\int_{\log{Y}}^\infty\int_{
\Gamma_N\backslash
N}e^{-(d-1)t}\left\|f(\bar{n}a(t))\right\|^2d_{\Gamma_N\backslash
N}(\bar{n})dt 
\end{align*}
is finite. Here the norm in
$\Lambda^p(\nL\oplus\aL)^*\otimes V_\rho$ 
is the norm induced by the metric on $\nL\oplus\aL\subset\gL$ and 
the admissible inner product on $V_\rho$. It follows that for
$\Phi\in\Lambda^p\nL^*\otimes V_\rho$ and $\lambda\in\R$, $\lambda<0$, 
the form $\tilde{\Phi}_\lambda$, restricted to $F_P(Y)$, where it defines 
an $E_\rho$-valued $p$-forms,  is square-integrable. 
More precisely, if $\Phi,\Psi\in\Lambda^p\nL^*\otimes V_\rho$ and if
$\lambda<0$, one has 
\begin{align}\label{InProd}
\left<\tilde{\Phi}_\lambda,\tilde{\Psi}_\lambda\right>_{
L^2(F_P(Y);\Lambda^pE_\rho)}=&\left<\Phi,\Psi\right>\vol(\Gamma_N\backslash
N)\int_{\log Y}^\infty
e^{2t\lambda}dt\nonumber\\ =&\left<\Phi,\Psi\right>\vol(\Gamma_N\backslash
N)\frac{Y^{2\lambda}}{2|\lambda|}
\end{align}

If $\rho\neq\rho_\theta$, then $\lambda_{p,\rho}<0$ for $p>n$ and
$\lambda_{n,\rho}^-<0$. 
Thus using Lemma \ref{HarmForm} and the preceding equation, for $p\geq n+1$ one
can define an orthonormal set of harmonic $E_\rho$-valued 
$p$-forms on $F_P(Y)$ which satisfy relative boundary conditions and which are
square-integrable. It will be shown in Lemma \ref{Propccusp} 
that one can obtains an orthornormal basis of the associated cohomology in this
way.

\section{Eisenstein series}\label{SecER}
\setcounter{equation}{0}

We let $X=\Gamma\backslash\widetilde{X}$ be the hyperbolic manifold 
of section \ref{subsmfld}, we let $\rho\in\Rep(G)$ and we let $E_\rho$ be the
flat 
vector bundle over $X$ as in section \ref{secDR}. In Corollary
\ref{KorPhilambda}, we have constructed closed
$E_\rho$ valued $p$-forms on the cusps of the manifold. 
The next step in Harder's approach \cite{Ha} now is to lift these forms to
$E_\rho$ valued $p$-forms on $X$ by 
averaging them over the cosets $\Gamma\cap N_{P_i}\backslash\Gamma$
corresponding 
to the cusps of $X$. Here one uses the theory of
Eisenstein series to treat the analytic properties of the obtained infinite
sums. 
Therefore, in this section we shall briefly review the theory of Eisenstein 
series. 

For $\sigma\in\hat{M}$, $P_i\in\mathfrak{P}_\Gamma$ we
define a representation 
$\sigma_{P_i}$ of $M_{P_i}$ by
\begin{equation}\label{repsigma}
\sigma_{P_i}(m_{P_i}):=\sigma(k_{P_i}^{-1}m_{P_i}k_{P_i}),\quad m_{P_i}\in
M_{P_i}.
\end{equation} 
Let $\nu\in\hat{K}$ and $\sigma\in\hat{M}$ such that
$\left[\nu:\sigma\right]\neq 0$. 
Then we let $\mathcal{E}_{P_i}(\sigma,\nu)$ be the set of all
continuous functions $\Phi$ on $G$ which are left-invariant under
$N_{P_i}A_{P_i}$ such
that for all
$x\in G$ the function
$m_{P_i}\mapsto \Phi_{P_i}(m_{P_i}x)$ belongs to $L^2(M_{P_i},\sigma_{P_i})$,
the
$\sigma_{P_i}$-isotypical component of the
right regular representation of $M_{P_i}$, and such that 
for all $x\in G$ the function $k\mapsto \Phi_{P_i}(xk)$ belongs to the
$\nu$-isotypical component of the right regular representation of $K$. 
The space $\mathcal{E}_{P_i}(\sigma,\nu)$ is finite-dimensional and 
one can show that 
$\dim(\mathcal{E}_{P_i}(\sigma,\nu))=\dim(\sigma)\dim(\nu)$.

We define an inner product on $\mathcal{E}_{P_i}(\sigma,\nu)$ as follows.
Any element of $\mathcal{E}_{P_i}(\sigma,\nu)$ can be identified
canonically with a function on $K$. For
$\Phi,\Psi\in\mathcal{E}_{P_i}(\sigma,\nu)$ put
\begin{align}\label{InprodaufE}
\left<\Phi,\Psi\right>:=\vol(\Gamma\cap N_{P_i}\backslash
N_{P_i})\int_K\Phi(k)\bar{\Psi}(k)dk.
\end{align}
Define the Hilbert space $\mathcal{E}_{P_i}(\sigma)$ by
\begin{align*}
\mathcal{E}_{P_i}(\sigma):=\bigoplus_{\substack{\nu\in\hat{K}\\\left[
\nu:\sigma\right]\neq 0}}\mathcal{E}_{P_i}(\sigma,\nu).
\end{align*}
Let 
\[
\boldsymbol{\mathcal{E}}(\sigma,\nu):=\bigoplus_{P_i\in\mathfrak{P}_\Gamma
}\mathcal{E}_{P_i}(\sigma,\nu);\quad
\boldsymbol{\mathcal{E}}(\sigma):=\bigoplus_{P_i\in\mathfrak{P}_\Gamma
}\mathcal{E}_{P_i}(\sigma).
\]

For $\sigma\in\hat{M}$, $P_i\in\mathfrak{P}_\Gamma$ and $\lambda\in\C$ we let
$\mathcal{E}_{P_i}(\sigma,\lambda)$ 
be the $L^2(K)$-closure of all continuous functions $\Psi$ on
$G$ which satisfy 
\begin{align*}
\Psi(n_{P_i}a_{P_i}(t)x) =  e^{(\lambda+n)t}\Psi(x),\quad
t\in\R,\:n_{P_i}\in N_{P_i}, x\in G
\end{align*}
and for which the functions $m_{P_i}\mapsto \Psi(m_{P_i}x)$, $x\in G$, belong 
to $L^2(M_{P_i};\sigma_{P_i})$. 
We let 
\begin{align*}
\boldsymbol{\mathcal{E}}(\sigma,\lambda):=\bigoplus_{i=1}^{\kappa(\Gamma)}
\mathcal {E}_{P_i}(\sigma,\lambda)
\end{align*}
and we let $G$ act on $\mathcal{E}_{P_i}(\sigma,\lambda)$ resp.
$\boldsymbol{\mathcal{E}}(\sigma,\lambda)$  by 
the right regular representation. 
For $\Phi\in\mathcal{E}_{P_i}(\sigma)$ we define 
a function $\Phi_\lambda$ on $G$ by
$\Phi_\lambda(n_{P_i}a_{P_i}(t)k):=e^{(n+\lambda)t}\Phi(k)$, 
$n_{P_i}\in N_{P_i}$, $a_{P_i}\in A_{P_i}$, $k\in K$.
Then the assignment 
$\Phi\mapsto \Phi_\lambda$, $\Phi\in\mathcal{E}_{P_i}(\sigma)$, is  
an isomorphism between $\mathcal{E}_{P_i}(\sigma)$ 
and $\mathcal{E}_{P_i}(\sigma,\lambda)$. Finally, 
$G$ acts on $\boldsymbol{\mathcal{E}}(\sigma,\lambda)$ by the 
right regular representation and it is easy to see 
that $\boldsymbol{\mathcal{E}}(\sigma,\lambda)$  with 
this $G$ action 
is equivalent to $\kappa(\Gamma)\dim(\sigma)$ copies of the principal 
series representation $\pi_{\sigma,-i\lambda}$, where the latter is parametrized
as in \eqref{subsps}, see 
for example \cite[section 4]{Pf1}.
\newline

Now we let $\rho\in\Rep(G)$ be a finite-dimensional
irreducible representation of $G$ on $V_{\rho}$. 
Let $P_i\in\mathfrak{P}_\Gamma$. Then the results of 
the previous sections \ref{SecDRcusp}, \ref{liftd} and \ref{secclcusp}  carry
over to the parabolic subgroup $P_i$, where  
we use the fixed $k_{P_i}$-s in $K$ to identify 
each $P_i=k_{P_i}Pk_{P_i}^{-1}$ with $P$. In particular, for $\sigma\in\hat{M}$,
we shall 
denote by $(\Lambda^p\nL_{P_i}^*\otimes V_\rho)_\sigma$
the $\sigma_{P_i}$-isotypical component of the representation
$\Lambda^p\Ad^*\otimes\rho$ 
of $M_{P_i}$ on $(\Lambda^p\nL_{P_i}^*\otimes V_\rho)$. 
For $\Phi_{P_i}\in\Lambda^p\nL_{P_i}^*\otimes V_\rho$
we define the element $\Phi_{P_i,\lambda}:=(\Phi_{P_i})_\lambda\in
C^\infty((\Gamma\cap
N_{P_i})\backslash G,\nu_p(\rho))$ as in \eqref{DefPhilambda}.
The complex $C^*(\gL,K;\mathcal{E}_{P_i}(\check{\sigma},
\lambda)\otimes V_\rho)$ is a subcomplex of $C^*(\gL,K;C^\infty((\Gamma\cap
N_{P_i})\backslash G)\otimes V_\rho)$. Moreover, we 
have the following Lemma. 
\begin{lem}\label{LemgK}
Let $\Phi_{P_i}\in(\Lambda^p\nL_{P_i}^*\otimes V_\rho)_\sigma$. 
Then with respect to the isomorphism \eqref{IsomKompl}, one has 
\begin{align*}
\Phi_{P_i,\lambda}\in
C^p(\gL,K;\mathcal{E}_{P_i}(\check{\sigma},
\lambda)\otimes V_\rho),
\end{align*}
where $\check{\sigma}$ is the contragredient representation of $\sigma$.
\end{lem}
\begin{proof}
For $X\in\Lambda^p(\gL/\kL)$ and $g\in G$, $n_{P_i}\in N_{P_i}$, $a_{P_i}(t)\in
A_{P_i}$  the function $\Phi_{P_i,\lambda;X}$,
defined 
as in the end of section \ref{SecDRcusp}
obviously 
satisfies 
\begin{align*}
\Phi_{P_i,\lambda;X}(n_{P_i}a_{P_i}(t)g)=e^{(\lambda+n)t}\Phi_{P_i,\lambda;X}
(g). 
\end{align*}
To show that the function $m_{P_i}\mapsto \Phi_{P_i,\lambda;X}(m_{P_i}g)$
belongs 
to $L^2(M_{P_i};\check{\sigma}_{P_i})\otimes V_\rho$, one 
just uses the Schur orthogonality relations \cite[Corollary 4.10]{Knapp2}.
\end{proof}

For $\Phi_{P_i}\in(\Lambda^p\nL_{P_i}^*\otimes V_\rho)_\sigma$ one defines the
Eisenstein series
$E(\Phi_{P_i}:\lambda:x)\in C^\infty(\Gamma\backslash G,\nu_p(\rho))$ 
by 
\begin{align}\label{Def ER}
E(\Phi_{P_i}:\lambda)(x):=\sum_{\gamma\in(\Gamma\cap N_{P_i})\backslash\Gamma}
\Phi_{P_i,\lambda}(\gamma g);\quad x=\Gamma g.
\end{align}
By Lemma \ref{LemgK}, under the 
isomorphism \eqref{IsomKompl}, $\Phi_{P_i,\lambda}$ can 
be written as a finite sum
\begin{align*}
\Phi_{P_i,\lambda}=\sum_j \omega_j\otimes\Psi^j_{P_i,\lambda}\otimes v_j
\end{align*}
where $\omega_j\in\Lambda^p(\gL/\kL)^*$,
$\Psi^j_{P_i}\in\mathcal{E}_{P_i}(\check{\sigma})$ resp.
$\Psi^j_{P_i,\lambda}\in\mathcal{E}_{P_i}(\check{\sigma},\lambda)$, and $v_j\in
V_\rho$. 
Thus the main results about the ``usual`` scalar valued Eisenstein series, due
to Selberg \cite{Se}, 
Langlands \cite{Langlands}, Harish-Chandra \cite{Harish-Chandra} and others, 
continue to hold for the vector-valued Eisenstein series defined in \eqref{Def
ER} 
with the obvious modifications. 

We shall 
now briefly review them. 
On $\Gamma\backslash G\times\{\lambda\in\mathfrak\C\colon
\Real(\lambda)>n\}$ the series \eqref{Def ER} is absolutely and locally
uniformly
convergent. As a function of $\lambda$, it has a meromorphic continuation to
$\C$ with only finitely many poles in the strip $0<\Real(\lambda)\leq n$
which are located on $(0,n]$ and it has no poles on the line
$\Real(\lambda)=0$.

For $P_i,P_j\in\mathfrak{P}_{\Gamma}$,
$\Phi_{P_i}\in(\Lambda^p\nL_{P_i}^*\otimes V_\rho)_\sigma$ and
$g\in G$ let
\begin{align*}
E_{P_j}(\Phi_{P_i}:\lambda)(g):=\frac{1}{\vol\left(\Gamma\cap N_{P_j}\backslash
N_{P_j}\right)}\int_{\Gamma\cap N_{P_j}\backslash
N_{P_j}}{E(\Phi_{P_i}:\lambda)(ng)dn}
\end{align*}
be the constant term of $E(\Phi_{P_i}:\lambda)$ along $P_j$. 
Then there exists a meromorphic function
\begin{align*}
{\underline{C}}_{P_i|P_j}(\sigma:\lambda):
C^*(\gL,K;\mathcal{E}_{P_i}(\sigma:\lambda)\otimes V_\rho)\to
C^*(\gL,K;\mathcal{E}_{P_j}(w_0\sigma:-\lambda)\otimes V_\rho)
\end{align*}
such that for
$P_i,P_j\in\mathfrak{P}_\Gamma$ one has
\begin{align}\label{C1}
E_{P_j}(\Phi_{P_i}:\lambda)=\delta_{i,j}\Phi_{P_i,\lambda}+{\underline{C}}
_{P_i|P_j}
(\sigma:\lambda)\Phi_{P_i,\lambda}.
\end{align}
Furthermore, let 
\begin{align*}
\underline{\mathbf{C}}(\sigma,\lambda):\bigoplus_{P_i\in\mathfrak{P}_\Gamma}
C^*(\gL,K;\mathcal{E}_{P_i}(\sigma:\lambda)\otimes V_\rho)
\to\bigoplus_{P_i\in\mathfrak{P}_\Gamma}
C^*(\gL,K;\mathcal{E}_{P_i}(w_0\sigma:-\lambda)\otimes V_\rho)
\end{align*}
be the map built from the maps $\underline{C}_{P_i|P_j}(\sigma,\lambda)$. Then
one has
\begin{align}\label{FG}
\underline{\mathbf{C}}(w_0\sigma,\lambda)\underline{\mathbf{C}}(\sigma,
-\lambda)=\Id;\quad\underline{\mathbf{C}}
(\sigma,\lambda)^{*}=\underline{\mathbf{C}}
(w_0\sigma,\bar{\lambda}).
\end{align}

Let $P_i,P_j\in\mathfrak{P}_\Gamma$ and let
$\Phi_{P_i}\in(\Lambda^p\nL_{P_i}^*\otimes V_\rho)_{\sigma}$. Then for
$\underline{C}_{P_i|P_j}(\sigma:\lambda)\Phi_{P_i,\lambda}$, regarded 
as an element of $C^\infty(\Gamma\cap N_{P_j}\backslash G,\nu_p(\rho))$,
and for $a\in A_{P_j}$  
one has 
\begin{align}\label{PropMtyp}
\iota_a^*\bigl(\underline{C}_{P_i|P_j}(\sigma:\lambda)\Phi_{P_i,\lambda}
\bigr)\in
(\Lambda^p\nL_{P_j}^*\otimes V_\rho)_{w_0\sigma}, 
\end{align}
where $\iota_a^*$ is as in section \ref{liftd}. 
This follows easily from the usual properties of the constant term matrix and
the above 
constructions.

\begin{bmrk}
The maps $\underline{C}_{P_i|P_j}(\sigma,\lambda)$ are of course  
constructed out of the usual constant-term matrices $C_{P_i|P_j}$ of Eisenstein 
series, used in \cite{MP1}, \cite{MP2}. However, by the 
isomorphism from Lemma \ref{LemgK}, $\underline{C}_{P_i|P_j}(\sigma,\lambda)$
is 
in fact constructed out of the matrix $C_{P_i|P_j}(\check{\sigma},\lambda)$ 
used in these articles. 
\end{bmrk}

\section{Bases of cohomology classes on certain flat bundles using
Eisenstein series}\label{secbases}
\setcounter{equation}{0}
We can now now give explicit bases of the 
De Rham cohomology groups 
$H^*(X;E_\rho)$
of $X$ with coefficients 
in the flat vector bundles $E_\rho$ 
for those $\rho\in\Rep(G)$ which satisfy $\rho\neq\rho_\theta$ using the
approach 
of Harder and Schwermer. 

We let 
$X$ be as in section \ref{subsmfld}.
and we let $Y\geq Y_0$, where $Y_0$ is as in section \ref{subsmfld}. Then we 
can restrict $E_\rho$ to $X(Y)$ and to the boundary $\partial X(Y)$ of 
$X(Y)$. We shall denote the corresponding De Rham (or equivalently singular)
cohomology
groups by $H^*(X(Y);E_\rho)$ and $H^*(\partial X(Y);E_\rho)$. By 
$H^*(X(Y),\partial X(Y);E_\rho)$ we 
denote the relative cohomology groups corresponding to 
the pair of spaces $X(Y)$ and $\partial X(Y)$ and the bundle $E_\rho$. 
We have a long exact sequence
\begin{align*}
&\dots\to H^p(X(Y),\partial X(Y);E_\rho)\to H^p(X(Y);E_\rho)\to H^p(\partial
X(Y);E_\rho)\\ &\to
H^{p+1}(X(Y);E_\rho)\to\dots.
\end{align*}
The cohomology groups $H^p(X(Y);E_\rho)$ are of course independent of 
$Y$ and are just equal to the De Rham cohomology groups $H^p(X;E_\rho)$ or can
be written as $H^p(\overline{X};E_\rho)$, 
where $\overline{X}$ denotes the Borel-Serre compactification of $X$. However,
since later on we shall work with the truncated manifolds $X(Y)$, we shall use 
them here too. 
Let $H_{!}^p(X(Y);E_\rho)$ denote the image of $H^p(X(Y),\partial X(Y);E_\rho)$
in $H^p(X(Y);E_\rho)$ in the long exact sequence above. 

Let $H^p_{(2)}(X;E_\rho)$ denote the $L^2$-cohomology groups of $X$ with
coefficients in $E_\rho$. By
definition, these are the cohomology groups of the subcomplex of the
$E_\rho$-valued De Rham complex on $X$ which is formed by the smooth 
$E_\rho$-valued differential forms on $X$ that are square integrable and whose 
exterior derivative is also square integrable. 
Then the following Proposition holds.
\begin{prop}\label{LemL2cohomology}
Let $\rho\in\Rep(G)$, $\rho\neq\rho_\theta$. Then one has
$H^*_{(2)}(X;E_\rho)=0$.
\end{prop}
\begin{proof}
The group $G$ acts on $L^2(\Gamma\backslash G)$ by the right regular
representation $\pi_\Gamma$ and one denotes by $L^2(\Gamma\backslash G)^\infty$ 
the corresponding smooth $K$-finite vectors. Then by a result of Borel 
\cite[Theorem 3.5]{Bo} one has $H^*_{(2)}(X;E_\rho)=H^*(\gL,K;V_\rho\otimes
L^2(\Gamma\backslash G)^\infty)$. 
According to the spectral decomposition of $\pi_\Gamma$, this cohomology 
splits into a finite direct sum of $(\gL,K)$-cohomology spaces with coefficients
either in irreducible unitary representations 
of $G$ or in direct integrals of tempered principal series representations of
$G$, \cite[Theorem 5.3]{BG}. 
Using the condition $\rho\neq\rho_\theta$, it follows that the cohomology of the
first type is zero by 
\cite[Proposition II. 6.12]{BW} and that the cohomology of the second type is
zero by
\cite[Theorem 3.6]{BC}.
\end{proof}

\begin{bmrk}\label{BemerkMP}
Proposition \ref{LemL2cohomology} can also be proved as follows. 
Let $\Delta_p(\rho)$ be the flat Hodge-Laplace operator
acting on the smooth compactly supported $E_\rho$-valued 
$p$-forms on $X$. Since $X$ is complete, this operator is essentially self
adjoint \cite{Ch} and 
its closure will be denoted by the same symbol.  
It folllows from the results of \cite[section 7]{MP1},
\cite[section 5]{MP2} that the spectrum
$\sigma\left(\Delta_p(\rho)\right)$ of $\Delta_p(\rho)$ 
satisfies $\sigma\left(\Delta_p(\rho)\right)\subseteq [\frac{1}{4},\infty)$. 
More precisely, by the theory of 
Eisenstein series applied in \cite{MP1}, \cite{MP2} to the present case, the
continuous spectrum of $\Delta_p(\rho)$ 
consists of the intervals $[\rho(\Omega)-c(\sigma),\infty)$, each 
occuring with multiplicity $\kappa(\Gamma)$, where
$\sigma\in\hat{M}$ is such 
that $[\nu_p(\rho):\sigma]\neq 0$, where $c(\sigma)$ is as in \eqref{csigma} 
and where $\rho(\Omega)$ is the Casimir eigenvalue of $\Omega$. 
Thus by \cite[Lemma 7.1]{MP1} and the condition $\rho\neq\rho_\theta$, the
continuous 
spectrum of $\Delta_p(\rho)$ lies in $[1/4,\infty)$. The same holds 
for the discrete spectrum by \cite[Lemma 7.3]{MP1}, \cite[Corollary 5.2]{MP2}. 
Since $\Delta_p(\rho)$ is essentially selfadjoint, by \cite[Lemma 3.8]{BL} the
complex formed 
by the smooth $E_\rho$-valued differential forms on $X$ with compact support has
a unique ideal boundary condition. By \cite[Theorem 3.5]{BL} the cohomology of
this Hilbert complex
coincides with the cohomology 
$H^*_{(2)}(X;E_\rho)$. Since the $\Delta_p(\rho)$ are the 
Laplacians of this Hilbert complex, Proposition \ref{LemL2cohomology}
follows therefore from \cite[Theorem 2.4, Corollary 2.5]{BL}. 
\end{bmrk}

By the next Lemma
the De Rham cohomlogy 
of $X$ with coefficients in $E_\rho$ is determined by its restriction 
to the boundary if $\rho\neq\rho_\theta$.

\begin{lem}\label{Lemshr}
Let $\rho\in\hat{G}$ and assume that $\rho\neq\rho_\theta$. Then one 
has $H_{!}^p(X(Y);E_\rho)=0$ for each $p$ or, in other words, 
the restriction map $H^p(X(Y);E_\rho)\to H^p(\partial X(Y);E_\rho)$ is injective
for each $p$.
\end{lem}
\begin{proof}
Let $H_{c}^*(X;E_\rho)$ denote the compactly supported De Rham cohomology of $X$
with coefficients 
in $E_\rho$ and let $H_{!}^*(X;E_\rho)$ be the interior cohomology, i.e. the
image
of $H_{c}^*(X;E_\rho)$ in
$H^*(X;E_\rho)$. Then we can identify $H_{!}^*(X(Y);E_\rho)$ with
$H_{!}^*(X;E_\rho)$. Moreover, Proposition \ref{LemL2cohomology} 
implies that $H_{!}^*(X;E_\rho)=0$. 
\end{proof}

For $P_i\in\mathfrak{P}_\Gamma$, we define 
the spaces $\mathcal{H}^k(\nL_{P_i};V_\rho)$, 
$\mathcal{H}^n(\nL_{P_i};V_\rho)_{\pm}$ as in section 
\ref{secclcusp}. We also let the $\lambda_{\rho,k}\in\R$ be as in
\eqref{lambdadecom}.

\begin{prop}\label{PropER}
Let $\rho\in\hat{G}$ and assume that $\rho\neq\rho_\theta$. For $k> n$,
$P_i\in\mathfrak{P}_\Gamma$
and $\Phi_{P_i}\in\mathcal{H}^k(\nL_{P_i};V_\rho)$ the
Eisenstein series 
$E(\Phi_{P_i}:\lambda)$ is holomorphic in $\lambda=-\lambda_{\rho,k}$. Moreover,
$E(\Phi_{P_i}:-\lambda_{\rho,k})$ is closed. 
For $\Phi_{P_i}\in\mathcal{H}^n(\nL_{P_i};V_\rho)_-$, the
Eisenstein series 
$E(\Phi_{P_i}:\lambda)$ is holomorphic in $\lambda=-\lambda_{\rho,n}^-$ and
$E(\Phi_{P_i}:-\lambda_{\rho,n}^-)$ is closed. 
\end{prop}
\begin{proof}
We can assume without loss of generality that $\nL_{P_i}=\nL$. If $w_0$ is the
non-trivial 
element of the restricted Weyl group $W(A)$, then for 
$\rho\neq\rho_\theta$ one has $\sigma_{\rho,w}\neq w_0\sigma_{\rho,w}$ 
for each $w\in W^1$ by \eqref{Tau theta}, \eqref{wsigma} and
\eqref{lambdadecom}. 
Thus by \cite[Propositions 49, 53]{Knapp Stein}, there exists no complementary 
series corresponding 
to $\sigma_{\rho,w}$, i.e. in our parametrization  no 
representation $\pi_{\sigma_{\rho,w},i\lambda}$, $\lambda\in\R$, is
unitarizable. 
As it is well known, this implies that
the Eisenstein-Series $E(\Phi:\lambda)$ is holomorphic on $\{\lambda\in\C\colon
\Real(\lambda)>0\}$: 
If $\lambda_0$ was a  pole of $E(\Phi:\lambda)$ lying in this set, it would be a
simple 
pole with
$\lambda_0\in(0,n]$ and by the spectral decomposition of the right regular
representation $\pi_\Gamma$ of $G$ on $L^2(\Gamma\backslash G)$,
\cite{Langlands}, \cite{Warner},  
the corresponding residue of the Eisenstein series would 
establish $\pi_{\sigma_{\rho,w},-i\lambda_0}$ as a direct summand of
$\pi_\Gamma$, which 
contradicts the absence of complementary series. 
If $\rho\neq \rho_\theta$, then by \eqref{Tau theta} and \eqref{lambdadecom},
for $k>n$ one has $\lambda_{\rho,k}<0$
and one has $\lambda_{\rho,n}^-<0$. Thus the statement about the holomorphy of
the corresponding Eisenstein series follows. 

Assume that $k>n$. If $\Phi\in\mathcal{H}^k(\nL;V_\rho)$, then by
\eqref{Deflambatau} one has 
$(\Lambda^k\Ad^*\otimes\rho)(H_1)\Phi=(
\lambda_{\rho,k}-n)\Phi$
and thus by Lemma \ref{Lemd} the 
form $\Phi_\lambda$ is closed at $\lambda=-\lambda_{\rho,k}$. This 
implies
that
also $E(\Phi:-\lambda_{\rho,k})$ is closed, \cite[4.3(4)]{Sc}. For 
$k=n$ one argues in the same way.
\end{proof}

In the next three propositions, we can 
now describe the cohomology groups $H^*(X;E_\rho)$ completely 
and we can also give explicity canonical bases. We would 
like to emphasize again that these propositions are 
essentially elaborations of the ideas of Harder \cite{Ha} and Schwermer
\cite{Sc}. 
For $Y\geq Y_0$ we shall use the identification 
\begin{align}\label{Isomkoho}
H^*(\partial X(Y);E_\rho)\cong \bigoplus_{P_i\in\mathfrak{P}_\Gamma}
H^k(\Gamma\cap N_{P_i}\backslash N_{P_i};E_\rho)\cong
\bigoplus_{P_i\in\mathfrak{P}_\Gamma}
\mathcal{H}^k(\nL_{P_i};V_\rho),
\end{align}
where the second isomorphism, induced by the map \eqref{mapPhi} is the
isomorphism of van Est \cite{van Est}.
Moreover, by $\iota_k^*(Y)$ we shall denote the restriction map 
\begin{align*}
\iota_k^*(Y):H^k(X(Y);E_\rho) \to H^k(\partial X(Y);E_\rho) 
\end{align*}
in cohomology. To ease 
notation, we shall denote the restriction $\omega|_{X(Y)}$  of an
$E_\rho$-valued 
differential form $\omega$ on $X$ to an $E_\rho$-valued
differential form on $X(Y)$ by the same symbol $\omega$. 

\begin{prop}\label{Propkoho1}
Let $\rho\in\Rep(G)$ and assume that $\rho\neq\rho_\theta$. 
Let $n<k\leq 2n$ and let $Y\geq Y_0$. Then $\iota_k^*(Y)$ is 
an isomorphism. More explicitly, for
$\Phi\in\bigoplus_{P_i\in\mathfrak{P}_\Gamma}
\mathcal{H}^k(\nL_{P_i};V_\rho)$ the Eisenstein series
$E(\Phi:-\lambda_{\rho,k})$ is a non-trivial 
cohomology class with 
$\iota_k^*(Y)E(\Phi:-\lambda_{\rho,k})=\Phi$.
\end{prop}
\begin{proof}
We can assume that there exists a $P_i\in\mathfrak{P}_\Gamma$ with 
$\Phi\in\mathcal{H}^k(\nL_{P_i};V_\rho)$. According to the previous section we 
shall write $\Phi=\Phi_{P_i}$.
By Proposition \ref{PropER}, the Eisenstein series $E(\Phi_{P_i}:\lambda)$
has no pole in $\lambda=-\lambda_{\rho,k}$ and 
$E(\Phi_{P_i}:-\lambda_{\rho,k})$ is closed. Thus by \cite[page 152]{Ha},
\cite[Satz 1.10]{Sc}
also the constant term $E_{P_j}(\Phi_{P_i}:-\lambda_{ \rho,k})$ along $P_j$
of $E(\Phi_{P_i}:-\lambda_{ \rho,k})$ is closed and 
the restriction of $E(\Phi_{P_i}:-\lambda_{ \rho,k})$ to the boundary component
associated to $P_j$ in 
the decomposition \eqref{Zerlegung X}  is 
cohomologous to the restriction of $E_{P_j}(\Phi_{
P_i}:-\lambda_{ \rho,k})$ to the boundary component associated to $P_j$ .
By \eqref{C1} one has 
\begin{align}\label{sum}
E_{P_j}(\Phi_{P_i}:-\lambda_{
\rho,k})=\delta_{i,j}\Phi_{P_i,-\lambda_{\rho,k}}+\underline{C}_{
P_i|P_j}(\sigma_{\rho,
k},-\lambda_{\rho,k})\Phi_{P_i,-\lambda_{\rho,k}}.
\end{align}
Since $k>n$ and since $\Phi_{P_i}$ is harmonic, one has $\Phi_{P_i}\in
(\Lambda^k\nL_{P_i}^*\otimes V_\rho)_{\sigma_{\rho,k}}$ by \eqref{lambdadecom}.
Thus by \eqref{PropMtyp}, for $a\in A_{P_j}$ one has 
\begin{align*}
\iota_a^*\left(\underline{C}_{
P_i|P_j}(\sigma_{\rho,
k},-\lambda_{\rho,k})\Phi_{P_i,-\lambda_{\rho,k}}\right)\in
(\Lambda^k\nL_{P_j}^*\otimes V_\rho)_{w_0\sigma_{\rho,k}}.
\end{align*}
However, as in the proof of Proposition \ref{PropER}, by 
the condition $\rho\neq\rho_\theta$ one has $\sigma_{\rho,k}\neq w_0
\sigma_{\rho,k}$. Thus, 
since $k>n$ one has $\mathcal{H}^k(\nL_{P_j}^*;V_\rho)\cap
(\Lambda^p\nL_{P_j}^*\otimes V_\rho)_{w_0\sigma_{\rho,k}}=0$ by
\eqref{lambdadecom}. 
Applying Corollary \eqref{KorPhilambda}, one obtains
$\iota_k^*(Y)E(\Phi_{P_i}:-\lambda_{\rho,k})=\Phi_{P_i}$
and the proposition follows from Lemma \ref{Lemshr}. 
\end{proof}

Next we prove the vanishing of the $E_\rho$-valued De Rham cohomology
in degree $k<n$. 

\begin{prop}\label{Propkoho2}
Let $\rho\in\Rep(G)$ and assume that $\rho\neq\rho_\theta$. Then for $k<n$, one
has $H^k(X(Y);E_\rho)=0$. 
\end{prop}
\begin{proof}
By Lemma \ref{Lemshr}, it suffices to show that for $k<n$ the map 
$\iota_k^*(Y)$ is zero. We let $E_{\rho}^*$ be the 
dual bundle of $E_\rho$. As in the case of trivial coefficients, we have a
canonical pairing
\begin{align*}
B: H^k(\partial X(Y);E_\rho)\otimes H^{2n-k}(\partial
X(Y);E_{\rho}^*)\to\C,\quad
B(\omega,\eta):=\int_{\partial X(Y)}(\omega\wedge\eta)(x)dx,
\end{align*}
where $\omega\wedge\eta$ is the ordinary $2n$-form on $\partial X(Y)$ defined as
in
\cite[page 729]{Mutorsunim}. 
Also, as in the case of trivial coefficients, the classical Hodge theorem
implies that the pairing $B$ is
non-degenerate. 
By Stoke's theorem, for $\omega\in H^k(X(Y);E_\rho)$ and $\eta\in
H^{2n-k}(X(Y);E_{\rho}^*)$ 
one has $B(\iota_k^*(Y)\omega,\iota_k^*(Y)\eta)=0$. 
Let $\check{\rho}$ be the contragredient representation
of 
$\rho$. Then the dual bundle $E_\rho^*$ is canonically isomorphic to
$E_{\check{\rho}}$. By \cite[section 3.2.5]{Goodman}, 
for $n$ odd $\check{\rho}$ is equivalent to $\rho$ and for 
$n$ even $\check{\rho}$ is equivalent to $\rho_\theta$. In particular, if $\rho$
is 
not equivalent to $\rho_\theta$, then also $\check{\rho}$ is not equivalent 
to $(\check{\rho})_\theta$ and thus by the previous Proposition \ref{Propkoho1},
for $k>n$ the map $\iota_k^*(Y):
H^{2n-k}(X(Y);E_\rho^*)\to H^{2n-k}(\partial X(Y);E_\rho^*)$ is surjective. 
Therefore, for $k<n$ the map $\iota_k^*(Y): H^{k}(X(Y);E_\rho)\to H^{k}(\partial
X(Y);E_\rho)$ must be the zero map.  
\end{proof}

It remains to consider the cohomology $H^n(X(Y);E_\rho)$. 
We will denote the projection of $\Psi\in
\bigoplus_{P_i\in\mathfrak{P}_\Gamma}\Lambda^n(\nL_{P_i})^*\otimes V_\rho$ 
to the subspace
$\bigoplus_{P_i\in\mathfrak{P}_\Gamma}\mathcal{H}^n(\nL_{P_i};V_\rho)_+$ by
$[\Psi]_+$. 
If we use that $\lambda_{\rho,n}^+=-\lambda_{\rho,n}^-$, it follows easily from 
Lemma \ref{Pullback} and from
equation \eqref{Deflambatau} that for $P_i,P_j\in\mathfrak{P}_\Gamma$,
$\Phi\in\mathcal{H}^n(\nL_{P_i};V_\rho)_-$ and $a\in A_{P_j}$ the form 
$[\iota_a^*(\underline{C}_{P_i|P_j}(\sigma_{\rho,n}^-:-\lambda_{\rho,n}^-)\Phi_{
-\lambda_{
\rho,n}^-})]_+\in \mathcal{H}^n(\nL_{P_j};V_\rho)_+$
does not depend on $a$. Therefore, we obtain 
a well defined map from 
$\bigoplus_{P_i\in\mathfrak{P}_\Gamma}\mathcal{H}^n(\nL_{P_i};V_\rho)_{-}$ to 
$\bigoplus_{P_i\in\mathfrak{P}_\Gamma}\mathcal{H}^n(\nL_{P_i};V_\rho)_{+}$
which we shall denote by
\begin{align*}
\Phi\mapsto [\underline{\mathbf{C}}
(\sigma^{-}_{\rho,n},-\lambda_{\rho,n}^-)\Phi]_+
\end{align*}
for brevity. 
Then we have the following proposition. 
\begin{prop}\label{PropKoho3}
Let $\rho\in\Rep(G)$, $\rho\neq \rho_\theta$. Then one has 
\begin{align}\label{Eqdim}
\dim H^{n}(X(Y);E_\rho)=\frac{1}{2}\dim H^n(\partial X(Y);E_\rho).
\end{align}
More explicitly, for
$\Phi\in\bigoplus_{P_i\in\mathfrak{P}_\Gamma}\mathcal{H}^n(\nL_{P_i};V_\rho)_-$
one has 
\begin{align}\label{Eqi}
i_n^*(Y)E(\Phi:-\lambda_{\rho,n}^-)=\Phi+[\underline{\mathbf{C}}
(\sigma^{-}_{\rho,n},-\lambda_{\rho,n}^-)\Phi]_+ .
\end{align}
and if $\{\Phi_l\}$ is a basis of
$\bigoplus_{P_i\in\mathfrak{P}_\Gamma}\mathcal{H}^n(\nL_{P_i};V_\rho)_-$, 
then a basis of $H^n(X(Y);E_\rho)$ is given by the
$E(\Phi_l:-\lambda_{\rho,n}^-)$.
\end{prop}
\begin{proof}
One has 
\begin{align*}
\chi(X(Y);E_\rho)=\frac{1}{2}\chi(\partial
X(Y);E_\rho)=\frac{1}{2}\dim(V_\rho)\chi(\partial X(Y))=0,
\end{align*}
where the last equality follows from the fact that $\partial X(Y)$ is a 
disjoint union of finitely many tori. 
Thus by Proposition \ref{Propkoho1} and Proposition
\ref{Propkoho2} we conclude
\begin{align*}
(-1)^{n+1}\dim H^n(X(Y);E_\rho)=&\sum_{k=n+1}^{2n}(-1)^k\dim
H^k(X(Y);E_\rho)\\=&\sum_{k=n+1}^{2n}(-1)^k\dim H^k(\partial X(Y);E_\rho).
\end{align*}
If we apply the isomorphism \eqref{Isomkoho}, Proposition \ref{Prop Kostant}
and 
\eqref{lambdadecom}, we see immediately that we have
$\dim H^k(\partial
X(Y);E_\rho)=\dim H^{2n-k}(\partial X(Y);E_\rho)$ for 
each $k$ (let us recall that $\rho$ is not necessarily self-dual). 
Thus from $\chi(\partial(X(Y)))=0$ we conclude
\begin{align*}
\sum_{k=n+1}^{2n}(-1)^k\dim H^k(\partial X(Y);E_\rho)=\frac{1}{2}(-1)^{n+1}\dim
H^n(\partial X(Y);E_\rho).
\end{align*}
Combining the last two 
equations, we obtain \eqref{Eqdim}.

To prove \eqref{Eqi}, one can proceed exactly as in the proof of Proposition
\ref{Propkoho1}, using that 
in degree $n$ one has $w_0\sigma_{\rho,n}^-=\sigma_{\rho,n}^+$. 
Since
$\dim\mathcal{H}^n(\nL_{P_i};V_\rho)_{-}=\dim\mathcal{H}^n(\nL_{P_i};V_\rho)_{+}
$ 
and since the sum in \eqref{Eqi} is direct, the $E(\Phi_l:-\lambda_{\rho,n}^-)$
form 
a basis of $H^n(X(Y);E_\rho)$ by \eqref{Eqdim}  and \eqref{Eqi}.
\end{proof}

At the end of this section, we remark that Harder's construction of Eisenstein
cohomology classes also applies if 
$\rho\neq\rho_\theta$. The difference then is that the Eisenstein 
series might have a pole at the point where one has to evaluate it, since
Proposition \ref{PropER} critically 
uses that $\rho\neq\rho_\theta$. If the Eisenstein series has a pole,  
one has to take its residue \cite{Ha} and the resulting cohomology class 
is actually a square-integrable cohomology class - Proposition
\ref{LemL2cohomology} holds only if $\rho\neq\rho_\theta$.  
Finally, as pointed out by Harder, the cohomology classes defined by Eisenstein
series in Propositions \ref{Propkoho1} and 
\ref{PropKoho3} are always harmonic in the sense that they are smooth 
$E_\rho$-valued $p$-forms on $X$ which are annihilated by the flat Hodge
Laplacian. 
In the present parametrization, this can just be read off from \cite[(3.8)]{MP2}
and Kuga's formula. However, 
in our case the Eisenstein cohomology classes are never square integrable.

\section{The Reidemeister torsion with respect to Eisenstein cohomology
classes} \label{SecReid}
\setcounter{equation}{0}

In this section we define the Reidemeister torsion
$\tau_{Eis}(\overline{X};E_\rho)$,
$\rho\neq\rho_\theta$ of 
$X$ with coefficients in $E_\rho$ with respect to the basis in the cohomology 
of the previous section. 

We will firstly briefly review 
the definition of the Reidemeister torsion of compact manifolds, possibly with 
boundary. For further details, we refer to the paper \cite{Milnor} of Milnor. We
follow \cite[section 1]{Mutorsunim}.
Firstly, we recall the definition of the Reidemeister torsion of 
a cochain complex. We remark that we take the opposite sign than
\cite{Mutorsunim} since 
we work with cochain complexes. 
For a finite-dimensional real vector space $V$ of dimension $d>0$ we use the
common notation
$\det V:=\Lambda^{d} V$. A non-zero element of $\det V$ is called a volume
element. If $L$ is a one-dimensional 
real vector space, we let $L^{-1}:=L^*$ and for $l\in L$, $l\neq 0$, we 
let $l^{-1}\in L^*$ be the unique element defined by $l^{-1}(l)=1$. 
If 
\begin{align*}
C^*:0 \to C^0\overset{d_0}{\longrightarrow} C^{1}\to\dots\to
C^{n-1}\overset{d_{n-1}}{\longrightarrow} C^n\to 0
\end{align*}
is a cochain complex of finite-dimensional non-zero real vector 
spaces $C^q$  and if $H^q$ are its cohomology groups, we let 
\begin{align*}
(\det C^*)^{-1}:=\bigotimes_q(\det C^q)^{(-1)^{q+1}};\quad \det
H^*:=\bigotimes_q(\det H^q)^{(-1)^{q}}.
\end{align*}
Then the Reidemeister torsion $\tau(C^*)$ 
is an invariant associated to the cochain complex $C^*$ which 
is defined as follows. Let $b_q:=\dim d_q(C^{q})$, $h_q:=\dim H^q$. 
Let $\theta_q\in {\Lambda}^{b_q}C^q$ such that $d_q(\theta^q)\neq 0$. Let
$Z^q:=\Ker(d_q)$ and let $\iota_q:Z_q\to C_q$ be the inclusion.
Let $\mu_q\in \det H^q$, $\mu_q\neq 0$, and let $\nu_q\in\Lambda^{h_q}(Z^q)$
such
that $\pi(\nu_q)=\mu_q$, where $\pi:Z^q\to H^q$ is the canonical projection. 
Then $\omega_q:=d_{q-1}(\theta_{q-1})\wedge\theta_q\wedge \iota_q(\nu_q)$ is a
nonzero element of $\det C^q$ and one defines 
\[ 
\tau(C^*):=\bigotimes_q\omega_q^{(-1)^{q+1}} \otimes
\bigotimes_q\mu_q^{(-1)^{q}}\in (\det C^*)^{-1}\otimes \det
H^*. 
\]
It is easy 
to see that $\tau(C^*)$ does not depend on the choices of $\theta_q$, $\mu_q$
and $\nu_q$. 

If for each $q$ bases 
$\mathcal{A}_q:=\{\omega_{q,1},\dots,\omega_{q,m_q)}\}$ of $C^q$, $m_q:=\dim
C_q$, and 
$\mathcal{B}_q:=\{\eta_{q,1},\dots,\eta_{q,h_q}\}$ of $H^q$ 
are given, 
one obtaines non-zero elements 
\begin{align*}
\omega:=\bigotimes_{q}(\omega_{q,1}\wedge\dots\wedge\omega_{q,m_q})^{(-1)^{q+1}}
;\quad
\eta:=\bigotimes_{q}(\eta_{q,1}\wedge\dots\wedge\eta_{q,h_q})^{(-1)^q}
\end{align*} 
of $(\det C^*)^{-1}/\{\pm 1\}$ resp. $\det H^*/\{\pm 1\}$. The element $\omega$
determines an 
identification 
$(\det C^*)^{-1}\otimes\det H^*\cong \det H^*/\{\pm 1\}$ and 
$\eta$ determines 
an identification $\det H^*/\{\pm 1\}\cong\R^{+}$.
Thus if $\omega$ is fixed, one can associate to $\tau(C^*)$ an element
$\tau(C^*,\omega)\in \det
H^*/\{\pm 1\}$ and if $\omega$ and $\eta$ are fixed, one can associate 
to $\tau(C^*)$ an element $\tau(C^*,\omega,\mu)\in (0,\infty)$.
\newline

Now let $M$ be a compact smooth manifold with boundary 
$\partial M$. 
Assume that 
$\partial M$ is a disjoint union $\partial M={\partial M}_1\sqcup{\partial
M}_2$.
We do not 
exclude that $\partial M$, $\partial M_1$ or $\partial M_2$ are empty.  
Let $\tilde{M}$ be the universal covering of $M$ with 
a fixed base-point. Then the fundamental group $\pi_1[M]$ of $M$
acts on $\tilde{M}$ by deck transformations. 
If $\rho$ is a finite-dimensional 
representation of $\pi_1[M]$ on a real or complex 
vector space $V_\rho$, we let $F_\rho:=\tilde{M}\times_{\pi_1[M]}V_\rho$ be 
the 
associated flat vector bundle over $M$; every flat vector 
bundle is obtained in this way. The representation $\rho$ is called 
unimodular if $|\det(\rho)(\gamma)|=1$ for all $\gamma\in\pi_1[M]$. 

Let $\mathcal{K}$ be a smooth triangulation of $M$ containing a subcomplex
$\mathcal{K}_1$
triangulating ${\partial M}_1$. Let 
$\widetilde{\mathcal{K}}$ be the lift of $\mathcal{K}$ to a triangulation of
$\tilde{M}$ and let 
$\widetilde{\mathcal{K}}_1$ be the corresponding lift of
$\mathcal{\mathcal{K}}_1$; then $\widetilde{\mathcal{K}}_1$
is a subcomplex of $\widetilde{\mathcal{K}}$. 
Let
$C^q(\widetilde{\mathcal{K}})$ resp. $C^q(\widetilde{\mathcal{K}}_1)$ be
the real vector space generated by the $q$-cochains of
$\widetilde{\mathcal{K}}$ resp. $\widetilde{\mathcal{K}}_1$. 
The real 
group algebra $\R[\pi_1(M)]$ acts on the spaces $C^q(\widetilde{\mathcal{K}})$
resp. $C^q(\widetilde{\mathcal{K}}_1)$ by deck transformations and one defines
\begin{align*}
C^q(\mathcal{K};F_\rho):=C^q(\widetilde{\mathcal{K}})\otimes_{\R[
\pi_1(M)]} V_\rho; \quad
C^q(\mathcal{K}_1;F_\rho):=C^q(\widetilde{\mathcal{K}}_1)\otimes_{\R
[
\pi_1(M)]} V_\rho.
\end{align*}
The operators $d_q:C^q(\widetilde{\mathcal{K}})\to
C^{q+1}(\widetilde{\mathcal{K}})$ and
$d_q:C^q(\widetilde{\mathcal{K}}_1)\to
C^{q+1}(\widetilde{\mathcal{K}}_1)$
induce coboundary operators
$d_q:C^q(\mathcal{K};F_\rho)\to
C^{q+1}(\mathcal{K};F_\rho)$ and $d_q:C^q(\mathcal{K}_1;F_\rho)\to
C^{q+1}(\mathcal{K}_1;F_\rho)$ and one obtains cochain complexs 
$C^*(\mathcal{K};F_\rho)$ and $C^*(\mathcal{K}_1;F_\rho)$. The second 
complex is a subcomplex of the first one and thus the spaces 
\begin{align*}
C^q(\mathcal{K},\mathcal{K}_1;F_\rho):=C^q(\mathcal{K};F_\rho)/C^q(\mathcal{
K}_1;F_\rho)
\end{align*}
also form a cochain complex which one denotes by
$C^{*}(\mathcal{K},\mathcal{K}_1;F_\rho)$. 
Let $H^{*}(\mathcal{K};F_\rho)$ resp. $H^*(\mathcal{K},\mathcal{K}_1;F_\rho)$
denote 
the cohomology groups of $C^*(\mathcal{K};F_\rho)$ resp.
$C^*(\mathcal{K},\mathcal{K}_1;F_\rho)$. Then these groups are canonically
isomorphic to the singular cohomology groups 
$H^*(M;F_\rho)$ resp. the relative singular cohomology groups $H^*(M,{\partial
M}_1;F_\rho)$ of $M$  resp. $(M,{\partial M}_1)$ with coefficients in $F_\rho$,
\cite{Wh}.

Fix a basis 
$x_1,\dots,x_m$, $m=\dim(V_\rho)$,  of $V_\rho$ and consider 
the associated element $\theta:=x_1\wedge\dots\wedge
x_m\in\det(V_\rho)/\{\pm 1\}$. Fix an 
embedding of $\mathcal{K}$ into $\tilde{\mathcal{K}}$.  
Then tensoring 
the $x_i$ with the $q$-cochains of $\mathcal{K}$ resp.
$\mathcal{K}-\mathcal{K}_1$, embedded into $\tilde{\mathcal{K}}$, one obtains a
canonical family of bases of 
$C^q(\mathcal{K};F_\rho)$ and of 
$C^q(\mathcal{K},\mathcal{K}_1;F_\rho)$. Since $\rho$ is unimodular, the choice 
of embedding of $\mathcal{K}$ into $\tilde{\mathcal{K}}$ changes this basis only
up to sign
\cite{Mutorsunim}. Therefore, the Reidemeister torsions
$\tau(\mathcal{K};F_\rho;\theta)\in \det
H^{*}(M;F_\rho)/\{\pm
1\}$ resp.
$\tau(\mathcal{K},\mathcal{K}_1;F_\rho;\theta)\in \det
H^{*}(M,{\partial M}_1;F_\rho)/\{\pm
1\}$ are well defined. These torsions are invariant under subdivision
\cite{Milnor}, \cite{Mutorsunim} and are 
thus independent of the choice of the particular smooth 
triangulations $\mathcal{K}$ resp. $\mathcal{K}_1$. Thus one obtains 
well defined elements 
\begin{align*}
\tau(M;F_\rho;\theta)\in \det H^{*}(M;F_\rho)/\{\pm 1\};\:
\tau(M,{\partial
M}_1;F_\rho;\theta)\in \det H^{*}(M,{\partial M}_1;F_\rho)/\{\pm
1\}.
\end{align*}
The latter torsions still depend on the choice of $\theta$. 
However, if the Euler characteristic of the complex $C^{*}(\mathcal{K})$ resp.
$C^*(\mathcal{K},\mathcal{K}_1)$
vanishes, then
they are independent of $\theta$ \cite{Mutorsunim} and thus they are
combinatorial
invariants. This is the case if $M$ is a closed
odd-dimensional 
manifold but it will also be the case in the present paper, since 
here $M$ will be an odd-dimensional manifold whose boundary is a disjoint 
union of tori. If the Euler characteristic does not vanish, one has 
to fix a volume element $\theta$ or more generally a metric
$h_F$ on $F_\rho$ resp. on $\det F_\rho$ as in \cite{Bismut}, \cite{BM1},
\cite{BM2}. In this 
context, we remark that in
the variation formula of Br\"uning and Ma, the integrand of the first term on
the right hand 
side of \cite[(3.26)]{BM2}, in which the variation of $h_F$ is incorporated 
in the unimodular case, vanishes if the boundary is a union of flat
tori.

If we assume that $\chi(M)=0=\chi(M,{\partial M}_1)$ and if for each $q$ bases
$\mathcal{B}^q$ of $H^q(M;F_\rho)$ and 
$\tilde{\mathcal{B}}^q$ of $H^q(M,{\partial M}_1;F_\rho)$ are given and 
if $\mathcal{B}:=\sqcup_{q}\mathcal{B}_q$,
$\widetilde{\mathcal{B}}:=\sqcup_q\mathcal{B}_q$, then 
by the above construction we obtain well defined elements
\begin{align*}
\tau(M,\mathcal{B};F_\rho)\in (0,\infty);\quad
\tau(M,{\partial M_1},\widetilde{\mathcal{B}};F_\rho)\in (0,\infty). 
\end{align*}

If for each $q$ one is given another set of bases $\mathcal{B}_q'$ of
$H^q(M;F_\rho)$ 
and if one denotes by $A_q$ the base change matrix from $\mathcal{B}_q$ to 
$\mathcal{B}_q'$, then it is easy to see that for
$\mathcal{B}':=\sqcup_q\mathcal{B}_q'$ one has
\begin{align}\label{BasecRT}
\tau(M,\mathcal{B}';F_\rho)=\prod_q|\det(A_q)|^{(-1)^{q}}\tau(M,\mathcal{B}
;F_\rho).
\end{align}
  
If $g$ and $h$ are metrics on $M$ resp. $F_\rho$, then 
for each $q$ the flat Hodge-Laplace operator is defined and by 
the Hodge De Rham isomorphism the corresponding space of harmonic forms 
satisfying absolute boundary conditions is canonically
isomorphic to the cohomology  
$H^*(M;F_\rho)$. Thus we obtain a family of bases of $H^*(M;F_\rho)$ 
which arise from $L^2$-orthonormal bases of harmonic forms under this
isomorphisms and  we shall 
denote by $\tau(M,g,h;F_\rho)\in \R^+$
the associated Reidemeister torsion.  Analogously one can define 
$\tau(M,{\partial M}_1,g,h;F_\rho)$  if one uses the Laplacians with relative
boundary 
conditions at ${\partial M}_1$ and absolute boundary conditions at ${\partial
M}_2$. 
\newline

Now we can define the Reidemeister torsion of the manifold $\overline{X}$ from
secion
\ref{subsmfld} with 
coefficients in the flat bundle $E_\rho$. We use the 
notations of the preceding section. 
For $P_j\in\mathfrak{P}_\Gamma$, the space $\Lambda^p(\nL_{P_j})^*\otimes
V_\rho$ carries an inner product 
induced by the restriction of the inner product \eqref{metr} on $\gL$ to
$\nL_{P_j}$ and the admissible 
inner product on $V_\rho$. 
We normalize this inner product by $\bigl(\vol((\Gamma\cap N_{P_j})\backslash
N_{P_j})\bigr)^{-1/2}$. Then 
the map in \eqref{mapPhi}, which induces an isomorphism on cohomology,
embeds $\Lambda^p\nL_{P_j}^*\otimes V_\rho$
isometrically into the space of square integrable $E_\rho$-valued $p$-forms on
the 
flat torus $(\Gamma\cap N_{P_j})\backslash N_{P_j}$ with its induced metric.  
For each $P_j\in\mathfrak{P}_\Gamma$ we fix an orthonormal basis
$\Phi_{i,j}^k$ of $\mathcal{H}^k(\nL_{P_j};V_\rho)$ if 
$k>n$ and we fix an orthonormal basis $\Phi_{i,j}^n$ of
$\mathcal{H}^n(\nL_{P_j};V_\rho)_{-}$. Then 
by Proposition \ref{Propkoho1} for $k>n$ we obtain bases
\[
\mathcal{B}_k:=\{E(\Phi_{i,j}^k,-\lambda_{\rho,k})\colon
j=1,\dots,\kappa(\Gamma)\colon i=1,\dots,\dim
\mathcal{H}^k(\nL_{P_j};V_\rho)\}
\]
of $H^k(X;E_\rho)$ and by Proposition \ref{PropKoho3}
we obtain a basis  
\[
\mathcal{B}_n:=\{E(\Phi_{i,j}^n,-\lambda_{\rho,n}^-)\colon
j=1,\dots,\kappa(\Gamma)\colon i=1,\dots,\dim
\mathcal{H}^n(\nL_{P_j};V_\rho)_{-}\}
\]
of $H^n(X;E_\rho)$. Moreover, by Proposition \ref{Propkoho2}, the cohomology
groups $H^k(X;E_\rho)$ vanish for $k<n$. 
If one is given another set of orthonormal bases $\Psi_{i,j}^k$ of each
$\mathcal{H}^k(\nL_{P_j};V_\rho)$, then, 
since the assignment $\Phi\mapsto E(\Phi:-\lambda_{\rho,k})$ is linear, 
the determinant of the base - change matrix $A_k$ from the basis
$\mathcal{B}_k$ to the corresponding basis $\mathcal{B}_k'$ 
has absolute value $1$. The same holds for the cohomology in degree $n$.
Therefore, by \eqref{BasecRT} one can define the Reidemeister torsion 
\begin{align*}
\tau_{Eis}(\overline{X};E_\rho)\in (0,\infty) 
\end{align*}
 as the Reidemeister torsion of $\overline{X}$
with coefficients in $E_\rho$ with respect to the set of bases in the
cohomology groups
$H^*(\overline{X};E_\rho)\cong H^*(X;E_\rho)$ formed by any 
set of bases $\mathcal{B}=\sqcup_{l=n}^{2n} \mathcal{B}_l$ of the above form. 
This construction is obviously independent of 
the truncation parameter $Y$ which we choose to realize 
$\overline{X}$ as $X(Y)$, since the cohomology classes we work with now are 
restrictions of cohomology classes from $X$ to $X(Y)$.

\section{Regularized and relative traces and torsions}\label{secrelreg}
\setcounter{equation}{0}
Let $X=\Gamma\backslash\widetilde{X}$ be the same 
manifold as in section \ref{subsmfld} and for $\nu\in\hat{K}$ 
let $E_\nu$ be the locally homogeneous vector bundle 
over $X$ as in section \ref{secDR} equipped with 
a metric $h$ induced by the inner product on $V_\nu$. 
We let $A_\nu$ be the differential operator on $E_\nu$ which on
$C^\infty(\Gamma\backslash
G,\nu)$ acts as $-\Omega$ under the isomorphism in \eqref{globsect1}. By the
arguments 
of \cite[section 4]{MP1}, $A_\nu$ with 
domain the smooth compactly supported sections is  bounded from below 
and essentially selfadjoint. Its closure will be denoted by the same symbol. 

In order to allow compact perturbations of the metrics on $X$ and $E_\rho$ and 
to allow the compactly supported variation of the De Rham complex due to Vishik
and Lesch, we will also need to consider operators which coincide with the
operator 
$A_\nu$ only outside a compact subset of $X$. Therefore, we 
make the following definition, which is parallell to the definition of M\"uller
\cite[Definition 5.3]{Mulecn}.

\begin{defn}\label{DefDO}
A second order elliptic differential operator $B_\nu$ on $C^\infty(X,E_\nu)$ is
called 
locally homogeneous at infinity if there exists a $Y_1\geq Y_0$ and a
$C(B_\nu)\in\R$ such
that  for 
all $f\in C^\infty(X,E_\nu)$ with $f|_{X(Y_1)}=0$  one has $B_\nu f=A_\nu
f+C(B_\nu)f$ and if there exist metrics $g_1$ on $X$ and $h_1$ on $E_\nu$ which
coincide with $g$ and $h$ outside 
$X(Y_1)$ such that $B_\nu$ with domain the smooth compactly supported sections
is symmetric with respect to the $L^2$-inner product induced by $g_1$ and $h_1$
and 
such that $B_\nu$ is bounded from below. 
\end{defn}

\begin{bmrk}\label{BmrkMu}
The manifold $X$ with the new metric $g_1$ and the bundle $E_\nu$ with the
metric 
$h_1$ are of course just 
special cases of a Riemannian manifold with cusps and a bundle which is locally 
homogeneous at infinity in 
the sense of M\"uller \cite[Definitions 5.1, 5.2]{Mulecn}. On the other hand, 
Definition \ref{DefDO} is slightly more general than
the corresponding definition given in \cite[section 5]{Mulecn}, 
since the operators 
considered by M\"uller, who mainly studied 
index-problems, arise as 
squares of generalized Dirac operators, see \cite[Definition 5.3]{Mulecn}. 
However, an inspection of the proofs shows that the results about the heat
kernels of such operators obtained 
by M\"uller in section 7 of \cite{Mulecn} continue to hold for the operators
$B_\nu$ 
introduced in the previous Definition \ref{DefDO}.
\end{bmrk}

Let $e^{-tB_\nu}$ be the heat semigroup of $B_\nu$. 
Then $e^{-tB_\nu}$ acts 
on $L^2(X,E_\nu)$ as an integral operator with smooth kernel
\[
K_{B_\nu}(t,x,y)\in C^\infty(X\times X,E_\nu \boxtimes E_\nu^*),
\]
where the $L^2$-space is taken with respect to the metrics $g_1$ and $h_1$
implicit in the definition of $B_\nu$. 
The kernel
$K_{B_\nu}(t,x,y)$ has been studied by M\"uller in \cite{Mulecn}. 
For each $i=1,\dots,\kappa(X)$ consider the complete infinite volume cusp  
$F_{P_i}:=(\Gamma\cap N_{P_i})\backslash G/K$ which 
is equipped with the metric induced from the hyperbolic metric on $G/K$.  
Then $\nu$ defines a locally homogeneous vector 
bundle $W_{P_i,\nu}$ over $F_{P_i}$. 
The non-uniformity in the Gaussian estimates for $K_{B_\nu}$ can be controlled
by the following function. 
Let $r_{P_i}$ be 
the function on $G$ given by 
\[
r_{P_i}(n_{P_i}\iota_{P_i}(y)k):=y;\quad n_{P_i}\in
N_{P_i},\:\iota_{P_i}(y)\in A_{P_i},\:k\in K. 
\]
The function $r_{P_i}$ obviously descends to a function on $F_{P_i}$.  
Then we fix a smooth fuction $r$ on $X$ such that $r(x)=r_{P_i}(x)$ 
if $x\in F_{P_i}(Y_0+1)$ and such that $r(x)\geq 1$ for all $x\in X$.  
If $C(B_\nu)\in\R$ is the constant from definition \ref{DefDO}, then by the same
arguments 
as in \cite[section 4]{MP1}
$-\Omega+C(B_\nu)$ defines an elliptic essentially selfadjoint operator
$B_{P_i,\nu}$
which acts on the
smooth section of $W_{P_i,\nu}$ and which is bounded from below. Let
$K_{B_{P_i,\nu}}(t,x,y)$ 
be the integral kernel of $e^{-tB_{P_i,\nu}}$.
In \cite{Mulecn}, the integral kernel $K_{B_\nu}(t,x,y)$ was constructed by
patching together the interior heat kernel of $B_\nu$ on the compact 
manifold $X(Y_1+1)$ with the restriction of the 
kernels $K_{B_{P_i,\nu}}(t,x,y)$ 
to the cusps $F_{P_i}(Y_1)$ of $X$. 
This construction implies the following proposition.

\begin{prop}\label{Estheatk}
For each $T>0$ there exist constants $C_1,C_2>0$ such that for all $x,y\in
X$ and all $t\in (0,T]$ and $j\in\{0,1\}$ one can estimate
\begin{align*}
\left\|B_\nu^j K_{B_\nu}(t,x,y)\right\|\leq C_1
r(x)^{n} r(y)^{n} 
t^{-\frac{d}{2}-j}e^{-C_2\frac{d^2(x,y)}{t}}.
\end{align*}
Moreover, for each $T>0$ there exists a constant $C>0$ such that for each 
$P_i\in\mathfrak{P}_\Gamma$, all 
$x\in F_{P_i}(Y_1+1)\subset F_{P_i}(Y_1)$ and all $t\in (0,T]$ one
has
\begin{align*}
\left\|K_{B_\nu}(t,x,x)-K_{B_{P_i,\nu}}(t,x,x)\right\|\leq
C e^{-\frac{\dist^2(x,\partial
F_{P_i}(Y_1))}{t}}
\end{align*}
\end{prop}
\begin{proof}
This follows immediately from the construction of the 
kernel $K_{B_\nu}$ given by M\"uller in \cite[chapter 3, chapter
7]{Mulecn} which works identically 
in the present situation (see remark \ref{BmrkMu}) and for 
$j=0$ it is stated in \cite[Proposition 7.10]{Mulecn}. Let us briefly
outline the
main steps of the argument: We may assume that $B_\nu$ coincides with 
$A_\nu$ outside a compact subset of $X$.
Let $\tilde{A}_\nu$ be the lift of $A_\nu$ to a differential operator of Laplace
type acting on the smooth sections
of the bundle $\tilde{E}_\nu$ over $\tilde{X}$ which is defined as in section
\ref{secDR} and let 
$K_{\tilde{A}_\nu}$ be the integral kernel of $e^{-t\tilde{A}_\nu}$. 
Then $K_{\tilde{A}_\nu}$
was constructed 
by Donnelly for the scalar case \cite{Donnelly} and, as 
remarked by Donnelly \cite[page 485]{Donnelly}, this construction carries over
to the vector valued case. We remark 
that Donnelly's construction is applicable in the present situation since 
in the present case there exist of course discrete, torsion-free subgroups
$\Gamma_1$ of
$G$ such that $\Gamma_1\backslash\tilde{X}$ 
is compact. By \cite[page 488 (P4)]{Donnelly}, for each $T>0$ there 
exists a constant $C_1$ such that for all $t\in (0,T]$ and for $j\in\{0,1\}$,
one has the estimate
\begin{align*}
\left\|\tilde{A}_\nu^jK_{\tilde{A}_\nu}(t,x,y)\right\|
=\left\|\frac{d^j}{(dt)^j}K_{\tilde{A}_\nu}(t,x,y)\right\|\leq C_1
t^{-\frac{d}{2}-j} e^{-\frac{d^2(x,y)}{4t}},\quad x,y\in \tilde{X}.
\end{align*}
According to the isomorphism $C^\infty(G,\nu)\cong
C^\infty(\tilde{X},\tilde{E}_\nu)$ from section \ref{secDR}, we regard
$K_{\tilde{A}_\nu}$ as a function 
on $G\times G$ with values in $\End(V_\nu)$. This function 
is invariant under the diagonal action of $G$, see \cite[section
4]{MP1}.  
Similarly, we regard the kernel $K_{B_{P_i,\nu}}$ as 
an $\End(V_\nu)$-valued function on $(\Gamma\cap
N_{P_i})\backslash G\times (\Gamma\cap
N_{P_i})\backslash G$  corresponding to the 
canonical isomorphism $C^\infty(F_{P_i},E_\nu)\cong C^\infty((\Gamma\cap
N_{P_i})\backslash G,\nu)$ which is obtained as in section \ref{SecDRcusp}.
Then 
$K_{B_{P_i,\nu}}$ is obtained 
by averaging $K_{\tilde{A}_\nu}$ over $\Gamma\cap N_{P_i}$,
i.e. 
\begin{align*}
K_{B_{P_i\nu}}(t,x,y)=\sum_{\gamma\in\Gamma\cap
N_{P_i}}K_{\tilde{A}_{\nu}}(t,x,\gamma y).
\end{align*}
If $B_\nu=A_\nu$ everywhere, one of course obtains the kernel $K_{B_\nu}$ on
$X\times X$
by averaging $K_{\tilde{A}_\nu}$ over the whole group $\Gamma$.
By \cite[Lemma 3.20]{Mulecn}, for fixed $Y>0$ and $T>0$ there exists a constant
$C>0$ such that one can estimate 
\begin{align*}
\sum_{\gamma\in\Gamma\cap N_{P_i}}e^{-\frac{d^2(x,\gamma y)}{4t}}\leq C
r(x)^nr(y)^n e^{-\frac{d^2(x,y)}{4t}}
\end{align*}
for all $x,y\in F_{P_i}(Y)$ and all $t\in (0,T]$
(let us remark
that the $\Gamma$ in 
\cite[chapter 3]{Mulecn} is the present $\Gamma\cap N_{P_i}$). This implies 
the estimates for the kernels $K_{B_{P_i\nu}}$ restricted to $F_{P_i}(Y)\times
F_{P_i}(Y)$. We recall that the $F_{P_i}(Y)$ form the 
cusps of $X$. 
Now as in \cite[chapter 7]{Mulecn}, one obtains the kernel $K_{B_\nu}$ 
by patching together the kernels $K_{B_{P_i\nu}}$, $P_i\in\mathfrak{P}_\Gamma$, 
with the heat kernel
$e_2$ of a compact manifold restricted to the compact part $X(Y_1+1)$ of $X$
using 
the parametrix method. A similar construction will be carried out in detail 
in section \ref{secheatpertL} of the present paper.
Since 
the kernel $e_2$ satisfies the required estimates, one obtains the estimates in 
(1) as in \cite[chapter 7]{Mulecn}. 
The estimate in (2) is an immediate consequence of this construction since the
term $P*Q$ which can be defined as in \cite[chapter 7]{Mulecn} on page 62 after
(7.8) is easily seen to satisfy this 
estimate for $x\in F_X(Y_1+1)$ if one employs the estimate (7.7) \cite[chapter
7]{Mulecn} for 
the function $Q$, the fact that $Q$ has uniform compact support in the first
variable 
and the fact that the parametrix $P$ can be estimated using (7.1) in 
\cite[chapter 7]{Mulecn}. 
\end{proof}

We shall now review some concepts to model the heat trace of the operators 
just introduced. Firstly, for the operators $A_\nu$, in \cite{MP1}, we
introduced the regularized trace of $e^{-tA_\nu}$ following
ideas 
of Melrose \cite{Me} and Park \cite{Park}. This trace was defined as follows. 
Let 
\begin{align*}
K_{A_\nu}(t,x,y)\in C^\infty(X\times X,E_\nu \boxtimes E_\nu^*)
\end{align*} 
be the integral kernel of $e^{-tA_\nu}$. Then by \cite[equation 5.7]{MP1} one
has 
an asymptotic expansion 
\begin{align}\label{asymexp}
\int_{X(Y)}\Tr K_{A_\nu}(t,x,x)dx =\sum_{\substack{\sigma\in\hat{M}\\
\left[\nu:\sigma\right]\neq
0}}\biggl(\frac{\kappa(\Gamma) e^{tc(\sigma)}\log{Y}\dim(\sigma)}{\sqrt{4\pi
t}}\biggr)+a_0(t)+o(1),
\end{align}
as $Y\to\infty$. Since on a compact manifold the trace
of the heat operator is given by the integral of the pointwise fibre-trace of
the
heat kernel on the diagonal, we defines the regularized trace 
$\Tr_{\reg}(e^{-tA_\nu)})$ as the constant term $a_0(t)$ of the asymptotic
expansion \eqref{asymexp}. 
\newline

There is also a different way to model the definition of the heat trace, namely 
the concept of the relative trace, which is due to M\"uller \cite{Mulecn}. 
This concept is more convenient for our approach to the gluing formula since
in this way we can avoid any interchange of limits. 
The idea behind the definition of the relative trace is as follows. Restricting 
the operator $A_\nu$ to the 0-th Fourier coefficients on the cusps of the 
manifold, one obtains an ordinary differential operator which can be computed
explicitly. 
Then by a result of M\"uller the difference of the associated heat kernels is
trace class 
and the trace of this difference is by definition the relative trace. For a
general 
discussion of relative traces, we refer to \cite{Murel}.

We shall now make this definition more explicit in the present context.
Let $P_i\in\mathfrak{P}_\Gamma$, let
$Y_0$ be as in section \ref{subsmfld} and let
$u\in\R^+$ with $u> Y_0$.
Let $L^2\left([u,\infty),y^{-d}dy;V_\nu\right)$ denote 
the space of $V_\nu$-valued  $L^2$-functions on $[u,\infty)$ 
with respect to the measure $y^{-d}dy$ and the fixed metric on $V_\nu$. 
Then the decomposition \eqref{Zerlegung des FB} gives a natural inclusion
$I_{P_i,u}$
of
$L^2\left([u,\infty),y^{-d}dy;V_\nu\right)$
into the space $L^2(\Gamma\backslash G,\nu)$. More explicitly, if 
$\chi_{u}$ is the characteristic function of the set  $[u,\infty)$ 
then by $\eqref{Zerlegung des FB}$ function 
\begin{align*}
I_{P_i,u}\phi(n_{P_i}\iota_{P_i}(y)k):=\frac{1}{\sqrt{\vol(\Gamma\cap
N_{P_i}\backslash
N_{P_i})}}\nu(k^{-1})\chi_u(y)\phi(y);\: n_{P_i}\in N_{P_i},\:\iota_{P_i}(y)\in
A_{P_i},\:k\in K
\end{align*}
on $G$ is $\Gamma$-invariant and can therefore be regarded as an element of
$L^2(\Gamma\backslash G,\nu)$.
By \eqref{Eigenschaft des FB} and
\eqref{integral}, the assignment $\phi\mapsto I_{P_i,u}\phi$  embeds the
space
$L^2\left([u,\infty),y^{-d}dy;V_\nu\right)$ isometrically into 
$L^2(\Gamma\backslash G,\nu)$. 
For $f\in L^2(\Gamma\backslash G,\nu)$ and $x\in \Gamma\backslash G$, $x=\Gamma
g $, $g=n_{P_i}\iota_{P_i}(y)k$ let
\begin{align}\label{Def. nullter Koeff.}
f_{P_i,u}(x):=\frac{1}{\vol(\Gamma\cap N_{P_i}\backslash
N_{P_i})}\chi_{u}(y)\int_{\Gamma\cap N_{P_i}\backslash
N_{P}}{f(n_{P_i} g)dn_{P_i}}.
\end{align}
Then $f\mapsto f_{P_i,u}$ is the orthogonal projection
of $L^2(\Gamma\backslash G,\nu)$ onto 
$I_{P_i,u}\left(L^2\left([u,\infty),y^{-d};V_\nu\right)\right)$.
Since $\Omega$ is $G$-invariant, one has $\Omega f_{P_i,u}=(\Omega
f)_{P_i,u}$
for every $f\in C^\infty(\Gamma\backslash G,\nu)$. Thus $-\Omega$ and the
embedding 
$I_{P_i,u}$ induce in a
canonical way a differential
operator $T_{\nu}$ on $C^\infty(\left[u,\infty\right);V_\nu)$. The operator
$T_\nu$ can be
computed explicitly. Let $\Omega_{M_{P_i}}$ be 
the casimir element of $\mL_{P_i}$ associated to the restriction fo the 
normalized Killing form of $\gL$ to $\mL_{P_i}$ and define an
endomorphism
$L(\nu)$ of
$V_\nu$ by $L(\nu):=-\nu|_{M_{P_i}}(\Omega_{M_{P_i}})$, where $\nu|_{M_{P_i}}$
denotes the restriction
of
$\nu$ to $M_{P_i}$. Then the following Lemma holds. 
\begin{lem}\label{LemT1}
One has $T_\nu=-y^2\frac{d^2}{dy^2}+(d-2)y\frac{d}{dy}+L(\nu)$. 
\end{lem}
\begin{proof}
Clearly, one can assume that $P_i=P_0$. 
For $\alpha\in\Delta^+(\mathfrak{g}_{\C},\mathfrak{h}_{\C})$ let
$\mathfrak{g}_{\C}^\alpha$ be
the corresponding root space. Then one can choose $X_{\alpha}$ in
$\mathfrak{g}_{\C}^\alpha$, 
$X_{-\alpha}\in\mathfrak{g}_{\C}^{-\alpha}$ such that
$B(X_\alpha,X_{-\alpha})=1$, $\left[X_{\alpha},X_{-\alpha}\right]=H_{\alpha}$,
where $H_{\alpha}$ is the coroot corresponding to $\alpha$ as in section
\cite[section 2.1]{MP1}.
By \cite[(2.3)]{MP1} one has
\begin{align*}
\sum_{\alpha\in\Delta^+(\mathfrak{g}_{\C},\mathfrak{a}_{\C})}H_{\alpha}=2nH_1.
\end{align*}
Thus by the definition of $\Omega$ and $\Omega_M$ one has
\begin{align*}
\Omega=&\sum_{i=1}^{n+1}H_i^2+\sum_{\alpha\in\Delta^+(\mathfrak{g}_{\C},
\mathfrak {h}_{\C})}
\left(X_{\alpha}X_{-\alpha}+X_{-\alpha}X_{\alpha}\right)\\
=&H_{1}^{2}+\sum_{\alpha\in\Delta^+(\mathfrak{g}_{\C},\mathfrak{a}_{\C})}
\left(H_{\alpha}+2X_{-\alpha}X_{\alpha}\right)+
\sum_{i=2}^{n+1}H_i^2+\sum_{\alpha\in\Delta^+(\mathfrak{m}_{\C},\mathfrak{b}_{\C
})}\left(X_{\alpha}X_{-\alpha}+X_{-\alpha}X_{\alpha}\right)\\
=&H_{1}^2+2nH_{1}+\Omega_{M}\quad\mod U(\mathfrak{g}_{\C})\mathfrak{n}_{\C}.
\end{align*}
Now the element $H_1$ induces the differential operator $-y\frac{d}{dy}$ on
$\left(0,\infty\right)$ under
$\iota_{P}$. 
Since $\Omega_M$ is invariant under the anti-involution of
$U(\mathfrak{m}_{\C})$ induced by $Y\mapsto -Y$, $Y\in\mathfrak{m}_{\C}$,
the proposition follows.
\end{proof}
Consider the differential operator 
\begin{align*}
T_\nu^0:=-y^2\frac{d^2}{dy^2}+(d-2)y\frac{d}{dy}+L(\nu) 
\end{align*}
acting on $C_c^\infty\left(\left(u,\infty\right);V_\nu\right)\subset
L^2\left(\left[u,\infty\right),y^{-d}dy;V_\nu\right)$. Its selfadjoint extension
with respect to the  Dirichelet boundary condition at $u$ will be denoted
by $T_\nu^0$ too.
First consider the operator $L(\nu)$. 
\begin{lem}\label{LemT2}
For $\sigma\in\hat{M}$ let $P_\sigma$ denote the orthogonal projection
from $V_\nu$ to the $\sigma$-isotypical subspace of $\nu|_{M}$. Then one has
\begin{align*}
L(\nu)=\sum_{\substack{\sigma\in\hat{M}\\ \left[\nu:\sigma\right]\neq
0}}-\left(c(\sigma)+\frac{(d-1)^2}{4}\right)P_\sigma,
\end{align*}
where $c(\sigma)$ is as in \eqref{csigma}.
\end{lem}
\begin{proof}
A standard computation gives $\sigma(\Omega_M)=c(\sigma)+\frac{(d-1)^2}{4}$ and
the Lemma follows. 
\end{proof}
Now consider the differential operator 
\begin{align}\label{Diffop}
T^0:=-y^2\frac{d^2}{dy^2}+(d-2)y\frac{d}{dy}
\end{align}
acting on $C_c^\infty\left(u,\infty\right)\subset
L^2(\left(\left[u,\infty\right),y^{-d}dy\right)$.
The
self-adjoint
extension of $T^0$ with Dirichelt boundary condition at $u$ will be
denoted by
$T^0$ too.

\begin{lem}\label{LemT3}
Let $H^u(t,y,y')$ denote the heat kernel of $T^0$. Then for $Y>u$ one has
\begin{align*}
\int_{u}^{Y}H^u(t,y,y)y^{-d}dy=e^{-t\frac{(d-1)^2}{4}}\left(\frac{\log{Y}
}{\sqrt{4\pi
t}}-\frac{\log{u}
}{\sqrt{4\pi
t}}-\frac{1}{4}\right)+o(1),
\end{align*}
as $Y\to\infty$.
\end{lem}
\begin{proof}
Consider the operator 
\begin{align*}
\tilde{T}^0:=-y^2\frac{d^2}{dy^2}-y\frac{d}{dy}+\frac{(d-1)^2}{4} 
\end{align*}
with  Dirichlet boundary condition at $u$ 
acting on $\dom(\tilde{T}^0)\subset L^2\left([u,\infty),y^{-1}dy\right)$. 
Then one has
\begin{align*}
T^0=y^{\frac{d-1}{2}}\circ \tilde{T}^0\circ y^{-\frac{d-1}{2}}. 
\end{align*}
Moreover, changing variables $y=e^{r}$, $\tilde{T}^0$ is equivalent to 
$D^0:=-\frac{d^2}{dr^2}+\frac{(d-1)^2}{4}$ 
where $\dom(D^0)\subset L^2\left((\log{u},\infty),dr\right)$ and where
Dirichlet boundary condition at $\log{u}$ for $D^0$ are taken.
Let $p^u(t,r,r')$ denote the heat kernel of $D^0$. Then it is well known that
\begin{align*}
p^u(t,r,r')=e^{-t\frac{(d-1)^2}{4}}\frac{1}{\sqrt{4\pi
t}}\left(e^{-\frac{(r-r')^2}{4t}}-e^{-\frac{(r+r'-2\log{u})^2}{4t}}\right),
\end{align*}
see for example \cite{CJ}. Thus one obtains
\begin{align*}
H^u(t,y,y')=e^{-t\frac{(d-1)^2}{4}}\frac{(yy')^{\frac{d-1}{2}}}{\sqrt{4\pi
t}}\left(e^{-\frac{\log^2{(y'/y)}}{4t}}-e^{-\frac{(\log{(yy')}-2\log{u})^2}{4t}}
\right)
\end{align*}
and the lemma follows from a simple computation.
\end{proof}

For every
$P_i\in\mathfrak{P}_{\Gamma}$ let $T_{\nu,u}^{P_i}$ be the operator $T_\nu^0$
regarded as an operator on $L^2(\Gamma\backslash G,\nu)$ via the embedding
$I_{P_i,u}$ defined above. Then we put
\begin{align*}
T_{\nu,u}:=\bigoplus_{P_i\in\mathfrak{P}_\Gamma}T_{\nu,u}^{P_i}.
\end{align*}
We let $K_{T_{\nu,u}}(t,x,y)$ be the integral kernel of $e^{-tT_{\nu,u}}$,
where 
the latter heat semigroup
acts on $L^2(\Gamma\backslash G,\nu)$ in the obvious way. If 
$B_\nu$ is locally homogeneous at infinity
in the sense of definition \ref{DefDO}, then we always 
choose the parameter $u>Y_1$, where $Y_1$ is as in Definition \ref{DefDO}.
We let $K_{T_{\nu,u}+C(B_\nu)}$ be the integral kernel of
$e^{-t(T_{\nu,u}+C(B_\nu))}$.  
Then the following proposition is due to M\"uller. 
\begin{prop}\label{Propreltr}
Let $B_\nu$ be locally homogeneous at infinity with  $C(B_\nu)$ as in
Definition 
\ref{DefDO}.
Then the operator $e^{-tB_\nu}-e^{-t(T_{\nu,u}+C(B_\nu))}$ is trace class and
one has 
\begin{align*}
\Tr\left(e^{-tB_\nu}-e^{-t(T_{\nu,u}+C(B_\nu))}\right) = \int_{X}\left(\Tr
K_{B_\nu}(t,x,x)-\Tr K_{T_{\nu,u}+C(B_\nu)}(t,x,x)\right)dx, 
\end{align*}
where the expression on the right is absolutely integrable. Here the 
integral is taken with respect to the volume induced by the 
metric $g_1$ implicit in the definition of $B_\nu$. 
\end{prop}
\begin{proof}
\cite[Theorem 9.1]{Mulecn}. 
\end{proof}

According to the previous proposition, we now define the relative trace of
$e^{-tB_\nu}$ with respect to the
parameter $u$ by 
\begin{align*}
\Tr_{\Rel,u}(e^{-tB_\nu}):=\Tr(e^{-tB_\nu}-e^{-t(T_{\nu,u}+C(B_\nu))}).
\end{align*}

Then using the previous computations one can
easily 
compare the regularized trace, introduced previously,  
and the relative trace: 
\begin{prop}\label{PropRelReg}
Let $A_\nu$  be as above and let $u>Y_0$. Then one has 
\begin{align*}
\Tr_{\Rel;u}(e^{-tA_\nu})=\Tr_{\reg}(e^{-tA_\nu})+\kappa(\Gamma)\sum_{\substack{
\sigma\in\hat{M}\\
\left[\nu:\sigma\right]\neq
0}}e^{tc(\sigma)}\dim(\sigma)\left(\frac{\log{u}}{\sqrt{4\pi
t}}+\frac{1}{4}\right).
\end{align*}
\end{prop}
\begin{proof}
By the preceding Proposition \ref{Propreltr} one has 
\begin{align*}
\Tr_{\Rel;u}(e^{-tA_\nu})=\lim_{Y \to\infty}\left(\int_{X(Y)}\Tr
K_{A_\nu}(t,x,x)dx-\int_{X(Y)}\Tr K_{T_{\nu,u}}(t,x,x)dx\right).
\end{align*}
By Lemma \ref{LemT1}, Lemma \ref{LemT2} and
Lemma \ref{LemT3} one has 
\begin{align*}
\int_{X(Y)}\Tr K_{T_{\nu,u}}(t,x,x)dx=&\sum_{\substack{\sigma\in\hat{M}\\
\left[\nu:\sigma\right]\neq
0}}\biggl(\frac{\kappa(\Gamma) e^{tc(\sigma)}\log{Y}\dim(\sigma)}{\sqrt{4\pi
t}}\biggr)\\ -&\kappa(\Gamma)\sum_{\substack{
\sigma\in\hat{M}\\
\left[\nu:\sigma\right]\neq
0}}e^{tc(\sigma)}\dim(\sigma)\left(\frac{\log{u}}{\sqrt{4\pi
t}}+\frac{1}{4}\right)+o(1),
\end{align*}
as $Y\to\infty$. 
Applying \eqref{asymexp}, the proposition follows. 
\end{proof}

We have the following short time asymptotic expansions 
for the regularized and relative traces.

\begin{prop}\label{Propasepx}
One has an asymptotic expansion 
\begin{align*}
\Tr_{\reg}(e^{-tA_{\nu}})\sim
\sum_{j=0}^\infty a_{j}t^{j-\frac{d}{2}}+\sum_{j=0}^\infty b_{j}t^{
j-\frac{1}{2}}\log{t}+\sum_{j=0}^\infty c_j t^j
\end{align*}
as $t\to 0+$. 
Moreover, if $B_\nu$ is locally homogeneous at infinity one 
has an asymptotic expansion 
\begin{align}\label{eqasrel}
\Tr_{\Rel;u}(e^{-tB_{\nu}})\sim
\sum_{j=0}^\infty
\tilde{a}_{j}t^{j-\frac{d}{2}}+\sum_{j=0}^\infty\tilde{b}_{j}t^{
j-\frac{1}{2}}\log{t}+\sum_{j=0}^\infty \tilde{c}_j t^j,
\end{align}
as $t\to 0+$.
\end{prop}
\begin{proof}
The short-time asymptotic expansion of the regularized trace
$\Tr_{\reg}(e^{-tA_{\nu}})$ was proved 
in \cite[Proposition 6.9]{MP1}. Applying proposition \ref{PropRelReg}, the
asymptotic expansion 
in \eqref{eqasrel} follows for the relative trace $\Tr_{\Rel;u}(e^{-tA_{\nu}})$.
For the asymptotic expansion of $\Tr_{\Rel;u}(e^{-tB_{\nu}})$, we can 
assume that the constant $C(B_\nu)$ in definition
\ref{DefDO} is zero. 
Let $Y_1>Y_0$ be such that the operators $A_\nu$ and $B_\nu$ coincide 
on sections supported on $X-X(Y_1)$.
Then by Proposition \ref{Propreltr} one has 
\begin{align*}
\Tr_{\Rel;u}(e^{-tA_{\nu}})=&\int_{X(Y_1+1)}\Tr
K_{A_\nu}(t,x,x)dx-\int_{X(Y_1+1)}\Tr K_{T_{\nu,u}}(t,x,x)dx\\ +&
\int_{X-X(Y_1+1)}(\Tr K_{A_\nu}(t,x,x)-\Tr
K_{T_{\nu,u}}(t,x,x))dx,
\end{align*}
Integrating the pointwise short-time asymptotic expansion of 
$\Tr K_{A_\nu}(t,x,x))$ \cite{Gilkey}, resp. of $\Tr K_{T_{\nu,u}}(t,x,x)$, over
the
relatively 
compact subset $X(Y_1+1)$ of $X$ and using the short time asymptotic expansion
of 
$\Tr_{\Rel;u}(e^{-tA_{\nu}})$, it follows
that one has 
a short-time asymptotic expansion
\begin{align}\label{asex1}
&\int_{X-X(Y_1+1)}(\Tr K_{A_\nu}(t,x,x)-\Tr
K_{T_{\nu,u}}(t,x,x))dx\nonumber \\ &\sim \sum_{j=0}^\infty
a_jt^{j-\frac{d}{2}}+\sum_{j=0}^\infty b_jt^{j-1/2}\log{t}+\sum_{j=0}^\infty
c_jt^j, 
\end{align}
as $t\to 0+$. Now by Proposition \ref{Propreltr} one has 
\begin{align}\label{asex2}
\Tr_{\Rel,u}(e^{-tB_\nu}) 
=&\int_{X(Y_1+1)}\Tr K_{B_\nu}(t,x,x)dx-\int_{X(Y_1+1)}\Tr
K_{T_{\nu,u}}(t,x,x)dx\nonumber\\ +&
\int_{X-X(Y_1+1)}(\Tr K_{B_\nu}(t,x,x)-\Tr K_{A_\nu}(t,x,x))dx \nonumber \\ +&
\int_{X-X(Y_1+1)}(\Tr K_{A_\nu}(t,x,x)-\Tr
K_{T_{\nu,u}}(t,x,x))dx,
\end{align}
where $dx$ in the first integral of the first line stands for the volume with
respect
to the metric $g_1$ implicit in the definition of $B_\nu$. 
Again, by the pointwise short-time asymptotic expansion of the term
$\Tr K_{B_\nu}(t,x,x)$ and of the term $\Tr K_{T_{\nu,u}}(t,x,x)$, 
the first integral on the right hand 
side of \eqref{asex2} admits a short time asymptotic expansion (without
logarithmic 
terms). The second integral on the right hand side of \eqref{asex2} is
$O(t^\infty)$ for 
$t\to 0$ by Proposition \ref{Estheatk}. Applying \eqref{asex1} to the 
third integral on the right hand side of \eqref{asex2}, the proposition
follows.
\end{proof}

Now let $\rho\in\Rep(G)$ and let $E_\rho$ be the induced flat vector bundle over
$X$  with 
the metric from section \ref{secDR} and let $\Delta_p(\rho)$ be 
the associated p-th flat Hodge Laplacian as in section \ref{secDR}. 
Then by the isomorphism \eqref{IsoMats} and by \eqref{kuga},  
the bundle $E_\rho$ and the operator $\Delta_p(\rho)$ fit into the class of
bundles and operators just studied. 
In particular, if we let $T_{p,u}^\rho:=T_{\nu_p(\rho),u}+\rho(\Omega)\Id$, 
$\nu_p(\rho)$ as 
in \eqref{nutau} and $\rho(\Omega)\in\R$ the Casimir eigenvalued of $\rho$, then
the regularized trace $\Tr_{\reg}(e^{-t\Delta_p(\rho)})$ and 
the relative trace
$\Tr_{\Rel,u}(e^{-t\Delta_p(\rho)})=\Tr(e^{-\Delta_p(\rho)}-e^{-t
T_{p,u}^\rho})$, are defined and satisfy
the above properties.  For brevity, according to definition \ref{DefDO} we call 
a second order elliptic differential operator $\Delta_p^1(\rho)$ acting on the
smooth sections of $\Lambda^pE_\rho$ a compact perturbation of $\Delta_p(\rho)$ 
if $\Delta_p^1(\rho)$ is formally symmetric and bounded from below with respect 
to the $L^2$-space induced by some fixed compact perturbations 
of the hyperbolic metric on $X$ resp. the admissible metric on $E_\rho$ 
and if there exists a $Y_1>Y_0$
such that one has $\Delta_p^1(\rho)f=\Delta_p(\rho)f$ for all smooth sections
$f$ of $E_\rho$ which vanish on $X(Y_1)$. Since 
$X$ is complete, $\Delta_p^1(\rho)$ is essentially selfadjoint \cite{Ch} and 
its closure will be denoted by the same symbol. 
If $\rho\neq\rho_\theta$,  we
have the following long 
time estimate for the regularized and relative trace.

\begin{prop}\label{Longtime}
Let $\rho\in\Rep(G)$ such that $\rho\neq\rho_{\theta}$ and let 
$p\in\{0,\dots,d\}$. Then there exists
a constant $C>0$ such that for each $p$ and for $t\in [1,\infty)$ one has
\begin{align}\label{eslt1}
\Tr_{\reg}(e^{-t\Delta_p(\rho)})\leq Ce^{-\frac{t}{4}};\quad 
\Tr_{\Rel,u}(e^{-t\Delta_p(\rho)})\leq C e^{-\frac{t}{4}}. 
\end{align}
Let $\Delta_p^1(\rho)$ be a compact perturbation of $\Delta_p(\rho)$. 
Then $\Delta_p^1(\rho)$ has 
pure point spectrum in $[0,1/4)$ and there exist constants $c_1, C_1>0$ 
such that one has
\begin{align}\label{eslt2}
\Tr_{\Rel,u}(e^{-t\Delta_p^1(\rho)})-\dim \Ker(\Delta_p^1(\rho)) \leq C_1
e^{-c_1t}. 
\end{align}
for $t\geq 1$.
\end{prop}
\begin{proof}
The estimate for the regularize trace of $e^{-t\Delta_p(\rho)}$ was proved in 
\cite[7.12]{MP1}.
By \cite[Lemma 7.1]{MP1} one has $\rho(\Omega)-c(\sigma)\geq 1/4$ for each
$\sigma\in\hat{M}$ with $[\nu_p(\rho):\sigma]\neq 0$, 
where $\nu_p(\rho)$ is as in \eqref{nutau}.  
Therefore, applying Proposition \ref{PropRelReg}, also the estimate for the
relative trace of $e^{-t\Delta_p(\rho)}$ follows. 
It was proved in \cite{MP1}, \cite{MP2} that the spectrum
$\sigma\left(\Delta_p(\rho)\right)$ of $\Delta_p(\rho)$ 
satisfies $\sigma\left(\Delta_p(\rho)\right)\subseteq [\frac{1}{4},\infty)$, see
remark \ref{BemerkMP}.  Since the metrics 
$g_1$ and $h_1$, implicit in the definition of $\Delta_p^1(\rho)$ as in
Definition \ref{DefDO}, 
are compact perturbations of $g$ and $h$, the underlying Hilbert 
spaces of $L^2$-sections can be identified by a bounded map $I$ which 
has bounded inverse and which moreover commutes with the operator
$e^{-tT_{p,u}^\rho}$.  
In this way we can fix the Hilbert space. In particular, 
$e^{-t\Delta_p^1(\rho)}-e^{-t\Delta_p(\rho)}$ is 
trace class by Proposition \ref{Propreltr}, and thus it follows from the
invariance 
of the essential spectrum under compact perturbations that $\Delta_p^1(\rho)$
has pure point spectrum in $[0,1/4)$. 
Moreover, it follows from \cite[Lemma 2.2]{Murel} that there 
exist constants $c_1,C_1>0$ such that
\begin{align*}
\Tr (e^{-t\Delta_p^1(\rho)}-e^{-t\Delta_p(\rho)})-\dim
\Ker(\Delta_p^1(\rho)) \leq C_1
e^{-c_1t}
\end{align*}
for $t\geq 1$. Since 
\begin{align*}
&\Tr_{\Rel,u}(e^{-t\Delta_p^1(\rho)})-\dim \Ker(\Delta_p^1(\rho))\\ =&\Tr
(e^{-t\Delta_p^1(\rho)}-e^{-t\Delta_p(\rho)})-\dim
\Ker(\Delta_p^1(\rho))+\Tr_{\Rel,u}(e^{-t\Delta_p(\rho)}),
\end{align*}
also the second estimate \eqref{eslt2} follows. 
\end{proof}

Let $\Delta_p^1(\rho)$ be a compact perturbation of $\Delta_p(\rho)$. Then,
modeling the corresponding definition on a compact manifold, we 
can define the relative zeta function of $\Delta_p(\rho)$ by 
\begin{align*}
\zeta_{\Rel,u}(\Delta_p^1(\rho);s):=\frac{1}{\Gamma(s)}\int_0^\infty
t^{s-1}(\Tr_{ \Rel,u}\Delta_p^1(\rho)-\dim \Ker (\Delta_p^1(\rho)))ds.
\end{align*}
By the preceding two propositions, the integral is absolutely convergent for
$\Real(s)>d/2$ and using standard arguments, see for example \cite{Gilkey},
these
propositions also imply that it 
admits a meromorphic continuation to $\C$ which is reguar at $s=0$.  
In the same way, we can define the regularized zeta function 
$\zeta_{\reg}(\Delta_p(\rho);s)$ using the regularized trace, see
\cite[Definition 7.13]{MP1}. Thus, if $g_1$ and $h_1$ are 
compact perturbations of the metric $g$ on $X$ and $h$ on $E_\rho$ and if
$\Delta_p^1(\rho)$ are 
the corresponding flat Hodge-Laplacians on $E_\rho$, one can define the 
relative analytic torsion generalizing the originial definition of Ray
and Singer \cite{RS} for compact manifolds:
\begin{align}\label{DefRelTors}
\log{T_{\Rel,u}}(X,g_1,h_1;E_\rho):=\frac{1}{2}\sum_p(-1)^p
p\frac{d}{ds}\biggr|_{s=0}\zeta_{\Rel,u}(\Delta_p^1(\rho);s),
\end{align}
where we have put $g_1$ and $h_1$ in the notation in order to 
indicate that, a priori, this torsion depends on these metrics. It will be 
shown in Proposition \ref{Propkohometr} that this is not the case for our
present 
strongly acylic case $\rho\neq\rho_\theta$. 

We finally remark that one can of course also define the regularized heat trace
and 
the regularized analytic torsion for 
any compact perturbation $\Delta_p^1(\rho)$ of the flat Hodge Laplacian 
$\Delta_p(\rho)$, $\rho\neq\rho_\theta$, in exactly the 
same way as one defines it for $\Delta_p(\rho)$. More precisely, by
Proposition \ref{Propreltr}, the 
difference of the fibre traces
$\Tr K_{\Delta_p^1(\rho)}(t,x,x)-\Tr K_{\Delta_p(\rho)}(t,x,x)$
is integrable over $X$ and thus by \eqref{asymexp} the integral of $\Tr
K_{\Delta_p^1(\rho)}(t,x,x)$
over $X(Y)$ has again an asymptotic expansion in $Y$ whose constant term one 
defines to be the regularized trace $\Tr_{\reg}(e^{-t\Delta_p^1(\rho)})$. 
Moreover, the above 
constructions immediately show that the difference 
$\Tr_{\Rel,u}(e^{-t\Delta_p^1(\rho)})-\Tr_{\reg}(e^{-t\Delta_p^1(\rho)})$
equals 
the difference
$\Tr_{\Rel,u}(e^{-t\Delta_p(\rho)})-\Tr_{\reg}(e^{-t\Delta_p(\rho)})$ 
and by Proposition \ref{PropRelReg} this 
difference, and therefore also the difference of the two torsions, is a
completely 
explicit function of the auxilliary parameter $u$. In summary, one can 
always pass back and forth between the regularized and relative objects.

\section{Some properties of the heat kernels on locally homogeneous vector
bundles}\label{secPropheatk}
\setcounter{equation}{0}
In this section we firstly prove the trace class 
property for compactly supported differential operators applied to the heat
kernel in order to treat the dependence of the torsion on compact perturbations
of the metric later. 
Then we prove that for a fixed compact subset of $X$ the heat kernels we
work 
with can be approximated in all derivatives by a sequence of heat kernels on
closed manifolds following 
and argument of L\"uck, Schick and Bunke.

We let $X$ be the same manifold as in section \ref{subsmfld} and we let 
$E_\nu$ be a locally homogeneous vector bundle over $X$. 
We let $D$ be a differential operator of order $\leq 2$ which acts on smooth
sections
of $E_\nu$ and
which 
has compact support. By this we mean that there exists a compact set
$\mathcal{V}$ such 
that $D\phi=0$ for all smooth sections $\phi$ of $E_\nu$ which 
vanish on $\mathcal{V}$. Then we have the following proposition.

\begin{prop}\label{TrDcomps}
Let $B_\nu$ be a differential operator acting on the smooth sections 
of $E_\nu$ which is locally homogeneous at infinity. 
Then the operators $D\circ e^{-tB_\nu}$ and $e^{-tB_\nu}\circ D$ are trace
class and their trace norm 
is uniformly bounded for $t$ in compact subsets of $(0,\infty)$. 
\end{prop}
\begin{proof}
Let $r$ be the function on $X$ defined in the previous section. 
We remark that for $C_2>0$, $T>0$ there exists a constant $C>0$ such that 
for all $x,y\in X$ and all $t\in (0,T]$ one can estimate 
\begin{align}\label{Estr}
r(x)^nr(y)^ne^{-C_2\frac{d^2(x,y)}{t}}\leq
Cr(x)^{3n}e^{-C_2\frac{d^2(x,y)}{2t}}.
\end{align} 
If $r(y)\leq r(x)^{2}$, this is clear 
and if $r(y)>r(x)^{2}$, then $d(x,y)\geq\frac{\log{r(y)}}{2}$ by the 
definition of the hyperbolic metric and 
the estimate also follows.

Now let $M_{r^{3n}}$, $M_{r^{-3n}}$ be the
operators
which
arise by 
multiplication with $r^{3n}$ resp. $r^{-3n}$. Then we have 
\[
D\circ e^{-tB_\nu}=(D\circ e^{-\frac{t}{2}B_\nu} \circ
M_{r^{3n}})\circ(M_{r^{-3n}}\circ e^{-\frac{t}{2}B_\nu}).
\]
It suffices to show that each of the operators in brackets is Hilbert-Schmidt
with 
Hilbert-Schmidt norm bounded on compact subsets of $(0,\infty)$. 
Let $I_1:=D\circ e^{-\frac{t}{2}B_\nu} \circ
M_{r^{3n}}$.
Then $I_1$ is an integral operator 
with smooth kernel
$K_1(t,x,y):=D_x K_{B_\nu}(t/2,x,y)r^{3n}(y)$,
where $D_x$ indicates that we apply $D$ to the $x$-variable. 
Let $\mathcal{V}$ be the support of $D$.  
Since $B_\nu$ is elliptic and of second order, by the local G{\aa}rding
inequality
\cite{Gilkey} 
there exists a constant $C>0$ such that 
\[
\int_X \left\|D f (x)\right\|^2 dx=\int_{\mathcal{V}} \left\|D f
(x)\right\|^2\leq
C\left(\int_{\mathcal{V}}\left\|f(x)\right\|^2dx+\int_{\mathcal{V}}
\left\|B_\nu f(x)\right\|^2dx \right)
\]
for each smooth section $f$ of $E_\nu$. 
Therefore we get
\begin{align*}
&\left\|I_1\right\|_{HS}^2=\int_{X\times X} \left\|K_1(t,x,y)\right\|^2
dxdy\leq 
C\int_{\mathcal{V}\times
X}\left\|K_{B_\nu}(t/2,x,y)\right\|^2r^{6n}(y)dxdy\\
+&C\int_{\mathcal{V}\times X}
\left\|(B_\nu)_xK_{B_\nu}(t/2,x,y)\right\|^2r^{6n}(y)dxdy,
\end{align*}
where $(B_\nu)_x$ indicates that we apply $B_\nu$ to the $x$-variable. 
There exists a 
compact set $\tilde{\mathcal{V}}$ of $X$  such that for all $x\in\mathcal{V}$
and all $y\in X-\tilde{\mathcal{V}}$ one 
can estimate $d(x,y)\geq \frac{\log{r(y)}}{2}$. Therefore Proposition 
\ref{Estheatk}
immediately implies that for $q\in\{0,1\}$ the norm 
$\left\|(B_\nu^q)_xK_{B_\nu}(t,x,y)\right\|^2r^{6n}(y)$ is bounded on
$\mathcal{V}\times X$ uniformly
for $t$ in compact subsets of $(0,\infty)$ and since $X$ is of finite 
volume, $\left\|I_1\right\|_{HS}^2$ is finite 
and locally uniformly bounded in $t$. 

Next, the operator
$I_2:=M_{r{-3n}}\circ e^{-\frac{t}{2}B_\nu}$
is an integral-operator with smooth kernel 
$K_2(t,x,y):=r^{-3n}(x)K_{B_\nu}(t/2,x,y)$
and by Proposition \ref{Estheatk} and \eqref{Estr} this kernel is bounded on $X$
locally unifomrly in $t$. 
Since $X$ has finite volume, the same therefore holds for the Hilbert Schmidt
norm of $I_2$. 
This proves the proposition for $D\circ e^{-tB_\nu}$. Since 
$e^{-tB_\nu}\circ D$ is 
the adjoint of $D^*\circ e^{-tB_\nu}$, the proposition follows also 
for $e^{-tB_\nu}\circ D$. 
\end{proof} 

We let $X$ be the same manifold as in section \ref{subsmfld} equipped 
with a metric $g_1$ which is a compact perturbation of the hyperbolic metric. 
Let $Y_0$ 
be as above. Then for $Y> Y_0$ 
we let $M(Y)$ be the closed orieted manifold which is the double of $X(Y+1)$: 
\begin{align}\label{DefM}
M(Y):=X(Y+1)\sqcup_{\partial X(Y+1)} -X(Y+1),
\end{align}
where $-X(Y+1)$ denotes the manifold $X(Y+1)$ with the reversed
orienation. 
Taking a metric on the manifold with boundary $X(Y+1)$, which 
coincides with the metric $g_1$ in a neighbourhood of $X(Y)$ in 
$X(Y+1)$ and which is of product structure in a neighbourhood of 
the boundary $\partial X(Y+1)$, we can 
equipp $M(Y)$ with a metric which 
coincides with the hyperbolic metric in a neighbourhood of $X(Y)\subset M(Y)$
and 
of $-X(Y)$ in $M(Y)$. 
We let $E$ be a smooth Hermitian vector bundle over $X$ and we let $\Delta[X]$ 
be a differential operator which acts on the smooth sections of $E$ and 
which is of Laplace-type. We fix a smooth Hermitian vector bundle 
$E'$ over $M(Y)$ such that the restrictions of $E$ and $E'$ to $X(Y)$ are
isometric. We shall from now on identify these bundles over $X(Y)$. If $E$ is 
flat, i.e. a bundle associated to a finite-dimensional representation 
of the fundamental group of $X(Y)$, then also $E'$ can be chosen 
to be flat by Van Kampen's theorem. 
We fix an operator $\Delta[M(Y)]$ of Laplace type which 
acts on the smooth sections of $E'$ and which has the property that it 
coincides with $\Delta[X]$ on smooth sections which are supported in the 
interior of $X(Y)$. 

The operators $\Delta[M(Y)]$ and $\Delta[X]$ with domain the smooth (compactly 
supported) sections of the corresponding bundles are essentially selfadjoint 
and their closes will be denoted by the same symbol. We let $e^{-t\Delta[X]}$
and
$e^{-t\Delta[M(Y)]}$ be 
their heat semigroups on the spaces of square integrable sections 
which act as integral operators with smooth kernels
\begin{align*}
K_{\Delta[M(Y)]}(t,x,y)\in C^\infty(M(Y)\times M(Y),E'\boxtimes (E')^*);\:
K_{\Delta[X]}(t,x,y)\in C^\infty(X\times X,E \boxtimes E^*).
\end{align*}
In the next proposition we show that for a fixed $Y_1>Y_0$ all derivatives of 
$K_{\Delta[X]}(t,x,y)$ can be approximated uniformly on $X(Y_1)\times X(Y_1)$ by
the corresponding derivatives of the 
kernels $K_{\Delta[M(Y)]}(t,x,y)$ as $Y\to\infty$. 
We use an argument of L\"uck and Schick, \cite[Theorem 2.26]{LS}, which they
attribute to Bunke.

\begin{prop}\label{Propapproxim}
Let $D$ be a differential operator which acts on the smooth sections of $E$. Let
$T>0$, $Y_1>Y(\Gamma_0)$.  
Then there exist constants $C, c>0$ which depend only on $D$, $T$, $Y_1$ and $X$
such that for all $t\in(0,T)$, all $Y>Y_1$ and all $x_0,y_0\in X(Y_1)$ 
one has
\[
\left\|\frac{d^i}{dt^i}D_xK_{\Delta[X]}(t,x_0,y_0)-\frac{d^i}{dt^i}
D_xK_{\Delta[M(Y)]}(t,x_0,y_0)\right\|\leq C
e^{-\frac{c\dist^2(X(Y_1),\partial X(Y))}{t}},
\]
where $i\in\mathbb{N}$ and where
$D_x$ indicates that 
we apply $D$ to the first variable.  
\end{prop}
\begin{proof}
Let $\Delta:=\Delta[M(Y)]$ or $\Delta=\Delta[X]$. 
By the local Sobolev embedding
theorem and by the local 
G{\aa}rding inequality \cite{Gilkey}, for $j\in\mathbb{N}$ there exists a
constant 
$C_j(Y_1)$ which depends only on $j$ and  $Y_1$ 
such that for all differential operators $D$ of order $\mu\in\mathbb{N}^0$ as in
the proposition, for all $j\in\mathbb{N}$ with $j>d/4+\mu/2$,
for all smooth sections $u$ of $E$ with 
compact support in $X(Y_1)$ and for
all $x\in X(Y_1)$ one has 
\begin{align}\label{Garding} 
\left\|Du(x)\right\|\leq C(D) C_j(Y_1)\left(\left\|u\right\|_{
L^2(X(Y_1);E)}+\left\|\Delta^ju\right\|_{
L^2(X(Y_1);E)}\right),
\end{align}
where $C(D)$ is a constant that depends only on $D$. 
This implies in particular that it suffices to 
prove the proposition for $D=\Delta^k$, $k\in\mathbb{N}$. Then one can assume
that $i=0$. 
We can now argue exactly as L\"uck and Schick in the proof of \cite[Theorem
2.26]{LS}. We assume that $Y\geq Y_1+2$. Let $u$ be a smooth sections with
compact support in the interior of $X(Y_1)$. 
Then 
as in \cite{LS}, we consider the section 
\[
f:=(\Delta[M(Y)]^ke^{-t\Delta[M(Y)]}-\Delta[X]^ke^{-t\Delta[X]}
)u, 
\]
which is defined on $X(Y)$, in particular on $X(Y_1)$. 
Let $m,l\in\mathbb{N}^ß$. As in \cite{LS}, if $\Delta=\Delta[X]$ or
$\Delta=\Delta[M(Y)]$ the
spectral theorem implies that restricted to $L^2(X(Y_1);E)$ one has
\begin{align}\label{eqwave}
&\Delta^m\Delta^lf\nonumber\\ 
=&\int_{0}^\infty
t^{-2(m+l+k)-\frac{1}{2}}P_{m,l,k}(s,\sqrt{t})e^{-\frac{s^2}{4t}}
(\cos(s\sqrt{
\Delta[M(Y)]})-\cos(s\sqrt{\Delta[X]}))u ds,
\end{align}
where $P_{m,l,k}$ is a universal polynomial depending only on $m,l$ and $k$.  
By unit
propagation speed of the wave equation \cite{Taylor}, for
$s\in(0,\infty)$ one has
$\supp\cos(s\sqrt{\Delta[X]})u\subset\{x\in X\colon \dist(x,\supp
u)\leq s\}$
and 
$\supp\cos(s\sqrt{\Delta[M(Y)]})u\subset\{x\in M(Y)\colon
\dist(x,\supp u)\leq s\}$.
Using the uniqueness of 
solutions of the wave equation with inital datum $u$, it follows that 
\begin{align}\label{EqW}
\cos(s\sqrt{\Delta[X]})u=\cos(s\sqrt{\Delta[M(Y)]})u;\quad\text{if
$s\leq\dist(X(Y_1),\partial X(Y))$}. 
\end{align}
Since $\cos{s\Delta}$ is of norm $\leq 1$, \eqref{eqwave} and \eqref{EqW} imply
that
\begin{align}
\left\|\Delta^m\Delta^lf\right\|_{L^2(X(Y_1);E)}\leq
C_{k,l,m}e^{-c_{k,l,m}\frac{\dist(X(Y_1),\partial X(Y))^2}{t}}\left\|u\right\|_{
L^2(X(Y_1);E)},
\end{align}
where $c_{k,l,m}$ and  $C_{k,l,m}$ are universal constant that depend only on
$k,l,m$.
If we apply \eqref{Garding} choosing $m$ appropriately, it follows that
there 
exists a constant $C_{k,l}(Y_1)$ which depends only on $k,l$ and $Y_1$ and a 
constant $c_{k,l}$, which depends only on $k$ and $l$, such that 
for $x\in
X(Y_1)$ one has 
\begin{align}
\left\|\Delta^lf(x)\right\|\leq 
C_{k,l}(Y_1)e^{-c_{k,l}\frac{\dist(X(Y_1),\partial
X(Y))^2}{t}}\left\|u\right\|_{
L^2(X(Y_1);E)}.
\end{align}
In other words, for $x\in X(Y_1)$ one has 
\begin{align}
&\|\int_{X(Y_1)}(\Delta^l_x\Delta^k_xK_{\Delta[M(Y)]}(t,x,
y)-\Delta^l_x\Delta _x^kK_{\Delta[X]}
(t,x,y))u(y)dy\| \nonumber\\
&=\|\int_{X(Y_1)}(\Delta^l_y\Delta^k_xK_{\Delta[M(Y)]}(t,x,
y)-\Delta^l_y\Delta
_x^kK_{\Delta[X]}
(t,x,y))u(y)dy \| \nonumber\\ &\leq 
C_{k,l}(Y_1)e^{-c_{k,l}\frac{\dist(X(Y_1),\partial
X(Y))^2}{t}}\left\|u\right\|_{
L^2(X(Y_1);E)}.
\end{align}
Since the smooth sections $u$ with compact interior support are dense in
$L^2(X(Y_1);E)$,
it 
follows that for each $x\in X(Y_1)$ one has
\begin{align*}
\left\|\Delta^l_y\Delta^k_xK_{\Delta[M(Y)]}(t,x,
-)-\Delta^l_y\Delta _x^kK_{\Delta[X]}
(t,x,-)\right\|_{L^2(X(Y_1);E)}\leq 
C_{k,l}(Y_1)e^{-c_{k,l}\frac{\dist(X(Y_1),\partial X(Y))^2}{t}}.
\end{align*}
Choosing $l$ appropriately and applying \eqref{Garding} again, it
follows that there exists a constant $C_k(Y_1)$ which 
depends only on $k$ and $Y_1$ and a constant $c_k$ which 
depens only on $k$ such that
all $x,y\in X(Y_1)$ one has
\begin{align*}
\left\|\Delta^k_xK_{\Delta[M(Y)]}(t,x,
y)-\Delta _x^kK_{\Delta[X]}
(t,x,y)\right\|\leq C_k(Y_1)e^{-c_{k}\frac{\dist(X(Y_1),\partial
X(Y))^2}{t}}
\end{align*}
and the proposition is proved. 
\end{proof}

\section{The variation of the De Rham complex after Vishik and
Lesch}\label{secvar}
\setcounter{equation}{0}
In this section we introduce the variation of the De Rham complex, 
due to Vishik and Lesch, fort the present case. 

To keep this article reasonably self-contained, let us briefly recall some of
those concepts and results of Br\"uning and Lesch 
\cite{BL} which we shall use in this article. Firstly, a Hilbert 
complex is a collection of Hilbert spaces $H_i$, $i=0,\dots,n$,
$n\in\mathbb{N}$ 
together with closed (in general unbounded)
operators $D_i$ with domain $\dom(D_i)\subseteq H_i$ 
and with $D_i:\dom(D_i)\to H_{i+1}$ such 
that $\mathcal{R}_i:=D_i(\dom(D_i)) \subset \dom(D_{i+1})$ and such 
that $D_{i+1}\circ D_i=0$. We assume that each $D_i$ is densely defined. 
Let $D_i^*$ be the adjoint of $D_i$. Then the $i$-th Laplacian $\Delta_i$ of the
Hilbert 
complex is defined as 
\begin{align*}
\Delta_i:=D_i^*\circ D_i+D_{i-1}\circ D_{i-1}^*, 
\end{align*}
where we use the usual conventions for the domains of sums and compositions 
of unbounded operators. Using the Gauss-Bonnet operator of the complex \cite{BL}
and 
von Neumann's theorem, one can show that each operator  $\Delta_i$
is selfadjoint. We denote by $\mathcal{H}_i:=\ker(D_i)/\mathcal{R}_{i-1}$ the
cohomology  of 
the Hilbert complex. A Hilbert 
complex is called a Fredholm complex if $\mathcal{R}_i$ is closed 
and if $\mathcal{H}_i$ is finite-dimensional.
Let $\hat{\mathcal{H}}_i:=\ker(\Delta_i)$. One has 
$\hat{\mathcal{H}}_i=\Ker(D_i)\cap\ker(D_{i-1}^*)$ and thus one has a canonical
map $\hat{\mathcal{H}}_i\to \mathcal{H}_i$.  
In all the situations that occur in the present article, this map will  
be an isomorphism by the strong Hodge decomposition 
theorem for Fredholm complexes due to  Br\"uning and Lesch: 
\begin{prop}\label{PropFr}
A Hilbert complex is a Fredholm complex if and only if 
$0$ does not belong to the essential spectrum of $\Delta_i$ for each $i$. For a 
Fredholm complex the canonical map from $\hat{\mathcal{H}}_i$ to $\mathcal{H}_i$
is an
isomorphism for each $i$. 
\end{prop}
\begin{proof}
\cite[Theorem 2.4, Corollary 2.5]{BL} .  
\end{proof}

The Hilbert complexes we work with will arise in the following context; we of 
course refer again to \cite{BL} for a much broader treatment. 
We let $M$ be a 
smooth Riemannian manifold with metric $g$ and with smooth boundary $\partial
{M}$, which might be empty. 
We
let $F$ 
be a flat Hermitian vector bundle over $M$ with Hermitian fibre metric $h$,
which 
we do not necessarily assume to be flat.  
We let $\Lambda^pF:=\Lambda^pT^*M\otimes F$. The underlying Hilbert spaces are 
the spaces $L^2(M;\Lambda^pF)$ 
of square integrable section of $\Lambda^pF$.  
If $M$ is not compact, then these spaces depend on the metrics $g$ and $h$. 
Let $\Omega_c^p(M;F)$ 
be the set of smooth $F$-valued $p$-forms on $M$ with compact support in the
interior
of $M$
and 
let $d_p:\Omega_c^p(M;F)\to \Omega_c^{p+1}(M;F)$ be 
the exterior derivative. Then one obtains a complex 
$(\Omega_c^*(M:F),d)$. In order to obtain a Hilbert complex, one 
needs to specify suitable closed extensions of the $d_p$. 
Thus for each $p$ let 
$d_p^t$ be the transposed differential operator of $d_p$, i.e. $d_p^t:
\Omega_c^{p+1}(M;F)\to \Omega_c^{p}(M;F)$ is the 
unique first order differential operator which satisfies
$\left<d_p^t\phi,\psi\right>_{L^2(M;\Lambda^p
F)}=\left<\phi,d_p\psi\right>_{L^2(M;\Lambda^{p+1}F)}$ 
for all $\phi\in\Omega_c^{p+1}(M;F)$ and all $\psi\in\Omega_c^p(M;F)$. 
Then one defines the maximal extension of $d_p$ by
$d_{p;\max}:=(d_p^t)^*$, which is closed. Then $\dom(d_{P;\max})$ is set of all
$\omega\in{L^2(M;\Lambda^{p}F)}$ for which
$d\omega\in{L^2(M;\Lambda^{p+1}F)}$, 
where $d\omega$ is meant in the distributional sense. 
One
defines the
minimal extension of $d_p$ as its closure:
$d_{p;\min}:=\overline{d_p}$.
We let $\Omega_{\max/\min}^p(M;F):=\dom(d_{p;\max/\min})$. Then by \cite[Lemma
3.1]{BL} these spaces together with the operators $d_{p;\max/\min}$ define 
Hilbert complexes $\mathcal{D}^*_{\max/\min}(M;F)$. Any Hilbert complex lying 
between $\mathcal{D}^*_{\min}(M;F)$ and $\mathcal{D}^*_{\max}(M;F)$ is called 
an ideal boundary condition. 
\newline

Now we let $X$ be the same manifold as in
section
\ref{subsmfld}. We let 
$Y_0$ be as in that section and we let $Y_1>Y_0+2$. 
We have an inclusion $X(Y_1)\subset X$ and  we fix a metric $g_1$ on $X$
which is of product structure on a neighourhood of $X(Y_1+1/2)\backslash
X(Y_1-1/2)$ in $X$  and 
which coincides with the hyperbolic metric $g$ of $X$ on 
the set $X(Y_1-1)$ and on the set $X\backslash X(Y_1+1)$. 
Thus $g_1$ is a compact perturbation of $g$ which is of product structure 
on some fixed cusp pieces. 
To save notation, we write
\begin{align}\label{notxpm}
X^+:=X(Y_1);\quad X^-:=X\backslash
X^+=\bigsqcup_{P\in\mathfrak{P}}F_P(Y_1);\quad
Z:=\partial
X^+. 
\end{align}
We equip $X^\pm$ with the metric $g_1$. We let $\rho\in\Rep(G)$ such that 
$\rho\neq\rho_\theta$. For brevity, we let $F:=E_\rho$ denote 
the corresponding flat vector bundle over $X$, where $E_\rho$ is as in section
\ref{secDR}. We let $h$ be the metric 
on $F$ introduced in section \ref{secDR}. Then we let $h_1$ be 
a metric on $F$ which is  of product structure in a neighbourhood of
$X(Y_1+1/2)-X(Y_1-1/2)$  and 
which coincides with the metric $h$ on 
the set $X(Y_1-1)$ and on the set $X\backslash X(Y_1+1)$. 
Now we consider the collar $W:=(-1/2,1/2)\times Z$ which 
we embed isometrically into $(X,g_1)$ such that
$(-1/2,0)\times Z$  is isometric to a neighbourhood of $Z$ in $X^-$ and such
that $(0,1/2)\times Z$ is 
isometric to a neighbourhood of $Z$ in $X^+$. 
We let $\phi\in C_c^\infty((-1/2,1/2)\times Z)$
be a smooh $\R$-valued function such that $\phi\equiv 1$ in a neighbourhood 
of $\{0\}\times Z$ in $(-1/2,1/2)\times Z$ and such that $\phi(-t,p)=\phi(t,p)$
for
all $(t,p)\in(-1/2,1/2)\times Z$. Then we consider the manifolds with two
components:
\begin{align*}
W^{cut}:=(-1/2,0]\times Z \bigsqcup [0,1/2)\times Z;\quad X^{cut}:=X^-\bigsqcup
X^+.
\end{align*}

The manifold $X$ is complete and therefore by \cite{Ch} the corresponding 
Laplacians with domains the smooth compactly supported sections are essentially
selfadjoint.  Thus by  \cite[Lemma 3.3]{BL} 
the De Rham complex of $F$-valued
differential forms on $X$ has a unique ideal boundary condition
$\mathcal{D}^*(X;F)$.  
We will denote the restriction of $F$ to $X^\pm$ by 
the same letter. Then we use the notation 
\begin{align*}
\mathcal{D}^*(X^\pm;F):=\mathcal{D}^*_{\max}(X^\pm;F), \quad
\mathcal{D}^*(X^\pm,Z;F):=\mathcal{D}^*_{\min}(X^\pm;F).
\end{align*}
We point out that in this section we equip $X$, $X^{\pm}$ and $F$ with 
compact perturbations of the original metrics; thus the flat Hodge 
Laplacians are only compact deformations of the original flat Hodge 
Laplacians in the sense of section \ref{secrelreg}.
Since $X^+$ is a compact manifold with boundary, the cohomology of
$\mathcal{D}^*(X^+;F)$ equals the De Rham or 
singular cohomology of $X^+$ with values in $F$, see for example \cite[Theorem
4.1]{BL}, where the 
proof extends without difficulty to the flat bundle $F$. Therefore, we 
shall just denote it unambiguously by $H^*(X^+;F)$. On the 
other hand, we will denote the cohomology of the complex
$\mathcal{D}^*(X^-,Z;F)$ 
by $H^*_{\min}(X^-,Z;F)$ in order to distinguish it from the corresponding
De Rham cohmology - since $X^-$ retracts onto $Z$, this De Rham cohomology
is of course always trivial. This is not the case for 
the cohomology $H^*_{\min}(X^-,Z;F)$. More precisely, 
as in \cite[section 4.2]{Le}, one can consider the long exact sequence of
Hilbert
complexes
\begin{align}\label{SeqI}
0\longrightarrow\mathcal{D}^*(X^-,Z;F)\overset{\alpha}{
\longrightarrow}
\mathcal{D}^*(X;F)\overset{\beta}{\longrightarrow}\mathcal{D}
^*(X^+;F)\longrightarrow 0, 
\end{align}
where $\alpha$ is extension by zero and
$\beta$ is restriction to $X^+$. To prove the 
exactness of this sequence, we firstly remark that it follows immediately from
the 
definitions that $\alpha$ and $\beta$ are indeed maps between 
the involved Hilbert complexes. If one multiplies with 
suitable cutoff functions, it follows easily that the 
exactness of \eqref{SeqI} is a purely local question in a neighbourhood of the
boundary $Z$. Therefore, 
it suffices to prove the exactness
if one replaces $X$ by the double $X^+\sqcup_Z-X^+$ of $X^+$ and 
$X^-$ by $-X^+$, i.e. $X^+$ with reversed orientation. For the situation of
the double, 
the exactness follows from \cite[equation (4.12)]{BL}: If $\alpha^*$ is the
corresponding involution
introduced on page 121 in \cite{BL}, and if a form
$\omega\in\mathcal{D}^*(X^+\sqcup_Z-X^+;F)$ satisfies
$\omega|X^+=0$, then $\omega|_{-X^+}=(\omega-\alpha^*\omega)|_{-X^+}$ and 
therefore $\omega$ is in the image of the map
$\mathcal{D}^*(-X^+,Z;F)\to\mathcal{D}^*(-X^+\sqcup_Z X^+;F)$. This 
proves exactness in the midde. The surjectivitiy of $\beta$ 
immediately follows from \cite[equation (4.12)]{BL}.
The exact sequence \eqref{SeqI} induces a long exact cohomology sequence. Let 
\begin{align}\label{connect}
\delta_p:H^p(X^+;F)\to H^{p+1}_{\min}(X^-,Z;F)
\end{align}
be the $p$-th connecting homomorphism of that sequence. 
Then we have the following Proposition (let us recall that $F$ was 
the flat vector bundle associated to a representation $\rho$ which 
satisfies $\rho\neq\rho_\theta$):

\begin{prop}\label{Kohofd}
The cohomology of the complex $\mathcal{D}^*(X;F)$ vanishes. 
In particular for each $p$ the map $\delta_p$ is an isomorphism and 
the cohomology $H^*_{\min}(X^-,Z;F)$ is finite-dimensional. Moreover,
$\mathcal{D}^*(X^-,Z;F)$ is a Fredholm complex. 
\end{prop}
\begin{proof}
By \cite[Theorem 3.5]{BL}, the cohomology of $\mathcal{D}^*(X;F)$ equals the
corresponding
$L^2$-cohomology. 
This cohomology is zero by Proposition \ref{LemL2cohomology}. We recall that
$g_1$ and $h_1$ are compact perturbations of the metrics $g$ and $h$ used in
that proposition and hence belong to the quasi isometry 
class of $g$ and $h$.
The last statement follows from \cite[Theorem
2.4 (2)]{BL}.
\end{proof}

Next we use the variation of the De Rham complex $\mathcal{D}^*(X;F)$ 
due to Vishik and Lesch. 
Let 
$\iota_{\pm}:Z\hookrightarrow X^\pm$ denote the inclusions. 
For $\theta\in\R$ and each $p$ let 
\begin{align*}
\mathcal{D}^p_\theta(X;F):=\{(\omega_1,\omega_2)\in
\mathcal{D}^p(X^-;F)\oplus\mathcal{D}^p(X^+;F)\colon
\cos{\theta}\cdot\iota_{-}^*\omega_1=\sin{\theta}\cdot\iota_{+}^*\omega_2 \},
\end{align*}
where, as in \cite[section 4.2]{Le}, $\iota_{\pm}^*(\omega)$ are boundary
values 
defined by traces on Sobolev spaces. More precisely, by \cite[Theorem
1.9]{Paq}, 
$\iota_{\pm}^*$ are well defined on $\mathcal{D}^p(X^{\pm};F)$ and
$\mathcal{D}^p(X^{\pm},Z;F)$ 
coincides with the kernel of $\iota_{\pm}^*$. 
Again, the non-compactness of $X^-$ is 
not an issue here, since one can localize everything to a neighourhood 
of $Z$ and $Z$ is compact. The flat bundle isn't an issue either here.  
For each p let $d_p^\theta:\mathcal{D}^p_\theta(X;F)\to
\mathcal{D}^{p+1}_\theta(X;F)$ be the restriction 
of the direct sum of the operators $d_{p;\max}$ on
$\mathcal{D}^{p}(X^{-};F)$  and on $\mathcal{D}^p(X^+;F)$.
Then the $\mathcal{D}^p_\theta(X;F)$ with $d_p^\theta$ form a subcomplex
$\mathcal{D}^{*}_\theta(X;F)$ of
$\mathcal{D}^{*}(X^-;F)\oplus\mathcal{D}^{*}(X^+;F)$. 
By the previous remark one has 
\begin{align}\label{Splittingkompl}
\mathcal{D}^{*}_{\theta=0}(X;F)=&\mathcal{D}^*(X^{-},Z;F)\oplus\mathcal{D}
^*(X^{+};F),
\end{align}
Furthermore, arguing in the same way as after \eqref{SeqI}, one can show that
one has a natural isomorphism of 
complexes
\begin{align}\label{IsokomplD}
\mathcal{D}^{*}_{\theta=\frac{\pi}{4}}(X;F)\cong\mathcal{D}^{*}
(X;F), 
\end{align}
\cite[page 16]{Vi}, \cite[page 22]{Le}. 
Also as in \eqref{SeqI} one has a short sequence
\begin{align}\label{Seqpara}
0\longrightarrow\mathcal{D}^*(X^-,Z;F)\overset{\alpha_\theta}{
\longrightarrow}
\mathcal{D}^*_\theta(X;F)\overset{\beta_\theta}{\longrightarrow}\mathcal{D}
^*(X^+;F)\longrightarrow 0, 
\end{align}
where $\alpha_\theta(\omega_1):=(\omega_1,0)$ and where
$\beta_\theta(\omega_1,\omega_2):=\omega_2$. 
\newline

The complexes $\mathcal{D}^*_\theta(X;F)$ have varying domains; following 
Lesch, \cite[section 5]{Le}, one can transform them into complexes with fixed
domains. Thus 
let $S:W^{cut}\to W^{cut}$ be the reflection map, given by
$S((t,p)):=(-t,p)$. Then define 
\begin{align*}
T:\Omega^*(W^{cut};F)\to\Omega^*(W^{cut};F), \quad
T(\omega_1,\omega_2):=(S^*\omega_2,-S^*\omega_1). 
\end{align*}

One has 
\begin{align}\label{eqmatr}
\begin{pmatrix}\cos{(\theta+\theta')},
&-\sin{(\theta+\theta')}\end{pmatrix}\cdot\begin{pmatrix}\cos{\theta}&\sin{
\theta}\\
-\sin{\theta}&\cos{\theta}\end{pmatrix}=\begin{pmatrix}\cos{\theta'},&-\sin{
\theta'}\end{pmatrix}.
\end{align}
Thereore, if one define
\begin{align}\label{DefPhitheta}
\Phi_\theta:=\cos(\theta\phi)\Id+\sin(\theta\phi)T:\Omega^*(W^{cut}
;F)\to\Omega^*(W^{cut};F),
\end{align}
then $\Phi_\theta$ canonically extends to a unitary transformation 
of $L^2(X,\Lambda^* T^*X\otimes F)$
and applying
\eqref{eqmatr}
it follows that $\Phi_\theta$ maps $\mathcal{D}^p_{\theta'}(X;F)$ onto
$\mathcal{D}^p_{\theta+\theta'}(X;F)$. Finally, since 
$T$ commutes with exterior differentiation, on
$\mathcal{D}^*_{\theta'}(X;F)$ one has
$\Phi_\theta^*d^{\theta+\theta'}\Phi_\theta=d^{\theta'}+\theta\ext{(d\phi)}T$,
see \cite[Lemma 5.1]{Le}.
Thus if one lets
\[
\tilde{d}_p^\theta:=d^{\frac{\pi}{4}}_p+\theta\ext{(d\phi)}T
\]
with domain $\mathcal{D}^p_{\theta=\pi/4}(X;F)$, one obtains 
a complex $\tilde{\mathcal{D}}^*_{\theta}(X;F)$ with fixed domain which is
isometric to the
complex
$\mathcal{D}^*_{\theta}(X;F)$.  
Since $\theta\ext{(d\phi)}T$ is bounded and since one can identifty
$d^{\pi/4}_p$ with $d_{p}$ on $\mathcal{D}^*(X;F)$, it follows that the 
complexes $\mathcal{D}^*_{\theta}(X;F)$ are Hilbert complexes too. 
Let $\Delta_p^\theta$ resp. $\tilde{\Delta}_p^\theta$ be the Laplacians of
$\mathcal{D}^*_{\theta}(X;F)$ resp. $\tilde{\mathcal{D}}^*_{\theta}(X;F)$. Let
\begin{align}\label{Smfam}
\frac{d}{d\theta}\tilde{\Delta}_p^\theta\nonumber:=&(\tilde{d}
_p^\theta)^*(\ext(d\phi)T) +(\ext(d\phi)T)^*\tilde{d}_p^\theta\nonumber\\
+&(\tilde{d}
_{p-1}^\theta)(\ext(d\phi)T)^*+(\ext(d\phi)T)(\tilde{d}_{p-1}^\theta)^*.
\end{align}
Then $\frac{d}{d\theta}\tilde{\Delta}_p^\theta$ is a closed operator 
with domain containing $\dom(\Delta_p)$. Moreover, for each
$\phi\in\dom(\Delta_p)$ one obviously has in $L^2$:
\begin{align}\label{Diffquot}
\lim_{\theta'\to\theta}\frac{(\tilde{\Delta}_p^{\theta'}-\tilde{\Delta}
_p^\theta)\phi}{\theta'-\theta}=\frac{d}{d\theta}\tilde{\Delta}_p^\theta\:\phi. 
\end{align}
Therefore, the family $\tilde{\Delta}_p^\theta$ is a smooth one-parameter
family 
of selfajoint operators with constant domains. We remark that the operators
$\tilde{\Delta}_p^\theta$ 
are not differential operators since $T$ doesn't decrease supports, see
\cite[page 24]{Le}. However, outside the collar where the variation takes 
place, they are differential operators.

\section{The heat kernel of the perturbed Laplacian}\label{secheatpertL}
\setcounter{equation}{0}
We keep the notation of the preceding section.
By \cite[Lemma 3.1]{Le}, the operator $e^{-t\tilde{\Delta}_p^\theta}$
acts as an integral operator with smooth kernel
$K_{\tilde{\Delta}_p^\theta}$.  
For the purposes of the present paper, we will need to construct this 
kernel explicitly. We use the parametrix method to obtain the 
kernel by patching together the original heat Kernel of $\Delta_p$ with 
the kernel of the perturbed operaterator on the double of $X^+$.  
This method of patching togehter heat kernels can be found in \cite[page 184
ff]{RS}, \cite[E III]{Berger}, \cite[chapter 4]{DO}, \cite[chapter VII]{Mulecn},
for
example .

We continue to work with the metric $g_1$ on $X$ and the metric $h_1$ on
$F:=E_\rho$. 
Consider the double $-X^+\sqcup_{\partial X^+} X^+$ of $X^+$, where $-X^+$
denotes 
$X^+$ with the reversed orientation. Let us recall that $X^+$ is a compact
manifold 
with boundary and that near its boundary all structures are product.
On this double, one can perform 
exactly the same variation as described above, replacing $X^-$ by 
$-X^+$. We denote the corresponding bundle by $F$ too and we denote the
corresponding operators by $\Delta_p^\theta$,
$\tilde{\Delta}_p^\theta$ too. Then 
the integral kernel $E_t^{p,\theta}(x,y)$ of $e^{-t\Delta_p^\theta}$ acting on
$L^2(-X^+\sqcup X^+;F)$ can
be computed explicitly \cite[equation 5.22]{Le}: Let
$E_t^p$ denote the kernel of the $p$-th F-valued Hodge-Laplacian 
on the closed manifold $-X^+\sqcup_{\partial X^+} X^+$. 
Then one has 
\begin{align}\label{KerL}
E_t^{p,\theta}(x,y)=\begin{cases}E_t^{p}(x,y)\pm\cos{(2\theta)}(S^*\circ
E_t^{p})(x,y) & \text{if}\: x,y\in \pm X^{+} \\ \sin{(2\theta)}E_t^{p}(x,y)
&\text{if} \:x\in \pm X^+,y\in \mp X^+,
\end{cases}
\end{align}
where we have put $+X^+:=X^+$. Thus, if $\Phi_\theta$ 
is now defined on $L^2(-X^+\sqcup X^+;F)$, the integral kernel
$\tilde{E}_t^{p,\theta}$ of $e^{-t\tilde{\Delta}_p^\theta}$ is 
given by 
\begin{align}\label{Etilde}
\tilde{E}_t^{p,\theta}=\Phi_{\frac{\pi}{4}-\theta}\circ E_t^{p,\theta}\circ
\Phi_{\frac{\pi}{4}-\theta}^{-1}. 
\end{align}
Equations \eqref{DefPhitheta} and \eqref{KerL} immediately imply  that
$\tilde{E}_t^{p,\theta}(x,y)$ 
depends smoothly on $\theta$ and using the 
Gaussian bounds of $E_t^p$ on the closed manifold
$-X^+\sqcup_{\partial X^+} X^+$ (see for example \cite[Theorem
1.4.3.1]{Greiner}), they imply that for each $T>0$ there exists constants
$C_1,C_2>0$ such that for 
all $t\in (0,T]$ and all $x,y\in -X^+\sqcup_{\partial X^+} X^+$ and
$i\in\{0,1\}$ one has
\begin{align}\label{KernEp}
\left\|(\tilde{\Delta
}_p^\theta)^i\tilde{E}_{t}^{p,\theta}(x,y)\right\|=\left\|\frac{d^i}{dt^i}\tilde
{E}_{t}^{p,\theta}(x,y)\right\|\leq
C_1t^{-\frac{d}{2}-i}e^{-C_2\frac{d^2(x,y)}{t}}.
\end{align}

Now let $K_{\Delta_p}(t,x,y)$ be the integral kernel of the original  $F$-valued
Laplacian $\Delta_p$ on $X$, where we of course still use the metrics $g_1$ and
$h_1$. This kernel satisfies the properties of Proposition \ref{Estheatk}. 
Let $\Phi_1$ and $\Psi_1$ be smooth 
functions on $X$ which have compact support in $X(Y_1+1/2)$, which are
identically $1$ on a neighbourhood of $X(Y_1)\cup \supp(\phi)$, where $\phi$ is
as in the previous section, 
and which satisfy
$\Phi_1=1$ on a neighbourhood of 
$\supp(\Psi_1)$. 
Let $\Psi_2:=1-\Psi_1$ and fix a smooth function $\Phi_2$ such that 
$\Phi_2=1$ 
on a neighbourhood of $\supp(\Psi_2)$ and such that $\Phi_2=0$ on a
neighbourhood of $\supp(\phi)$. 
Then one defines 
\begin{align}\label{DefP}
P(t,x,y;\theta):=\Phi_1(x)\tilde{E}_t^{p,\theta}(t,x,
y)\Psi_1(y)+\Phi_2(x)K_{\Delta_p}(t,x,
y)\Psi_2(y).
\end{align}
Let 
\begin{align}\label{DefQ1}
Q_1(t,x,y;\theta):=\left(\frac{d}{dt}+\tilde{\Delta}_p^\theta\right)P(t,x,
y;\theta),
\end{align}
where $\tilde{\Delta}_p^\theta$ is applied to the first variable. 
Since $\Phi_{1,2}=1$ on neighbourhoods of $\supp(\Psi_{1,2})$ and since 
$\tilde{\Delta}_p^\theta$ equals the corresponding operator on $X^+\sqcup -X^+$
on 
$\supp(\Phi_1)$ resp. equals $\Delta_p$ on $\supp(\Phi_2)$, there exists an
$\epsilon>0$ such that
\begin{align}\label{PropQeins}
Q_1(t,x,y;\theta)=0,\quad\text{if} \:\: d(x,y)<\epsilon. 
\end{align}
Moreover, there exists a compact subset $\mathcal{V}$ of 
$X$ such that $Q_1(t,x,y;\theta)=0$ if $x\notin \mathcal{V}$.
Thus invoking Proposition \ref{Estheatk}, \eqref{KernEp} and the local
G{\aa}rding inequality on $\mathcal{V}$, it follows easily 
that for each $T>0$ and $l\in\mathbb{N}$ there exist constants $C_1$
and $C_2$
such that for all $t\in (0,T]$ one can estimate
\begin{align}\label{EstQ1}
\left\|Q_1(t,x,y;\theta)\right\| \leq t^l C_1 r(y)^n
e^{-C_2\frac{d^2(x,y)}{t}}
\end{align}
for all $x,y\in X$, where we make from now on the convention 
that we use $C,C_1,C_2$ as constants in every estimate. 
In particular, $P$ is a parametrix or an approximate solution 
of the heat equation in the sense of \cite{Berger} resp. \cite{BGV}. 
Now define inductively 
\begin{align}\label{DefQk}
&Q_{k+1}(t,x,y;\theta):=\int_0^t\int_{X}Q_1(t-s,x,w;\theta)\circ
Q_k(s,w,y;\theta)dwds.
\end{align}
The inner integral ranges over the fixed compact set $\mathcal{V}$ only and
applying
\eqref{EstQ1} and the 
elementary inequality
\begin{align}\label{ELIE}
\frac{d^2(x,w)}{t-s}+\frac{d^2(w,y)}{s}\geq \frac{d^2(x,y)}{t},
\end{align}
$0<s<t$, \cite[page 67]{DO}, 
one shows by induction in $k$ that 
\begin{align*}
\left\|Q_{k}(t,x,y;\theta)\right\|\leq C_1 r(y)^n t^l
\frac{t^{k}}{k!}e^{-C_2\frac{d^2(x,y)}{t}}.
\end{align*}
In particular, the series 
\begin{align}\label{DefQ}
Q(t,x,y;\theta):=\sum_{k=1}^\infty (-1)^kQ_k(t,x,y;\theta)
\end{align}
converges absolutely and satisfies
\begin{align}\label{EstQ}
\left\|Q(t,x,y;\theta)\right\|\leq C_1 t^l r(y)^n
e^{-C_2\frac{d^2(x,y)}{t}}.
\end{align}
Let 
\begin{align}\label{conv}
P*Q(t,x,y,\theta):=\int_0^t\int_X
P(t-s,x,w;\theta)\circ Q(s,w,y;\theta)dwds.
\end{align}
This term can be estimated as follows. 
Since the inner integral in \eqref{conv} ranges over the fixed compact set
$\mathcal{V}$ only, the estimates 
from Proposition \ref{Estheatk}, \eqref{KernEp} and \eqref{EstQ} imply that
\begin{align*}
r(x)^jr(y)^j\left\|P*Q(t,x,y,\theta) \right\|\leq
&C_1t^l(r(x)r(y))^{j+n}e^{-C_2\frac{
\dist^2(x,\mathcal{V})}{t}}e^{-C_2\frac{\dist^2(y,\mathcal{V})}{t}}\\
&\cdot\int_{0}^t\int_{\mathcal{V}}
s^{-\frac{d}{2}}e^{-C_2\frac{d^2(x,w)}{s}}dwds.
\end{align*}
The last integral needs to be estimated only if $x$ belongs to some fixed
relatively compact neighbourhood of $\mathcal{V}$ and 
here one uses geodesic coordinates around $x$ and a change of varibales to see
that it is bounded. 
By the definition of $r$ and the hyperbolic distance, for each 
$j\in\mathbb{N}$, each $T>0$ and each $c>0$ there exists a constant $C>0$ 
such that for all $x\in X$ and all $t\in (0,T]$ one has 
\begin{align}\label{estrdist}
e^{-c\frac{\dist^2(x,\mathcal{V})}{t}}r(x)^j\leq C.
\end{align}
Thus for
all $j,l\in\mathbb{N}$, $T>0$ there exists a constant $C>0$ such that 
for all $x,y\in X$ and all $t\in (0,T]$ one has 
\begin{align}\label{Estconv1}
&r(x)^jr(y)^j\left\|P*Q(t,x,y,\theta)\right\|\leq
C t^{l} .
\end{align}

Using the previous results, one now  immediately checks
that 
\begin{align}\label{KerKDelta}
K_{\tilde{\Delta}_p^\theta}(t,x,y):=P(t,x,y;\theta)+P*Q(t,x,y,\theta)
\end{align}
is the integral kernel of $e^{-t\tilde{\Delta}_p^\theta}$. 
This implies in particular that the kernel $K_{\tilde{\Delta}_p^\theta}$ has
the same growth properties as the kernel $K_{\tilde{\Delta}_p^\theta}$ stated in
Proposition \ref{Estheatk}, i.e.
for each $T>0$ there exist constants $C_1,C_2>0$ such that for all $x,y\in
X$ and all $t\in (0,T]$ and all $j\in\{0,1\}$ one can estimate
\begin{align}\label{GaussBKerdelta}
\left\|(\tilde{\Delta}_p^\theta)^jK_{\tilde{\Delta}_p^\theta}(t,x,y)\right\|\leq
C_1
r(x)^{n} r(y)^{n} 
t^{-\frac{d}{2}-j}e^{-C_2\frac{d^2(x,y)}{t}}.
\end{align} 

We can now generalize the results of section \ref{secrelreg} to the operator
$\tilde{\Delta}_p^\theta$. We remark that 
this proposition is not a direct consequence of the results of section
\ref{secrelreg}, since the operator 
$\tilde{\Delta}_p^\theta$ is not a differential operator. However, since
outside 
the fixed collar where the variation takes place it is a differential operator 
of the form studied in section \ref{secrelreg}, it is of course not surprising
that 
the results of that section continue to hold. We let
$T_{p,u}:=T_{\nu_p(\rho),u}+\rho(\Omega)\Id$, where 
the latter operator is as in section \ref{secrelreg}. We make 
from now on the convention that $u>Y_1+2$, i.e. the auxilliary operator doesn't
interact 
with the variation.

\begin{prop}\label{PropDeltatheta}
In the preceding notations one has:
\begin{enumerate}
\item \label{It1} The operator $e^{-t\tilde{\Delta}_p^\theta}-e^{-tT_{p,u}}$ is
trace
class and one
has 
\begin{align*}
\Tr_{\Rel,u}(e^{-t\tilde{\Delta}_p^\theta}):=\Tr\left(e^{-t\tilde{\Delta}
_p^\theta}-e^{-tT_{p,u}
}\right)=\int_{X}\Tr\left(K_{\tilde{\Delta}_p^\theta}(t,x,x)-K_{
T_{p,u}}(t,x,x)dx\right)dx.
\end{align*}
\item \label{It2} 
There is a short-time asymptotic expansion
\begin{align*}
\Tr_{\Rel,u}(e^{-t\tilde{\Delta}_p^\theta})=\sum_{j=0}^\infty
a_{j,\theta} t^{j-\frac{d}{2}}+\sum_{j=0}^\infty b_{j,\theta} t^{
j-\frac{1}{2}}\log{t}+\sum_{j=0}^\infty c_{j,\theta} t^j,
\end{align*}
as $t\to 0+$, where the coefficients depend smoothly on $\theta$. 
\item \label{It3} 
The essential spectrum of $\tilde{\Delta}_p^\theta$ 
is contained in $[\frac{1}{4},\infty)$ and in particular the complex
$\mathcal{D}^*_{\theta}(X;F)$ 
is a Fredholm complex. 
\item \label{It4} 
For each $\theta$ there exist constants $C,c$ which depend on
$\theta$
such that one has 
\begin{align*}
\Tr_{\Rel,u}(e^{-t\tilde{\Delta}_p^\theta})-\dim\Ker(\tilde{\Delta}
_p^\theta)\leq C e^{-ct}
\end{align*}
for $t\geq 1$. 
\item \label{It5} 
The operator $e^{-t\Delta_p^\theta}-e^{-tT_{p,u}}$
is trace class and if
$\Tr_{\Rel,u}(e^{-t\Delta_p^\theta}):=\Tr\left(e^{-t\Delta_p^\theta}-e^{-tT_{p,u
}
}\right)$, then
$\Tr_{\Rel,u}(e^{-t\Delta_p^\theta})=\Tr_{\Rel,u}(e^{-t\tilde{\Delta}_p^\theta}
)$. 
\end{enumerate}
\end{prop}
\begin{proof}
Let $P$ be as in \eqref{DefP}. Then the function
$K_{\Delta_p}(t,x,y)-P(t,x,y;\theta)$ vanishes 
if $(x,y)\notin X(Y_1+1)\times X\cup X\times X(Y_1+1)$; let us 
recall that $X(Y_1+1)$ is compact. Thus using an estimate 
analogous to \eqref{estrdist},
\eqref{integral} and Proposition \ref{Estheatk} imply that for each
$j\in\mathbb{N}$ the
function 
\begin{align*}
(r(x)r(y))^j(K_{\Delta_p}(t,x,y)-P(t,x,y;\theta))
\end{align*}
is square integrable over $X\times X$. 
Thus, by \eqref{Estconv1} and \eqref{KerKDelta}, for each $j\in\mathbb{N}$ the
function 
\begin{align*}
(r(x)r(y))^j\left(K_{\Delta_p}(t,x,y)-K_{\tilde{\Delta}_p^\theta
}(t,
x, y)\right)
\end{align*}
is square-integrable over $X\times X$, i.e. the operators
$M_{r^j}\circ(e^{-t\Delta_p}-e^{-t\tilde{\Delta}_p^\theta})$ and
$(e^{-t\Delta_p}-e^{-t\tilde{\Delta}_p^\theta})\circ M_{r^j}$
are Hilbert - Schmidt operators, where $M_{r^j}$ denotes the 
operator induced by multiplication with the function $r^j$. 
If we apply \eqref{integral}, \eqref{Estr}, Proposition \ref{Estheatk} and 
\eqref{GaussBKerdelta}, we see that for $j\in\mathbb{N}$ sufficiently large the
functions
$(r(y))^{-j} K_{\Delta_p}(t,x,y)$ and
$(r(x))^{-j} K_{\tilde{\Delta}^\theta_p}(t,x,y)$ 
are square integrable over $X\times X$ i.e.
$e^{-t\Delta_p}\circ M_{r^{-j}}$ and $M_{r^{-j}}\circ
e^{-t\tilde{\Delta}_p^\theta}$
are 
Hilbert Schmidt operators. 
Put $\tau=\frac{t}{2}$. Then one has 
\begin{align*}
e^{-t\Delta_p}-e^{-t\tilde{\Delta}_p^\theta}=&\left(e^{-\tau\Delta_p}\circ
M_{r^{-j}}\right)\circ \left(M_{r^{j}}\circ(e^{-\tau\Delta_p}
-e^{
-\tau\tilde{\Delta}^\theta_p})\right)\\
+&\left((e^{-\tau\Delta_p}-e^{-\tau\tilde{ \Delta}
^\theta_p })\circ M_{r^{j}}\right)\circ
\left(M_{r^{-j}}e^{-\tau\tilde{\Delta}^\theta_p}\right)
\end{align*}
and thus $e^{-t\Delta_p}-e^{-t\tilde{\Delta}_p^\theta}$ is trace class and by
\cite[Theorem 4.1]{Warner} one has 
\begin{align*}
\Tr(e^{-t\Delta_p}-e^{-t\tilde{\Delta}_p^\theta})=\int_X(
\Tr K_{\Delta_p}(t,x,x)-\Tr K_{\tilde{\Delta}_p^\theta}(t,x,x))dx.
\end{align*}
Applying Proposition \ref{Propreltr}, \eqref{It1} follows. By \eqref{Estconv1}
and the 
fact that $X$ has finite volume, 
for 
the short time asymptotic expansion of the relative trace we can 
replace the kernel $K_{\tilde{\Delta}_p^\theta}$ by its parametrix $P$, i.e. we
have
\begin{align*}
\Tr_{\Rel,u}(e^{-t\tilde{\Delta}_p^\theta})=&\int_X\Tr\left(P(t,x,x,\theta)-K_{
T_{p,u}}
(t,x,x)\right)dx+O(t^{\infty})\\
=&\int_X \Psi_1(x) \Tr E_t^{p,\theta}(t,x,x)dx+\int_X(\Psi_2(x)-1)\Tr
K_{\Delta_p}(t,x,x)dx\\
+&\int_X(\Tr K_{\Delta_p}(t,x,x)-\Tr K_{T_{p,u}}(t,x,x))+O(t^\infty) 
\end{align*}
as $t\to +0$. The first two integrals in the last equation range over compact
subsets of $X$ only. 
Moreover, using \eqref{DefPhitheta}, \eqref{KerL} and the 
pointwise short time asymptotic expansion of $E_t^{p}$ \cite{Gilkey}, it follows
that the first integral has the required short time asymptotic 
expansion (without logarithmic terms) and that the coefficients in 
this expansion depend smoothly on $\theta$. Also, 
by the pointwise heat asymptotics of $\Tr
K_{\Delta_p}(t,x,x)$ \cite{Gilkey}, the second summand has the 
required asymptotics without logarithmic terms. The third summand has the
required asymptotic expansion by Proposition \ref{Propasepx}. This proves
\eqref{It2}.
\eqref{It3} follows from the trace class property 
of $e^{-t\Delta_p}-e^{-t\tilde{\Delta}_p^\theta}$, Proposition \ref{Longtime}
and the 
invaraince of the essential spectrum under compact perturbations.  
\eqref{It4} follows in exactly the same way as the corresponding statement in
Proposition \ref{Longtime}. Since $\Phi_\theta$ commutes 
with the operator $T_{p,u}$, \eqref{It5} follows. 
\end{proof} 

\begin{bmrk}
For $\theta\in (0,\frac{\pi}{2})$ one has in fact
$\dim\Ker(\tilde{\Delta}_p^\theta)=0$ as we will later show using Lesch's
argument. However, for $\theta=0$ this 
is not the case by \eqref{Splittingkompl} and the results of section
\ref{secbases}. The possibility of such a jump 
in the dimension of the cohomology has been predicted by Lesch \cite[section
5.2.3]{Le}.
\end{bmrk}

The previous proposition and standard results on Mellin transforms \cite{Gilkey}
imply in particular that  we can define 
the relative analytic torsion 
\[
T_{\Rel,u}^{\theta}(X,g_1,h_1;E_\rho)
\]
by using the relative traces $\Tr_{\Rel,u}(e^{-t\Delta_p^\theta})$ and the 
corresponding relative zeta functions in the same way as at the 
end of section \ref{secrelreg}. We have put $g_1$ and $h_1$ in the 
notation here in order to indicate that we still work with the perturbed metrics
$g_1$ and $h_1$.

Now for $Y>Y_1+1$ we consider the manifold $M(Y)$ from equation \eqref{DefM}. 
This manifold contains $X^+$ and also the collar $W$. 
Also, by construction, there exists a metric on $M(Y)$ which coincides with the
metric $g_1$ on $X(Y)$ and there exists a Hermitian flat vector bundle $F$ over
$M(Y)$  whose restriction
to $X(Y)\subset M(Y)$ is isometric to the bundle $F$ with the metric $h_1$.
Therefore, if we let $M(Y)^+:=X^+$, $M(Y)^-:=M(Y)-M(Y)^+$, then we are in a
special case of a situation 
which is considered by Lesch. We shall denote the corresponding operators on
$M(Y)$ 
over the flat bundle $F$ by $\Delta_p^\theta[M(Y)]$ resp.
$\tilde{\Delta}_p^\theta[M(Y)]$. 
Let $e^{-t\tilde{\Delta}_p^\theta}$ and $e^{-\tilde{\Delta}_p^\theta[M(Y)]}$
denote 
the heat semigroups of $\tilde{\Delta}_p^\theta$ and
$\tilde{\Delta}_p^\theta[M(Y)]$. 
We now show that the kernel of
$e^{-t\tilde{\Delta}_p^\theta}$ can be approximated by the kernel of
$e^{-t\tilde{\Delta}_p^\theta[M(Y)]}$ 
on the fixed collar $W$ if
$Y\to\infty$ in the same way as in Proposition \ref{Propapproxim}, which we
can again not apply directly since $T$ and thus the operators
$\tilde{\Delta}_p^\theta$ are not
differential operators.

\begin{prop}\label{Propaxim2}
Let $D$ be a differential operator acting on the smooth 
$F$-valued $p$-forms. Then one has uniformly in $x_0,y_0\in X(Y_1)$:
\begin{align*}
\frac{d^i}{(dt)^i}D_xK_{\tilde{\Delta}_p^\theta}(t,x_0,y_0)=\lim_{Y\to\infty}
\frac{d^i}{(dt)^i}D_xK_{\tilde{\Delta}
_p^\theta
[M(Y)]}(t,x_0,y_0),
\end{align*}
where $D_x$ indicates that $D$ is applied to the 
first variable and where $i\in\mathbb{N}$.  
\end{prop}

\begin{proof}
Let $K_{\Delta_p[M(Y)]}$ be the integral
kernel of $e^{-t\Delta_p[M(Y)]}$ and let $E_{t}^{p,\theta}$ 
be as above. Then one constructs the kernel $K_{\tilde{\Delta}_p^\theta[M(Y)]}$ 
as above. More precisely, let $\Phi_1$ and $\Psi_1$ be 
the same functions as above. Then $\Phi_1$ and $\Psi_1$ can 
also be regarded as functions on $M(Y)$. Also, $\Phi_2$ and $\Psi_2$ are 
equal to $1$ in a neighbourhood of $\partial X(Y)$ in $X$. Therefore, they 
may be extended by $1$ to functions on $M(Y)$ and denoting 
these functions by the same symbol, it follows as before that 
\begin{align*}
P^Y(t,x,y;\theta):=\Phi_1(x)\tilde{E}_t^{p,\theta}
\Psi_1(y)+\Phi_2(x)K_{\Delta_p[M(Y)]}\Psi_2(y)
\end{align*}
is a parametrix for the operator $\tilde{\Delta}_p^\theta[M(Y)]$, where 
$\tilde{E}_t^{p,\theta}$ is as in \eqref{Etilde}. 
As before let 
\begin{align*}
Q_1^Y(t,x,y;\theta):=\left(\frac{d}{dt}+\tilde{\Delta}_p^\theta[M(Y)]
\right)P^Y(t,x,
y;\theta),
\end{align*}
where $\tilde{\Delta}_p^\theta[M(Y)]$ is applied to the $x$-variable 
and define inductively 
\begin{align}\label{DefQY}
&Q_{k+1}^Y(t,x,y;\theta):=\int_0^t\int_{M(Y)}Q_1^Y(t-s,x,w;\theta)\circ
Q_k^Y(s,w,y;\theta)dwds
\end{align}
and then let
\begin{align*}
Q^Y(t,x,y;\theta):=\sum_{k=1}^\infty (-1)^kQ_k^Y(t,x,y;\theta).
\end{align*}
Then 
\begin{align*}
K_{\tilde{\Delta}_p^\theta[M(Y)]}(t,x,y):=P^Y(t,x,y;\theta)+\int_0^t\int_
{M(Y)}P^Y(t-s,x,w;\theta)\circ Q^Y(s,w,y;\theta)dw ds
\end{align*}
is the integral kernel of $\tilde{\Delta}_p^\theta[M(Y)]$. 
For $x,y\in X(Y)$ one has 
\begin{align*}
P(t,x,y;\theta)-P^Y(t,x,y;\theta)=\Phi_2(x)(K_{\Delta_p}(t,x,y)-K_{\Delta_p[M(Y)
]}(t,x,y))\Psi_2(y).
\end{align*}
Moreover, by construction one has
$\tilde{\Delta}_p[M(Y)]=\Delta_p[M(Y)]=\Delta_p$ 
on $\supp(\Phi_2)\cap X(Y)$. 
Thus by Proposition \ref{Propapproxim}, for $x_0,y_0\in X(Y_1+1)$ one can
estimate
\begin{align}\label{EstK}
\max&\left\{\left\|P(t,x_0,y_0;\theta)-P^Y(x_0,y_0;\theta)\right\|,
\:\left\|Q_1(t,x_0,
y_0;\theta)-Q_1^Y(x_0,y_0;\theta)\right\|\right\}
\nonumber \\ \leq &C_1e^{-C_2\frac{\dist^2(X(Y_1+1),\partial
X(Y))}{t}},
\end{align}
where $P(t,x_0,y_0;\theta)$ and $Q_1(t,x_0,y_0,\theta)$ are as in \eqref{DefP}
resp. \eqref{DefQ1}.
The inner integrals in \eqref{DefQk} and \eqref{DefQY} range over the same
compact set
$\mathcal{V}\subset X(Y_1+1)$ 
and by induction it follows that one has 
\begin{align*}
\left\|Q_{k+1}(t,x_0,y_0;\theta)-Q_{k+1}^Y(t,x_0,y_0;\theta)\right\|\leq
C_1e^{-C_2\frac{\dist^2(X(Y_1+1),\partial
X(Y))}{t}}\frac{t^{k}}{k!},
\end{align*}
for all $x_0,y_0\in X(Y_1+1)$,
where $Q_{k+1}(s,x_0,y_0;\theta)$ is as in \eqref{DefQk}. 
Thus one has
\begin{align}\label{Estq} 
\left\|Q(t,x_0,y_0;\theta)-Q^Y(t,x_0,y_0;\theta)\right\|\leq
C_1e^{-C_2\frac{\dist^2(X(Y_1+1),\partial
X(Y))}{t}},
\end{align}
where $Q(s,x_0,y_0;\theta)$ is as in \eqref{DefQ}. 
Writing 
\begin{align*}
&K_{\tilde{\Delta}_p^\theta}(t,x,y)-K_{\tilde{\Delta}_p^\theta[M(Y)]}(t,x,y)=P(t
,x,y;\theta)-P^Y(t,x,
y;\theta)\\ +&\int_0^t\int_{\mathcal{V}}
(P(t-s,x,w;\theta)- P^Y(t-s,x,w;\theta))Q(s,w,y)dwds \\
+&\int_0^t\int_{\mathcal{V}}
P^Y(t-s,x,w;\theta)(Q(s,w,y;\theta)-Q^Y(s,w,y;\theta))dwds
\end{align*}
and employing the previous estimates, the Proposition is proved if $D$ is 
the identity. For general $D$ resp. for the $t$-derivatives, the proof is word
for word the same since 
the results of Proposition \ref{Propapproxim} hold for 
general $D$ resp. the $t$-derivatives. 
\end{proof}

\section{The first version of the gluing formula}\label{secgl1} 
\setcounter{equation}{0}
In this section we prove the gluing formula under the 
additional assumption that we work with the metrics 
$g_1$ and $h_1$ which are compact perturbations of 
our original metric but which are of product structure on 
the part of the manifold where the variation of the De Rham 
complex takes place. We keep all the notations of the preceding 
section and we remark that by \eqref{It5} in Proposition \ref{PropDeltatheta}
also in our case we may 
interchange $\tilde{\Delta}_p^\theta$ with $\Delta_p^\theta$ if we consider
relative 
traces. We start with the following proposition.  

\begin{prop}\label{trclassder}
The kernel $K_{\tilde{\Delta}_p^\theta}(t,x,y)$ is differentiable in
$\theta$ and 
if $\frac{d}{d\theta}
e^{-t\tilde{\Delta}_p^\theta}$ denotes the integral operator 
with kernel $\frac{d}{d\theta}K_{\tilde{\Delta}_p^\theta}(t,x,y)$, then
$\frac{d}{d\theta}
e^{-t\tilde{\Delta}_p^\theta}$ is trace class. Also
$\Tr_{\Rel,u}(e^{-t\tilde{\Delta}_p^{\theta}})$ is differentiable in
$\theta$. 
One has
\begin{align*}
\Tr\left(\frac{d}{d\theta}e^{-t\tilde{\Delta}_p^\theta}\right)
=\int_X\Tr\left(\frac{d}{d\theta}K_{\tilde{\Delta}_p^\theta}(t,x,
x)\right)dx=\frac{d}{d\theta}\Tr_{\Rel,u }(e^{-t\tilde{\Delta}_p^{\theta}}),
\end{align*}
Moreover, one has 
\begin{align*}
\Tr\left(\frac{d}{d\theta}
e^{-t\tilde{\Delta}_p^\theta}\right)=-t\Tr\left(\left(\frac{d}{d\theta}\tilde{
\Delta}_p^\theta\right)e^{-t\tilde{\Delta}_p^\theta}\right)=
-t\int_{X}\Tr\left(\left(\frac{d}{d\theta}\tilde{\Delta}_p^\theta\right)K_{
\tilde{\Delta}_p^\theta}(t,x,x)\right)dx,
\end{align*}
where $\frac{d}{
d\theta}\tilde{\Delta}_p^\theta$ is applied to the first variable of
$K_{\tilde{\Delta}_p^\theta}(t,x,x)$. 
\end{prop}

\begin{proof}
Let $\tilde{E}_t^{p,\theta}(x,y)$ be as in \eqref{Etilde} and let
$P(t,x,y;\theta)$ be as in \eqref{DefP}.
Then \eqref{DefPhitheta}, \eqref{Smfam},  \eqref{KerL} and  and the Gaussian 
bounds for $E_t^p$ immediately show
that $P(t,x,y;\theta)$ is differentiable in $x,y,t$ and $\theta$ and that
for each $T>0$ there 
exist constants $C_1,C_2>0$ such that for all $t\in (0,T]$, all $\theta$ and 
all $x,y\in X$ one has 
\begin{align}\label{DerP0}
\left\|\frac{d}{d\theta}
P(t,x,y;\theta)\right\|=\left\|\Phi_1(x)\frac{d}{d\theta}\tilde{E}_t^{p,\theta}
(x,
y)\Psi_1(y)\right\|
\leq C_1t^{-\frac{d}{2}}e^{-C_2\frac{d^2(x,y)}{t}}.
\end{align}
and 
\begin{align*}
&\left\|\left(\frac{d}{dt}+\tilde{\Delta}_p^\theta\right)\frac{d}{d\theta}
P(t,x,y;\theta)\right\|=\left\|\left(\frac{d}{dt}+\tilde{\Delta}
_p^\theta\right)\Phi_1(x)\frac{d}{d\theta}\tilde{E}_t^{p,\theta}(x,
y)\Psi_1(y)\right\|
\\ &\leq C_1t^{-\frac{d}{2}-1}e^{-C_2\frac{d^2(x,y)}{t}}.
\end{align*}
Moreover, one has $\frac{d}{d\theta}\tilde{\Delta}_p^\theta=0$
on $\supp(\Phi_2)$. Therefore, one has 
\begin{align}\label{DerP}
\left\|\left(\frac{d}{d\theta}\tilde{\Delta}
_p^\theta\right)
P(t,x,y;\theta)\right\|=\left\|\left(\frac{d}{d\theta}\tilde{\Delta}
_p^\theta\right)
\Phi_1(x)\tilde{E}_t^{p,\theta}
(x,
y)\Psi_1(y)\right\|\leq
C_1t^{-\frac{d}{2}-1}e^{-C_2\frac{d^2(x,y)}{t}}.
\end{align}

This implies that $Q_1(t,x,y;\theta)$, defined as in \eqref{DefQ1}, is
differentiable in $\theta$. By
\eqref{PropQeins} one has $\frac{d}{d\theta}Q_1(t,x,y;\theta)=0$ if
$d(x,y)<\epsilon$. Thus the previous two estimates imply that for each
$l\in\mathbb{N}$ and 
$T>0$ there exist constants $C_1,C_2>0$ such that for all $t\in (0,T]$, all
$\theta\in\R$ and 
all $x,y\in X$ one has 
\begin{align*}
\left\|\frac{d}{d\theta}Q_1(t,x,y;\theta)\right\|=&\left\|\left(\frac{d}{d\theta
}\tilde{\Delta}
_p^\theta\right)
P(t,x,y;\theta)+\left(\frac{d}{dt}+\tilde{\Delta}_p^\theta\right)\left(\frac{d}{
d\theta}P(t,x,
y;\theta)\right)\right\|\\ &\leq C_1
t^{l}e^{-C_2\frac{d^2(x,y)}{t}}.
\end{align*}
Thus one shows by induction that each $Q_k(t,x,y;\theta)$, defined as in
\eqref{DefQk}, is
differentiable in $\theta$ 
with 
\begin{align*}
\left\|\frac{d}{d\theta}Q_k(t,x,y;\theta)\right\|=\leq C_1 t^l
\frac{t^{k}}{k!}e^{-C_2\frac{d^2(x,y)}{t}},
\end{align*}
where the constants $C_1, C_2$ are independent of $\theta$. This implies that
$Q(t,x,y;\theta)$, defined 
as in \eqref{DefQ} 
is differentiable in $\theta$ and that
\begin{align*}
\left\|\frac{d}{d\theta}Q(t,x,y;\theta)\right\|=\leq C_1 t^l
e^{-C_2\frac{d^2(x,y)}{t}}
\end{align*}
for all $\theta\in\R$, all $x,y\in X$.  
Arguing as in the previous section before \eqref{Estconv1} and applying
\eqref{DerP0} it follows that
for each $l$ there exists 
a constant $C$ such that
\begin{align*}
\left\|\frac{d}{d\theta}P*Q(t,x,y;\theta)\right\|\leq C t^l
\end{align*}
for all $\theta\in\R$, all $x,y\in X$. Thus together with \eqref{KerKDelta} and
\eqref{DerP0}, it
follows that the kernel
$\frac{d}{d\theta}K_{\tilde{\Delta}_p^\theta}(t,x,y)$ is differentiable in
$\theta$ and that  
that 
\begin{align*}
\left\|\frac{d}{d\theta}K_{\tilde{\Delta}_p^\theta}(t,x,y)\right\|\leq C_1
t^{-\frac{d}{2}}e^{-C_2\frac{d^2(x,y)}{t}}
\end{align*}
for all $\theta\in\R$, $x,y\in X$, $t\in(0,T]$.  
By construction, 
$\frac{d}{d\theta}K_{\tilde{\Delta}_p^\theta}(t,x,y)$ 	
vanishes if $\Phi_2(x)= 0$, in particular, it has uniform compact $x$-support. 
Thus if $r$ is the function introduced in section \ref{secrelreg}, then
together with \eqref{integral} and the estimate analogous to
\eqref{estrdist} it follows that 
$r(x)^jr(y)^j\frac{d}{d\theta}K_{\tilde{\Delta}_p^\theta}(t,x,y)$ is 
square-integrable over $X\times X$ for each $j\in\mathbb{N}$, i.e. the
operators 
$\frac{d}{d\theta}K_{\tilde{\Delta}_p^\theta}(t,x,y)\circ M_{r^j}$ and $
M_{r^j}\circ \frac{d}{d\theta}K_{\tilde{\Delta}_p^\theta}(t,x,y)$
are Hilbert Schmidt operators, where $M_{r^j}$ is again the operator induced by
multiplication with the function $r^j$. 
As in the proof of Proposition \ref{PropDeltatheta}, for $j\in\mathbb{N}$
sufficiently large 
the operators
$M_{r^{-j}}\circ e^{-t\tilde{\Delta}_p^\theta}$ and
$e^{-t\tilde{\Delta}_p^\theta}\circ
M_{r^{-j}}$
are Hilbert Schmidt operators and thus putting 
$\tau:=t/2$ and writing 
\begin{align*}
\frac{d}{d\theta}e^{-t\tilde{\Delta}_p^\theta}=\left(\frac{d}{d\theta}e^{
-\tau\tilde{\Delta}_p^\theta}\circ M_{r^j}\right)\circ
\left(M_{r^{-j}}e^{-\tau\tilde{\Delta}_p^\theta}\right)+
(e^{-\tau\tilde{\Delta}_p^\theta}\circ
M_{r^{-j}})\circ\left(M_{r^j}\circ\frac{d}{d\theta}e^{
-\tau\tilde{\Delta}_p^\theta}\right),
\end{align*}
it follows that $ \frac{d}{d\theta}e^{-t\tilde{\Delta}_p^\theta}$ is trace
class. Applying Proposition \ref{Propreltr} and \cite[Theorem
4.1]{Warner}, the first equation now follows easily.
 
To prove the second equation, we remark that as in \cite[(6.11)]{MS}, for
$\theta,\theta_1\in\R$ we have 
\begin{align*}
e^{-t\tilde{\Delta}_p^{\theta_1}}-e^{-t\tilde{\Delta}_p^\theta}
=-\int_0^t\frac{d}{ds}\left(e^{-(t-s)\tilde{\Delta}_p^{\theta_1}}\circ
e^{-s\tilde{\Delta}_p^\theta}\right)ds=-\int_0^te^{-(t-s)\tilde{\Delta}_p^{
\theta_1}}\left(\tilde{\Delta}_p^{\theta_1}-\tilde{\Delta}_p^\theta\right)e^{
-s\tilde{\Delta}_p^\theta}
\end{align*}
in the strong operator topology and therefore by \eqref{Diffquot}, one has
\begin{align}\label{Eqderop}
\frac{d}{d\theta}e^{-t\tilde{\Delta}_p^\theta}=-\int_0^t
e^{-(t-s)\tilde{\Delta}_p^\theta}\circ\left(\frac{d}{d\theta}
\tilde{\Delta}_p^\theta\right)\circ e^{-s\tilde{\Delta}_p^\theta}ds
\end{align}
in the strong operator topology. 
Since the operator $\left(\frac{d}{d\theta}\tilde{\Delta}_p^\theta\right)$ 
is compactly supported, the previous construction of
$K_{\tilde{\Delta}_p^\theta}$ shows
that proceeding analogously to the proof of Proposition \ref{TrDcomps} one can
establish that 
the operators $\left(\frac{d}{d\theta}\tilde{\Delta}_p^\theta\right)\circ
e^{-t\tilde{\Delta}_p^\theta}$ 
and
$e^{-t\tilde{\Delta}_p^\theta}\circ\left(\frac{d}{d\theta}\tilde{\Delta}
_p^\theta\right)$ are trace class and that their trace norm is uniformly
bounded for
$t$ in compat subsets of $(0,\infty)$. Therefore, splitting the 
integral in \eqref{Eqderop} in the integral from $(0,t/2)$ and $(t/2,t)$, we
obtain
\begin{align*}
\Tr\left(\frac{d}{d\theta}e^{-t\tilde{\Delta}_p^\theta}\right)=&-\int_0^t
\Tr\left(e^{-(t-s)\tilde{\Delta}_p^\theta}\circ\left(\frac{d}{d\theta}
\tilde{\Delta}_p^\theta\right)\circ
e^{-s\tilde{\Delta}_p^\theta}\right)ds\\ 
=&-\int_0^t
\Tr\left(\left(\frac{d}{d\theta}
\tilde{\Delta}_p^\theta\right)\circ e^{-s\tilde{\Delta}_p^\theta}\circ
e^{-(t-s)\tilde{\Delta}_p^\theta}
\right)ds=-t\Tr\left(\left(\frac{d}{d\theta}
\tilde{\Delta}_p^\theta\right)e^{-t\tilde{\Delta}_p^\theta}\right).
\end{align*}
which is the first equality of the second equation in the proposition. The
second 
equality of this equation follows again from \cite[Theorem 4.1]{Warner}.
\end{proof}

\begin{bmrk}
Let us remark that without the trace the second equation of the proposition
would not 
hold necessarily, since 
the operators under the integral in \eqref{Eqderop} do not commute in general. 
\end{bmrk}

Let $\beta_\theta$ be as in \eqref{Seqpara}. 
Using \eqref{GaussBKerdelta} and proceeding exactly as in the proof 
of Proposition \ref{TrDcomps}, one can show
that $\beta_\theta(e^{-t\Delta_p^\theta})$ is trace class.
We can now prove the analog of \cite[Theorem 5.3]{Le}. We point out that in our
understanding Lesch's assumption of discrete dimension
spectrum 
is essential in his proof of that theorem, in particular in the 
computations yielding \cite[(5.17), (5.18)]{Le}. Therefore, we do not 
generalize Lesch's proof, but we reduce our statement to 
his statement by approximating the variation by variations on closed manifolds.
Let us furthermore recall that 
in our situation the Euler-characteristics of $X^+$ resp. of $\partial X^+$
vanish. 

\begin{prop}\label{PropGl1}
For $0<\theta<\frac{\pi}{2}$ the function $\theta\mapsto \Tr_{\Rel,u}
(e^{-t\Delta_p^\theta})$ is differentiable in $\theta$ and one has 
\begin{align*}
\frac{d}{d\theta}\left(\sum_p(-1)^pp\Tr_{\Rel,u}
\left(e^{-t\Delta_p^\theta}\right)\right) =
-t\frac{d}{dt}\frac{4}{\sin{2\theta}}\left(\sum_p(-1)^p\Tr
\left(\beta_\theta e^{-t\Delta_p^\theta}\right)
\right).
\end{align*}
Furthermore, one has 
\begin{align}\label{zweitegl}
\sum_p(-1)^p\Tr
\left(\beta_\theta e^{-t\Delta_p^\theta}\right)
=O(t^\infty),
\end{align}
as $t\to 0+$. 
\end{prop}
\begin{proof}
For $Y>Y_1+1$ we let $M(Y)$ be the same manifold as in \eqref{DefM}. As in the
previous section in the paragraph 
before  Proposition \ref{Propaxim2}, we can perform 
the corresponding variation of the De Rham complex on 
$M(Y)$ an consider the operators $\Delta_p^\theta[M(Y)]$ resp. 
$\tilde{\Delta}_p^\theta[M(Y)]$ as well as their heat kernels. 
By construction, the operators
$\frac{d}{d\theta}\tilde{\Delta}_p^\theta$ and
$\frac{d}{d\theta}\tilde{\Delta}_p^\theta[M(Y)]$ 
agree on their support which lies in $X(Y_1+1)$. Thus by Proposition
\ref{Propaxim2},
one has 
\begin{align*}
\int_{X}\Tr\left(\left(\frac{d}{d\theta}\tilde{\Delta}_p^\theta\right)
K_{\tilde{\Delta}_p^\theta}(t,x,x)\right)dx=\lim_{Y\to\infty}
\int_X\Tr\left(\left(\frac{ d}{d\theta}\tilde{\Delta}_p^\theta[M(Y)]\right)
K_{\tilde{\Delta}_p^\theta[M(Y)]}(t,x,x)\right)dx,
\end{align*}
where the operators are applied to the fist variable. We remark that
\eqref{Smfam} immediately 
shows that we can apply Proposition \ref{Propaxim2}, although $T$, which 
occurs in $(d/d\theta)\tilde{\Delta}_p^\theta$, is not 
a differential operator. 
By \eqref{It5} in Proposition \ref{PropDeltatheta}, by Proposition
\ref{trclassder} and by the analogous statement for the closed
manifold $M[Y]$, the last equality is equivalent to the equality
\begin{align*}
\frac{d}{d\theta}\Tr_{\Rel,u}(e^{
-t\Delta_p^\theta})=\lim_{Y\to\infty}\frac{d}{d\theta}\Tr\left(e^{
-t\Delta_p^\theta[M(Y)]}\right).
\end{align*}
Now we apply \cite[Theorem 5.3, (5.9)]{Le} to the closed manifold
$M(Y)$. This gives: 
\begin{align*}
\frac{d}{d\theta}\left(\sum_p(-1)^pp\Tr_{\Rel,u}
\left(e^{-t\Delta_p^\theta}\right)\right)
=&\lim_{Y\to\infty}\frac{d}{d\theta}\left(\sum_p(-1)^pp\Tr\left(e^{-t\Delta
_p^\theta[
M(Y)]}\right)\right)\\
=&\lim_{Y\to\infty}-t\frac{d}{dt}\frac{4}{\sin{2\theta}}\left(\sum_p(-1)^p\Tr
\left(\beta_\theta e^{-t\Delta_p^\theta[M(Y)]}\right)
\right)\\
=&\lim_{Y\to\infty}-t\frac{d}{dt}\frac{4}{\sin{2\theta}}\sum_p(-1)^p\int_{X^+}
\Tr K_{\Delta_p^\theta[M(Y)]}(t,x,x)dx
\end{align*}
By Proposition \ref{Propaxim2}, which by \eqref{DefPhitheta} clearly also holds
if we replace
$\tilde{\Delta}_p^\theta$ by $\Delta_p^\theta$, one has
\begin{align*}
&\lim_{Y\to\infty}-t\frac{d}{dt}\frac{4}{\sin{2\theta}}\sum_p(-1)^p\int_{X^+}
\Tr
K_{\Delta_p^\theta[M(Y)]}(t,x,x)dx\\
=&-t\frac{d}{dt}\frac{4}{\sin{2\theta }}\sum_p(-1)^p\int_{X^+}
\Tr K_{\Delta_p^\theta}(t,x,x)dx\\
=&-t\frac{d}{dt}\frac{4}{\sin{2\theta}}\left(\sum_p(-1)^p\Tr
\left(\beta_\theta e^{-t\Delta_p^\theta}\right)
\right).
\end{align*}

This proves the first equation of the proposition. To prove the 
second one, one can proceed exactly as Lesch in the proof of \cite[equation
5.10]{Le} on page 27-28: By \eqref{DefP}, \eqref{Estconv1} and
\eqref{KerKDelta} 
one has 
\begin{align*}
\int_{X^+}\Tr(K_{\Delta_p}(t,x,x))dx=\int_{X^+}\Tr
E_{t}^{p,\theta}(x,x)dx+O(t^\infty),
\end{align*}
as $t\to 0+$. Using the computations of Lesch just cited, the second equation
follows, since $X^+$ resp. 
$\partial X^+$ have vanishing Euler characteristic. 
\end{proof}

Now we consider the contribution of the cohomology. We recall that our  bundle
$F$ was the bundle $E_\rho$, $\rho\in\Rep(G)$, $\rho\neq\rho_\theta$ 
from section \ref{secDR}. By Proposition \ref{Kohofd} 
resp. Proposition \ref{PropDeltatheta}, \ref{It3}, the
sequence \eqref{Seqpara} is a short 
exact sequence of Fredholm complexes. Thus by Proposition \ref{PropFr} the
cohomology spaces
of each complex in \eqref{Seqpara} are finite-dimensional and are canonically
isomorphic to the kernels 
of the corresponding Laplacians. We equip the cohomology spaces with the inner
products 
which are induced by this isomorphism and the $L^2$-inner products on the
harmonic forms corresponding 
to the metrics $g_1$ and $h_1$. We let 
$\tau\left(\mathcal{H}_\theta\left((X^-,Z),X,
X^+;g_1,h_1,E_\rho\right)\right)$ be the torsion of the long exact cohomology
sequence of
\eqref{Seqpara}, where the torsion of a long exact sequence of
finite-dimensional 
Hilbert spaces is defined 
as in \cite[section 2.2]{Le}. We have put $g_1$ and $h_1$ in the notation in
order to emphasize 
that we still work with these metrics which are compactly supported
perturbations 
of our original metrics.

\begin{prop}\label{Propdiff1}
For $0<\theta\leq\frac{\pi}{2}$, the cohomology groups $H^*_\theta(X;E_\rho)$ of
the
complex 
$\mathcal{D}^*(X;E_\rho)$ vanish. 
Moreover, the map 
\[
\theta\mapsto \log\tau(\mathcal{H}_\theta\left((X^-,Z),X,
X^+;g_1,h_1,E_\rho\right))
\]
is differentiable and one has
\begin{align}\label{PropL}
\frac{d}{d\theta}\log\tau(\mathcal{H}_\theta\left((X^-,Z),X,
X^+;E_\rho\right))=-\frac{2}{\sin(2\theta)}\left[\sum_{j\geq
0}(-1)^j\Tr(\beta_\theta | H^j_\theta(X;E_\rho)\right]=0.
\end{align}
\end{prop}
\begin{proof}
Since all complexes involved are Freholm complexes, one can just insert Lesch's
arguments \cite[section 5.2.2]{Le} without any
change. Thus for $0<\theta,\theta'<\frac{\pi}{2}$ there is a chain isomorphism
$\Phi_{\theta,\theta'}$ 
between the short exact sequences in \eqref{Seqpara} corresponding to $\theta$
and $\theta'$, \cite[equation (5.26)]{Le} .
Hence the first statement follows from the isomorphism \eqref{IsokomplD} and
Proposition \ref{Kohofd}. Also the second statement can be proved using exactly
the arguments of \cite[section
5.2.2]{Le}. We recall again that $\chi(X^+)=0$ resp. $\chi(\partial X^+)=0$. 
\end{proof}

\begin{kor}\label{Kordiff2}
For $0<\theta<\frac{\pi}{2}$ the map $
\theta\mapsto \log{T_{\Rel;u}^\theta(X,g_1,h_1;E_\rho)}$
is differentiable and one has 
\begin{align*}
\frac{d}{d\theta}\log
T_{\Rel;u}^\theta(X,g_1,h_1;E_\rho)=\frac{2}{\sin{2\theta}}\left[-\sum_{j\geq
0}(-1)^j\Tr(\beta_\theta | H_j^\theta(X;E_\rho)\right]=0.
\end{align*}
\end{kor}
\begin{proof}
Proposition \ref{PropDeltatheta} and the 
first statement of the previous Proposition \ref{Propdiff1} imply that all 
the assumptions are verified to carry over equation (2.20) in \cite[Proposition
2.4]{Le} to the 
present case and the corollary follows from Proposition \ref{PropGl1}. 
\end{proof}

\begin{prop}\label{Propdiff3}
The map 
\begin{align*}
\theta\mapsto \log\tau(\mathcal{H}_\theta\left((X^-,Z),X,
X^+;E_\rho\right))-\log{T_{\Rel;u}^\theta(X,g_1,h_1;E_\rho)}
\end{align*}
is differentiable for
$0\leq\theta\leq\frac{\pi}{2}$. 
\end{prop}
\begin{proof}
We recall that by Proposition \ref{PropDeltatheta} all operators
$\Delta_p^\theta$ have 
pure point spectrum in $[0,1/4)$ and that by Propopsition \ref{PropFr} the
kernels $\Ker(\Delta_p^\theta)$ 
are isomorphic to the cohomology groups of the complex
$\mathcal{D}^*_{\theta}(X;F)$. An inspection of Lesch's proof in \cite[section
5.2.3]{Le} of the 
differentiability at $0$ 
of the expression \cite[(4.13)]{Le} shows that these two 
properties suffice to carry over his proof to the present situation. 
\end{proof}

We can now derive the first version of the gluing formula. We recall our
notation \eqref{notxpm}. 
We let $\det(\delta_p^{Y_1}(g_1,h_1;\rho))$ denote 
the determinant of the matrix which represents 
the homomorphism $\delta_p:H^{p}(X(Y_1);E_\rho)\to
H^{p+1}_{\min}(F_X(Y_1),\partial F_X(Y_1);E_\rho)$ 
from \eqref{connect} 
with respect to $L^2$-orthonormal bases of harmonic forms. Here the Laplacians
respectively the $L^2$-inner products 
are taken with respect to the metrics $g_1$ and $h_1$. 

\begin{thrm}\label{KlebefVI}
Let $X$ be a hyperbolic manifold with cusps as in section \ref{subsmfld}. Let 
$\rho\in\Rep(G)$, $\rho\neq\rho_\theta$ and let $E_\rho$ be the corresponding 
flat vector bundle over $X$. Let $Y_0$ be as in section \ref{subsmfld} and fix
$Y_1>Y_0$. 
Let $g_1$ and $h_1$ be the same metrics on $X$ and $E_\rho$ as above. Then for
$Y_1>Y_0$ and 
$u>Y_1+2$ one has   
\begin{align*}
\log{T_{\Rel,u}(X,g_1,h_1;E_\rho)}=&\log T(X(Y_1),g_1,h_1;E_\rho)+\log
T_{\Rel,u}(F_X(Y_1),\partial
F_X(Y_1),g_1,h_1;E_\rho)\\ &+ 
\sum_p(-1)^p\log|\det( \delta_p^{Y_1},g_1,h_1;\rho)|.
\end{align*}
\end{thrm}
\begin{proof}
We proceed identically to \cite[section 6.1.1]{Le}: One 
has $\mathcal{D}_{\theta=\frac{\pi}{4}}(X;E_\rho)\cong\mathcal{D}(X;E_\rho)$
and 
by Proposition \ref{Kohofd}, in the present case one has 
\begin{align*}
\log{\tau\left(\mathcal{H}
_\theta\left((F_X(Y_1),\partial F_X(Y_1)),X,
X(Y_1);E_\rho\right)\right)}|_{\theta=\frac{\pi}{4}}=\sum_p(-1)^p\log|\det(
\delta_p^{Y_1},g_1,h_1;\rho)|.
\end{align*}
Next one has
$\mathcal{D}_{\theta=0}(X;E_\rho)=\mathcal{D}(F_X(Y_1),\partial
F_X(Y_1);E_\rho)\oplus\mathcal{D}
(X(Y_1);E_\rho)$. As in \cite{Le}, this implies that 
\begin{align*}
{\tau\left(\mathcal{H}
_\theta\left((F_X(Y_1),\partial F_X(Y_1)),X,
X(Y_1);E_\rho\right)\right)}|_{\theta=0}
=0. 
\end{align*}
Thus by Proposition \ref{Propdiff3} one can write
\begin{align*}
&(\log{T_{\Rel,u}(X,g_1,h_1;E_\rho)}-\sum_p(-1)^p\log|\det(
\delta_p^{Y_1},g_1,h_1;\rho)|)\\ -&(\log T(X(Y_1),g_1,h_1;E_\rho)+\log
T_{\Rel,u}(F_X(Y_1),\partial
F_X(Y_1),g_1,h_1;E_\rho))\\ 
=&\int_0^{\frac{\pi}{4}}
\frac{d}{d\theta}\left(\log{T_{\Rel;u}^\theta(X,g_1,h_1;E_\rho)}
-\log\tau\left(\mathcal{H}
_\theta\left((F_X(Y_1),\partial F_X(Y_1)),X,
X(Y_1);E_\rho\right)\right)\right)d\theta,
\end{align*}
and this expression is $0$ by Proposition \ref{Propdiff1} and Corollary
\ref{Kordiff2}. 
\end{proof}

\section{The second version of the gluing formula}\label{SecGlII}
\setcounter{equation}{0}
In this section we will replace the auxilliary metrics 
$g_1$ and $h_1$ in the gluing formula by the originial metrics 
$g$ and $h$. In this way the anomaly term of Br\"uning and Ma will appear. 

We firstly prove the invariance of the relative analytic torsion 
under compacct pertrubations of the metric $g$ on $X$ and $h$ on 
$E_\rho$ if $\rho$ satisfies $\rho\neq\rho_\theta$. We recall 
that by the condition $\rho\neq\rho_\theta$ the $L^2$-cohomology 
of $X$ with coefficients in $E_\rho$ vanishes and that $X$ is odd-dimensional.
Thus, if $X$ 
was a closed hyperbolic manifold, this invariance would be already known
\cite[Corollary 2.7]{Mutorsunim}.  
\begin{prop}\label{Propkohometr}
Let  $g_1$ and $h_1$ be compact perturbations of the metrics 
$g$ and $h$. Then one has 
\begin{align*}
T_{\Rel,u}(X,g,h;E_\rho)=T_{\Rel,u}(X,g_1,h_1;E_\rho)
\end{align*}
\end{prop}
\begin{proof}
Let $g_{v}$, $v\in[0,1]$, be a smooth path of metrics on $X$ such that $g_0=g$,
the hyperbolic metric, 
and assume that there exists a
fixed compact 
subset of $X$ such that all metrics $g_v$ coincide outside this set.
For each $p$ and $v\in [0,1]$ let $\Delta_p(v)$ be the corresponding 
flat Hodge Laplacian on $E_\rho$-valued $p$-forms over $X$.
Let $*_v$ be the corresponding Hodge operators, regarded as operators 
from $\Lambda^p E_\rho$ to $\Lambda^{d-p}E_\rho$, see \cite{Mutorsunim},
let $\alpha(v):=(*_v^{-1})\circ (d/dv)*_v$ and  
\begin{align*}
(d/dv)\Delta_p(v)=-\alpha(v)
\delta_p(v)d_p+\delta_p(v)\alpha(v)d_p-d_{p-1}\alpha(v)\delta_{p-1}(v)+d_{p-1}
\delta_{p-1}(v)\alpha(v).
\end{align*}
The operator $(d/dv)\Delta_p(v)$ is a compactly supported differential
operator. 
Therefore, $((d/dv)\Delta_p(v))e^{-t\Delta_p(v)}$ is trace class by Proposition
\ref{TrDcomps}.
Next we claim that the function
$\Tr_{\Rel,u}(e^{-t\Delta_p(v)})$ is differentiable in $v$ and that one has
\begin{align}\label{Glvarmetr}
(d/dv)\Tr_{\Rel,u}(e^{-t\Delta_p(v)})=-t\Tr\bigl(((d/dv)\Delta_q(v))e^{
-\Delta_p(v)}\bigr). 
\end{align}
To prove this, one can proceed in the same way as in the proof of Proposition
\ref{trclassder} and therefore we 
shall only indicate the main idea. Let $Y>Y_0+1$ such that all metrics $g_v$
coincide
on $F_X(Y-1)$. Then consider the closed manifold $M(Y)$ from \eqref{DefM} which 
contains the set $X(Y)$. For each $v$ one can extend the metrics $g_v$ to
metrics on $M(Y)$
and one can  extend $E_\rho$ with its fibre metric to a flat bundle over $M(Y)$.
Let 
$\Delta_p^v[M(Y)]$ be the corresponding flat Hodge-Laplacians and let
$K_{\Delta_p^v[M(Y)]}$ be their integral
kernels. 
Then the kernel $K_{\Delta_p(v)}$ can be obtained by the parametrix 
method gluing together the restriction of the kernel $K_{\Delta_{p}^v[M(Y)]}$ to
$X(Y)$ and the heat kernel 
$K_{\Delta_p}$ of $X$ with respect to 
the fixed metric $g$ in exactly the same way as
in section \ref{secheatpertL}. 
The kernels $K_{\Delta_{p}^v[M(Y)]}$ depend smoothly on $v$, see
\cite[Proposition
6.1]{RS}  
and \cite[chapter 2.7]{BGV}, which applies to the present case, 
and thus, using exactly the same arguments as in the proof of Proposition
\ref{trclassder}
one can conclude \eqref{Glvarmetr}. 

Next, $(d/dv)\Tr_{\Rel,u}(e^{-t\Delta_p(v)})$ 
has a short time asymptotic expansion which arises by integrating 
the pointwise asymptotic expansion of
$\Tr\bigl(((d/dv)\Delta_p(v))K_{\Delta_p(v)}(t,x,x)\bigr)$
over $X$ resp. the 
support of $(d/dv)\Delta_p(v)$, which is compact. Thus, since $X$ is
odd-dimensional, in the short time 
asymptotic expansion of
$\Tr\bigl(((d/dv)\Delta_p(v))K_{\Delta_p(v)}(t,x,x)\bigr)$ no
constant 
term appears \cite{Gilkey}. 
As above, since $X$ with the metric $g_v$ is complete the operators
$\Delta_p(v)$ are essentially selfadjoint \cite{Ch} 
and thus the De Rham complex of $E_\rho$
valued differential forms has a unique ideal boundary condition whose
cohomology 
equals the corresponding $L^2$-cohomology of $X$ with
coefficients in $E_\rho$, \cite[Lemma 3.3, Theorem 3.5]{BL}. 
Since all $g_v$ are quasi isometric to $g$, this cohomology is zero by
Proposition \ref{LemL2cohomology}. 
Therefore, one has $\Ker(\Delta_p(v))=0$ for all $v$ by the strong Hodge
decomposition, see Proposition \ref{PropFr}. 
Thus by Proposition \ref{Longtime}, $\Tr_{\Rel,u}(e^{-t\Delta_p(v)})$ is
exponentially decaying as $t\to\infty$ for each $v$. 

Now one can establish the invariance of the relative analytic torsion under
compact pertrubations of $g$ by inserting the arguments of the proof of
\cite[Theorem 2.1]{RS}.
For a compactly supported variation $h_v$ of the fibre metric $h$ on $E_\rho$
one proceeds 
analogously, taking the derivative of the 
corresponding operator $\#_v$, which varies with the fibre-metric, see
\cite[section 2]{Mutorsunim}.
\end{proof}

We will also need a gluing formula 
for the regularized analytic torsion on the manifolds $F_X(Y)$, $Y>0$.
Thus let $Y_2>Y$ and 
consider the manifold $F_X(Y_2)\subset F_X(Y)$. For
\[
Z_X[Y,Y_2]:=F_X(Y) \backslash F_X(Y_2)
\] 
one has
$F_X(Y)=Z_X[Y,Y_2]\sqcup_{\partial F_X(Y_2)} F_X(Y_2)$.
The boundary of $Z_X[Y,Y_2]$ is the disjoint union 
of $\partial F_X(Y)$ and $\partial F_X(Y_2)$.
Let 
$g_1$ and $h_1$  metrics on $Z_X[Y,Y_2]$ and $E_\rho|_{Z_X[Y,Y_2]}$. 
Then we denote by 
$T(Z_X[Y,Y_2],\partial F_X(Y),g_1,h_1;E_\rho)$
the analytic torsion of $Z_X[Y,Y_2]$ with coefficients 
in $E_\rho$ with relative boundary conditions 
at $\partial F_X(Y)$ and with absolute boundary conditions 
at $\partial F_X(Y_2)$ with respect to the metrics $g_1$ and $h_1$. 
If $g_1$ and $h_1$ are metrics on $F_X(Y)$ resp. $E_\rho|_{F_X(Y)}$ which 
are compact perturbations of $g$ and $h$ and 
which are of product structure in a neighbourhood of $\partial F_X(Y_2)$, then
as in \eqref{SeqI} one has a short exact sequence of Hilbert complexes 
\[
0\to\mathcal{D}(F_X(Y_2),\partial F_X(Y_2);E_\rho)\to\mathcal{D}(F_X(Y),\partial
F_X(Y);E_\rho)\to \mathcal{D}(Z[Y,Y_2],\partial F_X(Y);E_\rho)\to 0.
\]
Since the cohomology groups
$H^*(Z_X[Y,Y_2],\partial
F(Y);E_\rho)$ are trivial, one obtains isomorphisms
\begin{align}\label{extbyzero}
e_p(Y,Y_2;E_\rho):H^p_{\min}(F_X(Y_2),\partial
F_X(Y_2);E_\rho)\to
H^p_{\min}(F_X(Y),\partial F_X(Y);E_\rho),
\end{align}
which are defined by extension by zero. We recall 
that the cohomology groups in \eqref{extbyzero} are finite-dimensional and are 
isomorphic to the kernels of the corresponding Laplacians. Moreover, by 
the definition of the minimal extension, these cohomology 
groups depend only on the quasi-isometry class of $g_1$ and $h_1$ and 
are in particular invariant under compact perturbations of the metrics. 
Let
$E_p(Y,Y_2,g_1,h_1;E_\rho)$ be the matrix representing
$e_p(Y,Y_2;E_\rho)$ 
with respect to bases in the cohomology coming from $L^2$-orthonormal bases of
harmonic 
forms associated to the metrics $g_1$ and $h_1$. Then the 
following gluing formula holds.

\begin{prop}\label{GlCusp}
Let $Y>0$, $Y_2>Y$. Let $g_1$ and $h_1$ be metrics on $F_X(Y)$ resp.
$E_\rho|_{F_X(Y)}$ which are 
compact perturbations of the metrics $g$ and $h$ and which 
are of product structure in a neighbourhood of $\partial F_X(Y_2)$ in
$F_X(Y)$.
Then
one has 
\begin{align*}
&\log T_{\Rel,u}(F_X(Y),\partial F_X(Y),g_1,h_1;E_\rho)= 
\log T_{\Rel,u}(F_X(Y_2),\partial
F_X(Y_2),g_1,h_1;E_\rho)\\ +&\log T(Z_X[Y,Y_2],\partial
F(Y),g_1,h_1;E_\rho)+\sum_p(-1)^p\log|\det E_p(Y,Y_2,g_1,h_1;E_\rho)|.
\end{align*}
\end{prop}
\begin{proof}
For the proof one can proceed exactly in the 
the same way as in the proof of Theorem \ref{KlebefVI}. 
\end{proof}

Now we want to pass from the metrics $g_1$ on $F_X(Y_1)$ and $h_1$ on
$E_\rho|_{F_X(Y_1)}$ in Theorem \ref{KlebefVI}  to the original metrics $g$ and
$h$. 
For this purpose, we will apply the anomaly formula 
of Br\"uning an Ma \cite{BM1}, \cite{BM2}. In the unimodular case, this 
formula is local on the boundary, and therefore, using 
the gluing formula from Proposition \ref{GlCusp}, we will 
be able to apply it also in the present non-compact stetting.
We firstly recall the formula of Br\"uning and Ma for 
the very special case in which we will need it. For a metric $g_1$ on 
$Z_X[Y,Y_2]$ we let 
$B(\nabla_{g_1}^{TZ_X[Y,Y_2]})$
be the $d-1$-form on $\partial Z_X[Y,Y_2]$ defined by Br\"uning and Ma in
\cite[(1.19)]{BM1}. 
The form of \cite[(1.19)]{BM1} depends only on the Riemannian structure of an
arbitrary neighbourhood 
of the boundary. Therefore, for a metric $g_1$ on $F_X(Y)$, one can also define
a $d-1$-form $B(\nabla_{g_1}^{T F_X(Y)})$ on 
$\partial F_X(Y)$ as in \cite[(1.19)]{BM1}, although $F_X(Y)$ is not
compact. 
Our flat bundle is unimodular and 
the boundaries we work with are always flat tori. Therefore, as 
already remarked in section \ref{SecReid}, only the term involving the form
\cite[(1.19)]{BM1}
of 
the Br\"uning Ma anomaly formulas \cite[Theorem 0.1]{BM1}, \cite[Theorem
0.1]{BM2} is present in our situations and the next proposition can thus be seen
as its 
analog for the relative analytic torsion of the cusp.

\begin{prop}\label{PropBrM}
Let $Y>0$. Let $g_1$ and $h_1$ be metrics on $F_X(Y)$ and $E_\rho|_{F_X(Y)}$
which 
are compact perturbations of the metrics $g$ and $h$ and which are of product
structure in a neighbourhood of $\partial F_X(Y)$.
Let 
$B_p(Y;g_1,h_1)$ be the matrix which represents the base change 
on $H^p_{\min}(F_X(Y),\partial F_X(Y))$ from 
an orthonormal basis of harmonic forms with 
respect to the metrics $g$ and $h$ to an orthonormal 
basis of harmonic forms with respect to the metrics $g_1$ and $h_1$.
Then one has 
\begin{align*}
&\log T_{\Rel,u}(F_X(Y),\partial
F_X(Y),g_1,h_1;E_\rho)=\log T_{\Rel,u}(F_X(Y),\partial
F_X(Y),g,h;E_\rho)\\ -&\frac{1}{2}\rk(E_\rho)\int_{\partial
F_X(Y)}B(\nabla_g^{TF_X(Y)})
+\sum_{p}(-1)^{p}\log|\det
B_p(Y;g_1,h_1;E_\rho)|.
\end{align*}
\end{prop}

\begin{proof}
We let $Y_2>Y$ and consider the manifold 
$F_X(Y_2)\subset F_X(Y)$. Assume 
that the metrics $g_1$ and $h_1$ are of product structure in a neighbourhood of
$\partial
F_X(Y_2)$ 
in $F_X(Y)$. Now we let $g_2$  and $h_2$ be metrics on $F_X(Y)$  resp.
$E_\rho|_{F_X(Y)}$ which coincide with
the metrics $g$ and $h$ in a neighbourhood of $\partial F_X(Y)$ and 
which coincide with the metrics $g_1$ and $h_1$ in a neighbourhood of
$F_X(Y_2)$ in $F_X(Y)$. 
If we apply Proposition \ref{GlCusp} first for the metrics $g_1$ and $h_1$ and
then
for the metrics $g_2$ and $h_2$ and then subtract the 
first obtained equality from the second one and use that $g_1=g_2$, $h_1=h_2$ on
$F_X(Y_2)$, 
we obtain:
\begin{align}\label{EqBM1}
&\log T_{\Rel,u}(F_X(Y),\partial
F_X(Y),g_2,h_2;E_\rho)-\log T_{\Rel,u} (F_X(Y),\partial
F_X(Y),g_1,h_1;E_\rho)\nonumber\\=&\log T(Z_X[Y,Y_2],\partial
F_X(Y),g_2,h_2;E_\rho)-\log
T(Z_X[Y,Y_2],\partial F_X(Y),g_1,h_1;E_\rho)\nonumber\\
&+\sum_p(-1)^p\log|\det E_p(Y,Y_2,g_2,h_2;E_\rho)|-\sum_p(-1)^p\log|\det
E_p(Y,Y_2,g_1,h_1;E_\rho)|.
\end{align}
The metric 
$g_1$ is of product structure in a neighbourhood of $\partial Z_X[Y,Y_2]$ and
the metric 
$g_2$ is of product structure in a neighbourhood of the component $\partial
F_X(Y_2)$
of $\partial Z_X[Y,Y_2]$ and it equals the hyperbolic metric $g$ in a
neighbourhoud
of the component $\partial F_X(Y)$ of $\partial
Z_X[Y,Y_2]$. 
Thus, since the cohomology $H^*(Z_X[Y,Y_2],\partial F_X(Y);E_\rho)$
is trivial and since the boundary $\partial Z[Y,Y_2]$ with its induced metric is
a disjoint 
union of flat tori, it follows from \cite[Theorem 3.4, (3.26)]{BM2} and the 
local definition of the form in the second line of \cite[(3.7)]{BM2} that 
\begin{align}\label{EqBM2}
&\log T(Z_X[Y,Y_2],\partial F_X(Y),g_1,h_1;E_\rho)\nonumber\\ 
=&\log T(Z_X[Y,Y_2],\partial
F_X(Y),g_2,h_2;E_\rho)-\frac{1}{2}\rk(E_\rho)\int_{\partial
F_X(Y)}B(\nabla_g^{TF_X(Y)}).
\end{align}

We let $B_p(Y,g_2,h_2;E_\rho)$ be the matrix which represents the 
change of bases on the space $H^p_{\min}(F_X(Y),\partial F_X(Y);E_\rho)$ from 
the basis of harmonic forms with respect to the metrics $g$ and $h$ to 
the basis of harmonic forms with respect to $g_2$ and $h_2$. 
Then we claim that, since $g_2$ and $h_2$ coincide with $g$ and $h$ in a
neighbourhood of
$\partial
F_X(Y)$, one has
\begin{align}\label{EqBM3}
\log T_{\Rel,u}(F_X(Y),\partial
F_X(Y),g_2,h_2;E_\rho)=&\log T_{\Rel,u}(F_X(Y),\partial
F_X(Y),g,h;E_\rho)\nonumber \\ &+\sum_{p}(-1)^{p}\log|\det
B_p(Y,g_2,h_2;E_\rho)|.
\end{align}
This equality woud again be well known for the situation of compact 
odd-dimensional manifolds with boundary, unitary resp. unimodular flat bundles
and interior variations of the metric, \cite{Cheeger 1}, \cite{BM2}. 
To extend it to the present case, one can proceed exactly as in the proof of
Proposition 
\ref{Propkohometr} and we shall outline the main steps only: Firstly, one joins
$g$ and $g_2$, $h$ and $h_2$ by a smooth path of metrics  
$g_v$, $h_v$ $v\in[0,1]$ such that all metrics $g_v$ resp. $h_v$ coincide with
$g$ and $h$ in a
neighbourhood of the boundary and outside a fixed compact subset of $F_X(Y)$. 
Then, since the variation is compactly supported in the interior of $F_X(Y)$ and
since $F_X(Y)$ is odd-dimensional, 
the short-time asymptotic expansion of
$\Tr\bigr(((d/dv)\Delta_q^v)e^{-t\Delta_q^v}\bigl)$ has 
no constant term \cite{Gilkey}. Therefore
one can
proceed exactly 
as in the closed case, \cite[Theorem 2.1, Proposition 6.1]{RS}, \cite[Theorem
7.6]{M1} \cite[Theorem 2.6]{Murel} and as in the proof of Proposition
\ref{Propkohometr} to obtain 
\begin{align}\label{vAbl}
&\frac{d}{dv}\log T_{\Rel,u}(F_X(Y),\partial
F_X(Y),g_v,h_v;E_\rho)\nonumber\\
&=\frac{1}{2}\sum_p(-1)^p\Tr\bigl((*_v^{-1}\frac{d}{dv}*_v+ {\#}_v^{-1}
\frac{d}{dv}{\#}_v) H_v^p\bigr),
\end{align}
where $H_v^p$ is the orthogonal projection onto the space of harmonic forms 
with respect to the metrics $g_v$ and $h_v$. The argument of Br\"uning an Ma
given in the last two paragraphs of the proof of  
\cite[Theorem 4.5]{BM1} carries over to the present 
situation of a Fredholm complex and it follows that the right hand side of
\eqref{vAbl} is the
$v$-derivative 
of the corresponding change of the volume element on the determinant line
of the cohomology, see also \cite[Proposition 6.4 and page 206]{RS}. This 
establishes \eqref{EqBM3}. 

Combining \eqref{EqBM1}, \eqref{EqBM2} and \eqref{EqBM3} we obtain
\begin{align*}
&\log T_{\Rel,u} (F_X(Y),\partial F_X(Y),g,h;E_\rho)=\log T_{\Rel,u}(F_X(Y),
\partial F_X(Y),
g_1,h_1;E_\rho)\\
&+\sum_{p}(-1)^p\bigl(\log|\det
E_p(Y,Y_2,g_2,h_2;E_\rho)|-\log|\det
E_p(Y,Y_2,g_1,h_1;E_\rho)|\\ &-
\log|\det{B_p(Y,g_2,h_2;E_\rho)}|\bigr)+\frac{1}{2}\rk(E_\rho)\int_{\partial
F_X(Y)}B(\nabla_g^{TF_X(Y)}).
\end{align*}
Using that $g_1=g_2$, $h_1=h_2$ on $F_X(Y_2)$, it follows directly from the
definitions that
\begin{align*}
B_p(Y,g_1,h_1;E_\rho)=E_p(Y,Y_2,g_1,h_1;E_\rho)\circ
(E_p(Y,Y_2,g_2,h_2;E_\rho))^{-1}\circ B_p(Y,g_2,h_2;E_\rho)
\end{align*}
and the Proposition is proved. 
\end{proof}

The previous proposition implies the 
gluing formula for the cusps for the original metrics $g$ and $h$ 
which now contains the anomaly. We use the 
notation from the beginning of this section.

\begin{kor}\label{CorGlCusp}
Let $Y>0$, $Y_2>Y$. Then one has 
\begin{align*}
&\log T_{\Rel,u}(F_X(Y),\partial F_X(Y),g,h;E_\rho)-\rk(E_\rho)\int_{\partial
F_X(Y)}B(\nabla_g^{TF_X(Y)})\\ =& 
\log T_{\Rel,u}(F_X(Y_2),\partial
F_X(Y_2),g,h;E_\rho)-\rk(E_\rho)\int_{\partial F_X(Y_2)}B(\nabla_g^{TF_X(Y_2)})
\\ &+\sum_p(-1)^p\log|\det E_p(Y,Y_2,g,h;E_\rho)|.
\end{align*}
\end{kor}
\begin{proof}
We firstly replace the metrics $g$ on $F_X(Y)$ and $h$ on $E_\rho|_{F_X(Y)}$ by 
metrics $g_1$ and $h_1$ which are compact perturbations of $g$ and $h$ and which
are of product structure in a neighbourhood of $\partial F_X(Y_2)$ in $F_X(Y)$
and which 
coincide with $g$ and $h$ in a neighbourhood of $\partial F_X(Y)$. 
The torsion on $F_X(Y)$ changes as in \eqref{EqBM3}. Then we apply Proposition
\ref{GlCusp} for the metrics $g_1$ and $h_1$. 
Finally, we apply Proposition \ref{PropBrM} for the manifold $F_X(Y_2)$ to
replace the 
metrics $g_1$ and $h_1$ on $F_X(Y_2)$ by $g$ and $h$. We obtain
\begin{align*}
&\log T_{\Rel,u}(F_X(Y),\partial F_X(Y),g,h;E_\rho)=\log
T_{\Rel,u}(F_X(Y_2),\partial
F_X(Y_2),g,h;E_\rho)\\ &-\frac{1}{2}\rk(E_\rho)\int_{\partial
F_X(Y_2)}B(\nabla_g^{TF_X(Y_2)})+\sum_p(-1)^p\log|\det
E_p(Y,Y_2,g,h;E_\rho)|\\ &+\log
T(Z_X[Y,Y_2],g_1,h_1;E_\rho),
\end{align*}
where $E_p(Y,Y_2,g,h;E_\rho)$ is the matrix representing 
the map \eqref{extbyzero} with respect to $L^2$-orthonormal bases of harmonic
forms associated to the metrics $g$ and $h$. 
For the relative Reidemeister torsion of the zylinder one has 
$\log\tau(Z_X[Y,Y_2],\partial
F_X(Y);E_\rho)=0$.
For the case of orthogonal representations this was proved by Milnor \cite[Lemma
7.5]{Milnor}
and the proof carries over to the unimodular case. Alternatively, one can 
use the combinatorial gluing formula \cite{Milnor}, \cite{Luck},
\cite{BM2}, \cite {Le}: $\log\tau(Z_X[Y,Y_2],\partial
F_X(Y);E_\rho)=\log\tau(Z_X[Y,Y+\epsilon],\partial
F_X(Y);E_\rho)+\log\tau(Z_X[Y+\epsilon,Y_2],\partial
F_X(Y+\epsilon);E_\rho)=2\log\tau(Z_X[Y,Y_2],\partial
F_X(Y);E_\rho)$.
Thus by  \cite[Theorem 0.1 , Remark 1.8 (ii); Theorem 3.4, (3.26)]{BM2} and 
the definition of $g_1$ one has 
\begin{align*}
\log T(Z_X[Y,Y_2],\partial
F_X(Y),g_1,h_1;E_\rho)=\frac{1}{2}\rk(E_\rho)\int_{\partial
F_X(Y)}B(\nabla_g^{TF_X(Y)}).
\end{align*}
and the corollary
follows. 
\end{proof}
We remark that in the compact case gluing formulas of the previous 
type were established by Br\"uning and Ma \cite[Theorem 0.3, Theorem 0.4]{BM2}.

Now we  can study the formula in Theorem \ref{KlebefVI} further.
The metrics $g_1$ and $h_1$ used in this theorem are of product
structure in a
neighbourhood 
of $\partial X(Y_1)$. 
Therefore, since $\chi(\partial X_1)=0$, it follows from the Cheeger-M\"uller
Theorem 
for manifolds with boundary, which in the present situation is due to 
Br\"uning 
and Ma \cite[Theorem 0.1 , Remark 1.8 (ii)]{BM2},  
that 
the analytic torsion $T(X(Y_1),g_1,h_1;E_\rho)$ equals the corresponding 
Reidemeister torsion $\tau(X(Y_1),g_1,h_1;E_\rho)$, i.e. that 
\begin{align}\label{CMBrMA}
T(X(Y_1),g_1,h_1;E_\rho)=\tau(X(Y_1),g_1,h_1;E_\rho),
\end{align}
where we use from now on the notation of section \ref{SecReid}. 
For the case of unitary representations of the fundamental 
group, the corresponding result would be due to L\"uck \cite{Luck} and Vishik
\cite{Vi}. 

From now on, we shall drop the metrics $g$ and $h$ from 
the notation since these are the metrics we finally want to work with. Thus 
we shall write for example $T(F_X(Y_1),\partial F_X(Y_1);E_\rho)$ for
$T(F_X(Y_1),\partial F_X(Y_1),g,h;E_\rho)$.
Let $D_p(Y_1;E_\rho)$ be the matrix which represents the connecting homomorphism
$\delta_p^{Y_1}$ from \eqref{connect}
with respect to the basis of $H^p(X(Y_1);E_\rho)$ 
consisting of restrictions of Eisenstein series to $X(Y_1)$ as in the end of 
section \ref{SecReid}
and to the orthonormal basis of harmonic 
forms in $H^{p+1}_{\min}(F_X(Y_1),\partial F_X(Y_1);E_\rho)$
with respect to the original metrics $g$ and 
$h$. 
Then combining \eqref{BasecRT}, Theorem \ref{KlebefVI}, Proposition
\ref{Propkohometr},
Proposition \ref{PropBrM} and \eqref{CMBrMA}, we obtain:
\begin{align}\label{EqGLDet}
\log T_{\Rel,u}(X;E_\rho)=&\log{\tau_{Eis}(\overline{X};E_\rho)}+\log
T_{\Rel,u}(F_X(Y_1),\partial F_X(Y_1);E_\rho)\nonumber\\ +&\sum_p(-1)^p\log|\det
D_p(Y_1;E_\rho)|-\frac{1}{2}\rk(E_\rho)\int_{\partial
F_X(Y_1)}B(\nabla_g^{TF_X(Y)}). 
\end{align}

We next replace $F_X(Y_1)$ by $F_X$ using the gluing 
formula for the cusp from Corollary \ref{CorGlCusp}. Here 
$F_X$ is $F_X(1)$ as in section \ref{subsmfld}. We 
can assume that $Y_1\geq 1$ in \eqref{EqGLDet}. 
Let $E_p(Y_1;E_\rho)$ be the matrix which represents the extension 
by zero from $H^p_{\min}(F_X(Y_1),\partial F_X(Y_1);E_\rho)$ to 
$H^p_{\min}(F_X,\partial F_X;E_\rho)$ from \eqref{extbyzero} with respect to the
inner products 
on the cohomologies induced 
by the harmonic forms associated to the metrics $g$ and $h$ on 
$F_X(Y_1)$, $E_\rho|_{F_X(Y_1)}$ resp. $F_X$, $E_\rho|_{F_X}$. 
Then from \eqref{EqGLDet} and Corollary \ref{CorGlCusp} we obtain 
\begin{align}\label{VorlGl}
\log T_{\Rel,u}(X;E_\rho)=&\log{\tau_{Eis}(\overline{X};E_\rho)}+\log
T_{\Rel,u}(F_X,\partial F_X;E_\rho)-\frac{1}{2}\rk(E_\rho)\int_{\partial
F_X}B(\nabla_g^{TF_X}) \nonumber\\+&\sum_p(-1)^p(\log|\det
D_p(Y_1;E_\rho)|-\log|\det E_p(Y_1;E_\rho)|).
\end{align}

Now we use again results of Br\"uning an Ma who for the present case computed
the anomaly explicitly in 
\cite[section 4.5]{BM1}. 
We realize a neighbourhood of $\partial F_X$ in $F_X$ as  
$\partial F_X\times [0,\epsilon)$; the hyperbolic metric $g$ is then given by
$g(w,x_d)=(1+x_d)^{-2}(dx_d^2+g_{\partial F_X})$, $w\in\partial F_X$,
$x_d\in[0,\epsilon)$ and   
where $g_{\partial F_X}$ denotes the restriction of $g$ to $\partial F_X$. We 
may assume that the restrictions of the metrics $g$ and $g_1$ to $\partial F_X$
coincide. According to \cite[(4.38)]{BM1}, for $f(v):=(1+v)^2$ we let
$g_s(v, x_d):=f(sx_d)g$, which is is a smooth path between $g$ and the metric
$g_1$ 
on $\partial F_X\times [0,\epsilon)$ for $\epsilon$ sufficiently small. 
Since $\partial F_X$ with the hyperbolic metric is 
a union of finitely many disjoint flat tori, 
by \cite[(4.42)]{BM1}, \cite[Theorem 0.1, Theorem 4.5, (4.7)]{BM1} 
for
\begin{align}\label{KonstBM}
c(n):=\frac{(-1)^n(2n-1)!}{2^{2n+1}\pi^n n!}
\end{align}
one has
\begin{align}\label{BM}
\frac{1}{2}\int_{\partial
F_X}B(\nabla^{TF_X}_g)=\frac{(-1)^n(2n)!}{2^{2n+1}\pi^n n!}\vol(\partial
F_X)\int_0^1 s^{2n-1}ds=c(n)\vol(\partial F_X). 
\end{align}
Here, as above we may 
replace $F_X$ by the manifold $Z_X[1,2]$ with suitable metrics in order 
to apply \cite[Theorem 4.5, (4.7)]{BM1}. 

We remark that the left hand side in \eqref{VorlGl} does not
depend on $Y_1$ anymore and we will see below that the second line in
\eqref{VorlGl}
doesn't depend on $Y_1$ either.

\section{The contribution of the cohomology to the gluing
formula}\label{secContrcoho}
\setcounter{equation}{0}
We keep the notations of the previous section. In order 
to complete the proof of our main result, it remains 
to compute the contribution of the cohomology, i.e. the determinants 
in the second line of \eqref{VorlGl}.

To begin with, we recall that if $\omega$ 
is a smooth $E_\rho$-valued $p$-form on $F_X(Y)$ which is square-integrable, for
which 
also $d\omega$ is square-integrable and which restricts to zero on $\partial
F_X(Y)$, then 
$\omega$ belongs to $\mathcal{D}^*_{\min}(F_X(Y),\partial F_X(Y);E_\rho)$.
This can be proved proceeding exactly as in the proof 
of the corresponding statement by Br\"uning in Lesch \cite[Theorem 4.1]{BL} for
compact 
manifolds with boundary. Here we remark that to deal with the non-compact end
of $F_X(Y)$ one 
just takes a sequence $\phi_n$ of smooth compactly supported functions which are
idenitcally $1$ on $F_X(Y+n)$, which are bounded and whose derivative 
is bounded. Such a sequence obviously exists in the present case. Thus, using 
the methods of section \ref{secclcusp}, one can now construct explicitly an
orthonormal basis as
follows. We equip the 
spaces $\mathcal{H}^k(\nL_{P_j};V_\rho)$ with the inner product 
defined in the last paragraph of section \ref{SecReid}. Then 
we have the following proposition. 

\begin{prop}\label{Propccusp}
Let $Y>0$. Let $\rho\in\hat{G}$, $\rho\neq\rho_\theta$.
Let $k>n$. For 
each $P_j\in\mathfrak{P}_\Gamma$ let
$\Phi_{i,j}^k$, $i=1,\dots,\dim\mathcal{H}^k(\nL_{P_j};V_\rho)$, be an
orthonormal
basis of
$\mathcal{H}^k(\nL_{P_j};V_\rho)$. Then an orthonormal basis of 
$H^{k+1}_{\min}(F_X(Y),\partial F_X(Y);E_\rho)$ is given 
by the cohomology classes of the sections
\[
\sqrt{2|\lambda_{\rho,k}|}\cdot
Y^{-\lambda_{\rho,k}}\cdot(\widetilde{\Phi_{i,j}^k})_{ \lambda_{\rho,k}},
\]
where $(\widetilde{\Phi_{i,j}^k})_{ \lambda_{\rho,k}}$ is as in
\eqref{Phitilde}. 
Let 
$\Phi_{i,j}^n$, $i=1,\dots,\dim\mathcal{H}^n(\nL_{P_j};V_\rho)_{-}$, be an
orthonormal basis of
$\mathcal{H}^n(\nL_{P_j};V_\rho)_{-}$. Then an orthonormal basis of 
$H^{n+1}_{\min}(F_X(Y),\partial F_X(Y);E_\rho)$ is given 
by the cohomology classes of the sections
\[
\sqrt{2|\lambda_{\rho,n}^-|}\cdot
Y^{-\lambda_{\rho,n}^{-}}\cdot(\widetilde{\Phi_{i,j}
^n})_{\lambda_{\rho,n}^{-}}.
\]
\end{prop}
\begin{proof}
By Lemma \ref{HarmForm}, the forms in the proposition are harmonic and restrict
to zero on
$\partial F_X(Y)$. 
By the assumption on $\rho$, one has $\lambda_{\rho,n}^-<0$ and 
$\lambda_{\rho,k}<0$ for $k>n$. 
Thus by \eqref{InProd} they form an orthonormal system in
$L^2(F_X(Y);\Lambda^{k+1}E_{\rho})$.
Using Proposition \ref{Propkoho1}, Proposition \ref{PropKoho3} Proposition 
\ref{Kohofd} and the preceding remarks
it follows that the corresponding cohomology classes form a basis of
$H^{k+1}_{\min}(F_X(Y),\partial F_X(Y);E_\rho)$. 
\end{proof}

Using the preceding proposition, the isomorphism \eqref{Isomkoho} and 
the fact that
$\dim\mathcal{H}^n(\nL_{P_j};V_\rho)=2\dim\mathcal{H}^n(\nL_{P_j};V_\rho)_{-}$,
we can immediately compute the term $\det
E_p(Y_1;E_\rho)$ which 
appears in the second line of \eqref{VorlGl}: 

\begin{kor}\label{Korcoho}
Let $Y_1>1$. Then for $k>n$ one has 
\begin{align*}
\det E_{k+1}(Y_1;E_\rho)=Y_1^{\lambda_{\rho,k}\dim
H^k(\partial\overline{X};E_\rho)}.
\end{align*}
For $k=n$ one has 
\begin{align*}
\det E_{k+1}(Y_1;E_\rho)=Y_1^{\lambda_{\rho,k}\dim
H^n(\partial\overline{X};E_\rho)/2}.
\end{align*}
\end{kor}
\begin{proof}
This follows immediately from Proposition \ref{Propccusp} and \eqref{InProd}. 
\end{proof}

We now compute the determinant of $D_p(Y_1;E_\rho)$, which 
represents the matrix $\delta_p(Y_1)$ in the same way as in 
the previous section. 
By construction, the map $\delta_p(Y_1)$ is defined as follows.  
Let $[f]\in H^p(X(Y_1);E_\rho)$ be a cohomology class, where $f$ 
is a smooth section of $E_\rho$ which is closed. Let $\tilde{f}$ be an
extension of $f$ to a smooth, square-integrable section of $E_\rho$ on $X$ such 
that also $d\tilde{f}$ is square-integrable. 
Then one has $\delta_p(Y_1)([f])=[d(\tilde{f}|_{F_X(Y_1)})]$. 
If $h$ is any smooth section of $E_\rho$ over 
$F_X(Y_1)$ which is square-integrable and for which $dh$ is square-integrable
and which coincides with $f$ on $\partial F_X(Y_1)$, then by 
the above remarks
$h-\tilde{f}\in\mathcal{D}^p(F_X(Y_1),\partial F_X(Y_1);E_\rho)$. Thus 
one also has $\delta_p(Y_1)([f])=[d h]$.

\begin{prop}
For each $k=n+1,\dots,2n$ one has 
\begin{align*}
\det(D_k(Y_1;E_\rho)=\left(\sqrt{2|\lambda_{\rho
,k
}|} \cdot
Y_1^{-\lambda_{\rho,k}}\right)^{\dim H^k(\partial\overline{X};E_\rho)}.
\end{align*}
For $k=n$ one has 
\begin{align*}
\det(D_n(Y_1:E_\rho))=\left(\sqrt{2|\lambda_{
\rho,
n}^-|}
\cdot Y_1^{-\lambda_{\rho,n}^-}\right)^{\dim
H^n(\partial\overline{X};E_\rho)/2}.
\end{align*}
\end{prop}
\begin{proof}
First assume that $k>n$. Fix $\Phi\in\mathcal{H}^k(\nL_{P_j};V_\rho)$. 
Since $\lambda_{\rho,k}<0$, the forms
$\Phi_{\lambda_{\rho,k}}$ and $\tilde{\Phi}_{\lambda_{\rho,k}}$
belong to $L^2(F_X(Y_1),\nu_p(\rho))\cong L^2(F_X(Y_1),\Lambda^pE_\rho)$ resp.
to $L^2(F_X(Y_1),\nu_{p+1}(\rho))\cong L^2(F_X(Y_1),\Lambda^{p+1}E_\rho)$. 
By Lemma \ref{Lemd} and equation  
\eqref{Deflambatau} we 
have 
\begin{align}\label{eqd}
d(\Phi_{\lambda_{\rho,k}})=2\lambda_{\rho,k}
\tilde{\Phi}_{\lambda_{\rho,k}}.
\end{align}
By Lemma \ref{Pullback}, equation \eqref{Deflambatau} and Proposition
\ref{Propkoho1}, the 
restrictions of $Y_1^{-2\lambda_{\rho,k}}\cdot\Phi_{\lambda_{\rho,k}}$ 
and $E(\Phi:-\lambda_{\rho,k})$  to $\partial X(Y_1)$
are cohomologous. 
Thus by the above remarks one has 
\begin{align*}
\delta_k(Y_1)\left(E(\Phi:-\lambda_{\rho,k})\right)=d\left(Y_1^{-2\lambda_{\rho,
k}}
\Phi_{
\lambda_{\rho,k}}\right).
\end{align*}
Applying \eqref{eqd} and Proposition \ref{Propccusp}, the Proposition follows
for $k>n$. 

It remains to consider the case $k=n$. Fix
$\Phi\in\mathcal{H}^n(\nL_{P_j};V_\rho)$. 
By Proposition \ref{PropKoho3}, for $P_l\in\mathfrak{P}_\Gamma$ the restriction
of 
$E(\Phi:-\lambda_{\rho,n}^-)$ 
to the boundary component $\partial X(Y_1)_{P_l}$ of $X(Y_1)$ associated to
$P_l$ is cohomologous to 
\begin{align*}
\delta_{j,l}\Phi+[\underline{C}_{P_j|P_l}(\sigma_{\rho,n}^{-}:-\lambda_{\rho,n}
^-)\Phi]_+.
\end{align*}
Here one has
$[\underline{C}_{P_j|P_l}(\sigma_{\rho,n}:-\lambda_{\rho,n}^-)\Phi]_+
\in
\mathcal{H}^n(\nL_{P_{l}},V_\rho)_+$. Thus, since
$\lambda_{\rho,n}^+=-\lambda_{\rho,n}^-$,  
by Lemma \ref{Pullback} and equation \eqref{Deflambatau} the form
\begin{align}\label{DefForm}
\bigl([\underline{C}_{P_j|P_l}(\sigma_{\rho,n}:-\lambda_{\rho,n}^-)\Phi]
_+\bigr)_{\lambda_{
\rho,n}^-},
\end{align}
defined as in \eqref{DefPhilambda}, 
restricts to
$[\underline{C}_{P_j|P_l}(\sigma_{\rho,n}:-\lambda_{\rho,n}^-)\Phi]_+$ on 
$\partial F_{P_l}(Y_1)$. However, since $\lambda_{\rho,n}^-<0$, the form in
\eqref{DefForm} 
is square integrable on $F_{P_l}(Y_1)$ and by Lemma \ref{Lemd} it is closed.
Therefore, exactly the same argument as before gives 
\begin{align*}
\delta_n(Y_1)\left(E(\Phi:-\lambda_{\rho,n}^-)\right)=d\left(Y_1^{-2\lambda_{
\rho,
n}^-}\Phi_{\lambda_{\rho,n}^-}\right)
\end{align*}
and applying \eqref{eqd} and Proposition \ref{Propccusp}, the Proposition
follows
also for $k=n$. 
\end{proof}

Now we remark that by the Proposition \ref{Propkoho2} and by Proposition
\ref{Kohofd} the last line in \eqref{VorlGl} can be written as
\begin{align*}
\sum_{k=n}^{2n}(-1)^k(\log|\det
D_k(Y_1;E_\rho)|+\log|\det E_{k+1}(Y_1;E_\rho)|).
\end{align*}
Thus it equals 
\begin{align}\label{LetzteGl}
&\frac{(-1)^n}{4}\log{(2|\lambda_{\rho,n}^-|)}\cdot\dim
H^n(\partial\overline{X};E_\rho)+
\sum_{k=n+1}^{2n}\frac{(-1)^k}{2}\log (2|\lambda_{\rho,k}|)\cdot \dim
H^k(\partial\overline{X};E_\rho).
\end{align}
Now we use that $\chi(\partial
\overline{X})=0$. Then we use that that $\dim H^k(\partial
\overline{X};E_\rho)=\dim H^{2n-k}(\partial \overline{X};E_\rho)$ by 
the computations of section \ref{secclcusp} and by \eqref{Isomkoho}, although 
$E_\rho$ may not be self-dual, and that 
$\lambda_{\rho,k}=-\lambda_{\rho,2n-k}$ for $k\neq n$ 
and that $|\lambda_{\rho,n}^-|=|\lambda_{\rho,n}|$. Then \eqref{LetzteGl}
becomes
\begin{align*}
\frac{1}{4}\sum_{k=0}^{2n}(-1)^k\log (|\lambda_{\rho,k}|)\cdot \dim
H^k(\partial\overline{X};E_\rho)
\end{align*}
and Theorem \ref{Theorem} is established. Here we remark 
that by the remarks from the end of section \ref{secrelreg}, we may 
pass from the relative to the regularized traces in \eqref{VorlGl}.

\end{document}